%% file: arXiv-v2.tex
\documentclass[11pt]{amsart}

\input{macros}

\title[Integral Hasse principle for stacky curves]{The integral Hasse principle for stacky curves associated to a family of generalized Fermat equations}
\author[J. Duque-Rosero]{Juanita Duque-Rosero}
\address{Juanita Duque-Rosero, Department of Mathematics and Statistics, Boston University, Boston, MA, USA}
\email{juanita@bu.edu}
\urladdr{\url{https://juanitaduquer.github.io}}

\author[C. Keyes]{Christopher Keyes}
\address{Christopher Keyes, Department of Mathematics, King's College London, London, UK, and Heilbronn Institute of Mathematical Research, Bristol, UK}
\email{christopher.keyes@kcl.ac.uk}
\urladdr{\url{https://c-keyes.github.io}}

\author[A. Kobin]{Andrew Kobin}
\address{Andrew Kobin, Center for Communications Research, La Jolla, CA, USA}
\email{ajkobinmath@gmail.com}
\urladdr{\url{https://www.andrewkobin.com}}

\author[M. Roy]{Manami Roy}
\address{Manami Roy, Department of Mathematics, Lafayette College, Easton, PA, USA}
\email{royma@lafayette.edu}
\urladdr{\url{https://manamiroy.github.io}}

\author[S. Sankar]{Soumya Sankar}
\address{Soumya Sankar, Mathematical Institute, Utrecht University, Utrecht, Netherlands}
\email{s.sankar@uu.nl}
\urladdr{\url{https://sites.google.com/site/soumya3sankar/}}

\author[Y. Wang]{Yidi Wang}
\address{Yidi Wang, Department of Mathematics, University of Western Ontario, 
London, Ontario, Canada}
\email{ywan6443@uwo.ca}
\urladdr{\url{https://ywang-math.github.io}}


\begin{document}

\begin{abstract}
We characterize the integral Hasse principle for an infinite family of spherical stacky curves with genus $g\in [2/3,1)$ that are defined using generalized Fermat equations, extending a result of Darmon and Granville.  We then apply our methods to find that a positive proportion of curves in our family satisfy the integral Hasse principle.

\end{abstract}

\maketitle

\tableofcontents

\section{Introduction}
\label{sec:introduction}

An algebraic curve $X$ over \(\Q\) is said to satisfy the Hasse principle if the existence of a $\Q$-point on $X$ is equivalent to the existence of points over all completions $\Q_{v}$ of $\Q$. There is a dichotomy in the arithmetic of the curve: when \(X\) has genus \(0\), the Hasse principle holds for $X$ by the Hasse--Minkowski theorem. On the other hand, when its genus is at least \(1\) there are counterexamples to the Hasse principle \cite{Poonen} (in fact, they exist over every global field).

When the curve is replaced by a \emph{stacky curve}, \(\X\), the nature of the problem changes. If $\X$ has coarse space $X$, then away from the stacky locus, the coarse moduli map $\pi \colon \X\rightarrow X$ is an isomorphism, so understanding $\X(\Q)$ is nearly equivalent to understanding $X(\Q)$. However the following \emph{integral} version of the Hasse principle  exhibits more interesting phenomena. 

\begin{definition}
    A (relative) stacky curve $\X$ over $\Z$ satisfies the integral Hasse principle if $\X(\Z)\not = \emptyset$ is equivalent to $\X(\Z_{p})\not = \emptyset$ for all primes $p$ and $\X(\R)\not = \emptyset$.
\end{definition}

This problem was first studied in \cite{BhargavaPoonen}, where the authors show that the integral Hasse principle holds for all tame stacky curves of genus $g < \frac{1}{2}$. They also exhibit counterexamples to the Hasse principle for $g = \frac{1}{2}$. More counterexamples to the Hasse principle may be constructed using \cite[Prop.~8.1]{DarmonGranville}, which characterizes local and global solutions to certain \emph{generalized Fermat equations}, i.e.~Diophantine equations of the form 
\begin{equation}\label{eq:genfermat}
Ax^{p} + By^{q} = Cz^{r}
\end{equation}
for some $A,B,C,p,q,r\in\Z$ with $ABC\ne 0$ and $p,q,r\geq 2$. See \S\ref{subsec:gen-fermat} for more background on these equations, which are well-studied in the literature. 

In this article, we study the case when $A=1$ and $(p,q,r) = (2,2,n)$ for $n$ odd. For each odd $n$, we construct a stacky curve $\Y_{B,C}$ of genus $\frac{n - 1}{n}$ such that integral points on \(\Y_{B,C}\) correspond exactly to primitive solutions of~\eqref{eq:genfermat}. We generalize the local characterization of \cite{DarmonGranville} to all pairs of non-zero integers \((B,C)\) in Proposition~\ref{prop:localTest}. Extending their criterion for global points to allow \(B\) to be non-squarefree or in any congruence class modulo \(4\) turns to be quite subtle: it requires working with norm forms of and factorizations in arbitrary orders in quadratic fields. Our first main result towards a characterization of the existence of global integral points on the curves \(\Y_{B,C}\) is as follows. 

\begin{thma}[Theorem~\ref{thm:TFAEunified}]
\label{thm:TFAE-intro}
    Let $B = f^{2}B_{0}$, with $-1\ne B_{0}$ squarefree, $\gcd(f,C) = 1$, and $B_0\not\equiv 3\pmod 4$ if $C$ is even. Let $K = \Q\left (\sqrt{-B_{0}}\right )$ with ring of integers $\O_{K}$. Then there exist a localization $R_{K}$ of $\O_{K}$, an order $\O\subseteq\O_{K}$, and a finite and explicitly computable set $\admissible_{B,C}\subseteq H^{1}(R_{K},\mu_{n})$ such that the following are equivalent.
    \begin{enumerate}[label = (\roman*)]
        \item\label{part:intro-i} There exists an integral point on $\Y_{B,C}$, corresponding to a primitive integer solution $(x,y,z)$ to the equation $x^{2} + By^{2} = Cz^{n}$, such that $x + f\sqrt{-B_0}y \notin p\O_K$ for all prime numbers $p$. 
        \item\label{part:intro-ii} Each prime number $p$ divides $C$ at most once if $p$ is ramified in $K$ (resp.~zero times if $p$ is inert) and $\admissible_{B,C}$ is nonempty. 
        \item\label{part:intro-iii} In the ring $\O_{K}$, $C$ splits as a product of coprime ideals $C\O_{K} = \J_{+}\J_{-}\r$, where $\J_{\pm}$ are each products of split primes, $\r$ is a product of ramified primes, $\J_{+} = \overline{\J_{-}}$, and $[\J_{+}\cap\O]\in n\Cl(\O)$. 
    \end{enumerate}
\end{thma}

We denote by \(\Y_{B,C}(\Z; \annulus^0)\) the groupoid of integral points satisfying Theorem~\ref{thm:TFAE-intro}\ref{part:intro-i}.  When $B_0\equiv 3\pmod{4}$ and $C$ is even, slight modifications are required, but we show that a similar statement holds (Theorem~\ref{thm:TFAEunified_prime}).  The results we highlight in the rest of the introduction have analogous versions for the modified set when $B_0\equiv 3\pmod{4}$.

In Section~\ref{sec:cascade}, we develop methods that allow us to use Theorem~\ref{thm:TFAE-intro} to handle equations $x^{2} + By^{2} = Cz^{n}$ for general $(B,C)$, and detect when \(\Y_{B,C}(\Z) \neq \emptyset\), rather than just \(\Y_{B,C}(\Z; \annulus^0)\). We also discuss how to approach the problem when $A\ne 1$ in~\eqref{eq:genfermat} (see~\S\ref{subsec:Acoeff}). The main result is summarized in the following theorem. 

\begin{thmb}
\label{thm:MainCascadeTheorem}
   For any $(B,C)$,  there is a finite, explicitly described set of pairs $\mathcal{P} \subset \Z^2$ such that
   \[\Y_{B,C}(\Z) \neq \emptyset \iff \coprod_{(B', C') \in \mathcal{P}} S_{B', C'} \neq \emptyset,\]
   where $S_{B', C'} \subseteq \Y_{B', C'}(\Z)$ is a subgroupoid (depending on $B,C$) described explicitly in terms of divisibility conditions in Section~\ref{sec:cascade}.
\end{thmb}

The subgroupoids $S_{B', C'}$ are in fact contained in the points described in Theorem~\ref{thm:TFAE-intro}\ref{part:intro-i}, with additional divisibility conditions needed on the $y$- and $z$-coordinates. Of the two, the conditions on $y$ are more difficult to remove (see Proposition~\ref{prop:cascade-with-gcds}). All together, Proposition~\ref{prop:localTest} and Theorems~\ref{thm:TFAE-intro} and~\ref{thm:MainCascadeTheorem} characterize the Hasse principle for $\Y_{B,C}$ for all pairs of non-zero integers \((B,C)\).

\subsection{Approach via descent}
\label{subsec:intro-descent}

Our approach to Theorem~\ref{thm:TFAE-intro}, inspired by \cite{PSS}, involves identifying an \'etale $\mu_{n}$-torsor $\mathcal{C}'\to(\Y_{B,C})_{R_{K}}$ and computing a set of admissible twists $\{\mathcal{C}_{d}'\to(\Y_{B,C})_{R_{K}}\}$, whose $R_K$-points descend to points in \(\Y_{B,C}(\Z; \annulus^0).\) In particular, the set $A_{B,C}$ in Theorem~\ref{thm:TFAE-intro}\ref{part:intro-ii} consists of those $d$ for which $\mathcal{C}_{d}'\to(\Y_{B,C})_{R_{K}}$ is admissible, so that the equivalence of parts~\ref{part:intro-i}~and~\ref{part:intro-ii} of Theorem~\ref{thm:TFAE-intro} says that $\Y_{B,C}(\Z; \annulus^0) \neq \emptyset$ if and only if there exists an admissible twist $\mathcal{C}_{d}'\to(\Y_{B,C})_{R_{K}}$. In parallel to this, \cite[Prop.~8.1]{DarmonGranville} provides a class group obstruction to the existence of $\Z$-points on $\Y_{B,C}$. Theorem~\ref{thm:TFAE-intro} illustrates the connection between these two obstructions.

Having a clean description in terms of the class group of \(\O\) also explains why Theorem~\ref{thm:TFAE-intro} restricts to the case \(\gcd(f,C)=1\) and only detects points in \(\Y_{B,C}(\Z; \annulus^0)\). Indeed, the candidates for points that descend to \(\Z\) come from ideals trivial in \(\Cl(\O)\) and these principal ideals are typically of the form \((a + bf\sqrt{-B_0})\) with \(a \in \Z\) and \(\gcd(a,f)=1\) (see \S\ref{subsec:prelim-algebraic-number-theory}). Theorem~\ref{thm:MainCascadeTheorem} makes the jump to the more general case using primarily elementary techniques. It might be possible to combine Theorems~\ref{thm:TFAE-intro} and~\ref{thm:MainCascadeTheorem} into one by making different choices of torsors, but we choose not to do so in the interest of readability and the comparison with the class group. 

Theorem~\ref{thm:TFAE-intro} can be interpreted as showing that an explicitly computable \emph{refined} descent obstruction is the only obstruction to the Hasse principle for points in \(\Y_{B,C}(\Z; \annulus^0)\).
Indeed, Santens showed that for a stacky curve \(\Y\) of genus less than \(1\) there exists a finite \'etale group scheme $G$ and a $G$-torsor $\pi \colon \mathcal{C} \to \Y$ such that the descent along $\pi$ obstruction is the only obstruction to the integral Hasse principle (\cite[Thm.~1.1]{santens23}). However, in practice, it is difficult to study the integral Hasse principle in this way as it is hard to construct non-trivial \'etale covers of \(\Y\) over $\Z$.

To get around this issue, we compute a smaller descent locus from the subset $A_{B,C}\subseteq H^1(R_K, \mu_n)$ (which can be viewed as a ``refined Selmer set'') that allows one to describe the obstruction to $\Z$-points from the usual descent locus for $R_K$-points. More precisely, define the refined descent locus as
\[ \Y^{\pi, \textnormal{refined}}(\A_{R_K}) \coloneqq\coprod_{\twist \in {\admissible}_{B,C}} \pi_\twist(\C(\A_{R_K})) \subseteq \coprod_{\twist\in H^1(R_K, \mu_n)} \pi_{\twist}(\C_d(\A_{R_K})). \] 
Then, Theorem~\ref{thm:TFAE-intro} characterizes the following refined Hasse principle: \( \Y^{\pi, \textnormal{refined}}(\A_{R_K}) \neq \emptyset \implies \Y_{B,C}(\Z; \annulus^0) \neq \emptyset\).

\subsection{Statistics}
\label{subsec:intro-statistics}

Fixing the signature $(2,2,n)$ and letting $B$ and $C$ vary, we can study the following counting functions for \(T \in \R_{>0}\):
\begin{align*}
    N^{\loc}_n(T) &\colonequals \#\{(B,C) \in \Z^2 : |B|, |C| < T,\ \Y_{B,C}(\Z_p) \neq \emptyset \text{ for all primes } p\},\\
    N_n(T) &\colonequals \#\{(B,C) \in \Z^2 : |B|, |C| < T,\ \Y_{B,C}(\Z) \neq \emptyset \}. 
\end{align*}
These capture the statistics of the existence of local and global points, respectively, in the family of stacky curves $\Y_{B,C}$. We obtain the following asymptotic results for these functions. 

\begin{thmc}[Theorem~\ref{thm:local_stats}]
\label{thm:intro-local-stats}
    For all odd $n \geq 3$, $N_n^\loc(T)$ is asymptotically bounded above and below by a constant times $T^2/\sqrt{\log T}$, i.e.,\
    \[ N_{n}^{\loc}(T) \asymp \frac{T^{2}}{\sqrt{\log T}}. \]
\end{thmc}

\begin{thmd}[Theorem~\ref{thm:pos_prop_HP_n=3}]
\label{thm:intro-global-stats}
    We have
    \[\liminf_{T \to \infty} \frac{N_3(T)}{N_3^\loc(T)} > 0.00988.\]
\end{thmd}

That is, a positive proportion of the stacky curves $\Y_{B,C}$ of signature $(2,2,3)$ non-vacuously satisfy the Hasse principle for integral points\footnote{The Hasse principle is vacuously satisfied if $\Y_{B,C}(\Z_p) = \emptyset$ for some prime $p$.}. We obtain a similar result for odd primes $n > 3$ assuming a consequence of the Cohen--Lenstra heuristics for the $n$-part of the class group of $K = \Q\left (\sqrt{-B_0}\right )$, where $B_0$ is squarefree; see Theorem~\ref{thm:pos_prop_HP_n>3} and Assumption~\ref{assumption:CL}. 

Asymptotics for the number of locally soluble generalized Fermat equations of the form $Ax^n + By^n = Cz^n$ were recently established \cite{KPSS_locsolfermat}, fitting into an existing framework for local solubility statistics for fibrations of varieties. In particular, given a geometrically irreducible fibration $X \to \P^n$ over $\Q$, let $X_P$ denote the fiber over a point $P$. If there exist codimension 1 fibers which are irreducible but not geometrically integral, then 0\% of the fibers $X_P$ for $P \in \P^n(\Q)$ are everywhere locally soluble \cite[Thm.~1.1]{LoughranSmeets}. This is made more precise in terms of an asymptotic upper bound of the form $T^{n+1}/(\log T)^\Delta$ for the number of everywhere locally soluble fibers $X_P$ for $P \in \P^n(\Q)$ of height at most $T$, where $\Delta$ depends on the Galois action on these nonsplit codimension 1 fibers \cite[Thm.~1.2]{LoughranSmeets}.

While this perspective has yet to be extended to the setting of fibrations of stacky curves or to integral points, at a glance, Theorem~\ref{thm:intro-local-stats} follows the same pattern if we consider the fibration $\Y \to \A^2$ whose fiber over a point $(B,C) \in \A^2(\Z)$ is $\Y_{B,C}$. There appears to be one codimension 1 point $D \in \A^2$, namely $\{C = 0\}$, for which $\Y\vert_D$ is irreducible but not geometrically integral. Its components are defined over a quadratic extension, with the Galois group acting with no fixed components. In the setting of varieties as in \cite{LoughranSmeets}, this would suggest an asymptotic of the form $N_n^{\loc}(T) \ll T^2/\sqrt{\log T}$, which is indeed confirmed by Theorem~\ref{thm:intro-local-stats}. It would be interesting to see how asymptotics for the number of everywhere locally soluble fibers behave in other families of stacky curves, or perhaps to investigate an analogue of \cite[Thm.~1.2]{LoughranSmeets} in the setting of fibrations of stacks. 

\subsection{Structure of the paper}\label{subsec:structure}

In Section~\ref{sec:preliminaries}, we define stacky curves (\S\ref{subsec:stacky-curves}) and recall some of their properties, including quotient structures (\S\ref{subsec:prelim-points-pqr}), as well as the specific properties of the stacky curves $\Y_{B,C}$ attached to a generalized Fermat equation (\S\ref{subsec:gen-fermat}). We also collect some results from algebraic number theory in \S\ref{subsec:prelim-algebraic-number-theory} that we will use throughout the paper. 

In Section~\ref{sec:local-solubility}, we characterize $\Z_{p}$-points on the stacky curves $\Y_{B,C}$. Section~\ref{sec:global-descent} sets up the \'{e}tale descent problem over $\Y_{B,C}$ that will be used to characterize $\Z$-points on these curves. 

Sections~\ref{sec:mainthm} and~\ref{sec:cascade} are devoted to the proofs of Theorems~\ref{thm:TFAE-intro} and~\ref{thm:MainCascadeTheorem}, respectively. Additionally, we interpret our results in terms of cohomological descent in \S\ref{subsec:coho-descent}  and comment briefly on the case $A\not = 1$ and $(p,q,r) = (2,2,n)$ of \eqref{eq:genfermat}  in \S\ref{subsec:Acoeff}. In \S\ref{subsec:detailedExample}, we give a detailed example illustrating some of the main features of Sections~\ref{sec:mainthm} and~\ref{sec:cascade}.

Finally, in Section~\ref{sec:statistics} we analyze the counting functions $N_{n}^{\loc}(T)$ and $N_{n}(T)$ and prove Theorems~\ref{thm:intro-local-stats} and~\ref{thm:intro-global-stats}. All examples and the code accompanying this paper are available at~\cite{githubRepo}.

\subsection*{Acknowledgments} The authors would like to thank Brandon Alberts, Santiago Arango-Pi\~{n}eros, Stephanie Chan, Frits Beukers, Peter Koymans, Bjorn Poonen, John Voight, and David Zureick-Brown for helpful discussions.  This project began during the 2023 AMS Mathematics Research Communities: \textit{Explicit Computations with Stacks}.  We would like to thank the organizers for providing an energizing environment for starting this research. We are also grateful to the International Centre for Mathematical Sciences for their hospitality during a Research in Groups event, which allowed us to make progress on this manuscript.

JD was partially supported by the Simons Foundation grant \#550023. CK was partially supported by the Additional Funding Programme for Mathematical Sciences, delivered by EPSRC (EP/V521917/1) and the Heilbronn Institute for Mathematical Research, as well as an AMS-Simons Travel Grant. AK was partially supported by the American Mathematical Society and the Simons Foundation. MR was partially supported by an AMS-Simons Travel Grant. SS was supported by the Dutch Research Council (NWO) grant OCENW.XL21.XL21.011. YW was partially supported by the NSF grant DMS-2102987.

For the purpose of open access, a CC BY public copyright license is applied to any Author Accepted Manuscript version arising from this submission.

\section{Preliminaries}
\label{sec:preliminaries}

\subsection{Stacky curves}
\label{subsec:stacky-curves}
We proceed to give a brief introduction to the main objects of this paper, following \cite{kob, vzb}.

A \emph{stacky curve} $\calX$ is a smooth, proper, irreducible, $1$-dimensional Deligne--Mumford stack $\X$ over a field $k$ which is generically a scheme (an algebraic curve) \cite[Def.~5.2.1]{vzb}. 

A \emph{point} on a stacky curve $\X$ is an equivalence class of morphisms $x \colon \Spec K\rightarrow\X$, with $K/k$ a field extension, where two points $x \colon \Spec K\rightarrow\X$ and $x' \colon \Spec K'\rightarrow\X$ are equivalent if there is a common refinement 
\begin{center}
\begin{tikzpicture}[scale=1.5]
    \node at (0,0) (a) {$\Spec L$};
    \node at (1,.7) (b) {$\Spec K$};
    \node at (1,-.7) (c) {$\Spec K'$};
    \node at (2,0) (d) {$\X$};
    \draw[->] (a) -- (b);
    \draw[->] (a) -- (c);
    \draw[->] (b) -- (d) node[above,pos=.6] {$x$};
    \draw[->] (c) -- (d) node[below,pos=.6] {$x'$};
\end{tikzpicture}
\end{center}
with $L$ a field extension of $K$ and $K'$. 

The \emph{automorphism group} of a point $x$ on $\X$ is the group $\Aut(x) = \operatorname{Isom}_{\X(k)}(x,x)$ of automorphisms of $x$ as an object of the groupoid $\X(k)$. Since $\X$ is Deligne--Mumford, each $\Aut(x)$ is a reduced, finite group scheme. In particular, if $\bar{k}$ is an algebraic closure of $k$, then $\Aut(x)(\bar{k})$ is a finite group. The \emph{degree} of a stacky point $x$ is $\deg(x) \colonequals \frac{[k(x) : k]}{|\Aut(x)(\bar{k})|}$, where $k(x)/k$ is the minimal field of definition of $x$. We call $x$ a \emph{stacky point} if $|\Aut(x)(\bar{k})| > 1$. The condition that $\X$ is generically a scheme further implies that only finitely many points of $\X$ are stacky. We call $\X$ {\it tame} if the characteristic of $k$ does not divide any of the orders of the nontrivial automorphism groups of $\X$; otherwise $\X$ is {\it wild}. 

Following \cite[Sec.~11.2]{vzb} and other authors, we will also consider \emph{relative stacky curves} $\X/S$ over a scheme $S$, which are smooth, proper families of stacky curves $\X\rightarrow S$. In this article, these are just integral models of stacky curves over a number field, such as $\X$ and $\Y$ in the introduction; we will refer to them simply as stacky curves when the context is clear. 

A stacky curve $\X$ has a \emph{coarse moduli space} $X$ \cite[Thm.~11.1.2]{ols}, which can be viewed as the ``underlying curve'' of $\X$ (or underlying relative curve in the relative setting). If $K_{\X}$ is a canonical divisor on $\X$, we may define the genus $g(\X)$ of $\X$ by $\deg(K_{\X}) = 2g(\X) - 2$. The stacky Riemann--Hurwitz formula\footnote{See \cite[Prop.~5.5.6]{vzb} or \cite[Prop.~4.13]{kob} for the tame case and \cite[Prop.~7.1]{kob} for the wild case.} allows one to compute $g(\X)$ from the data of $g(X)$ and the structure of the nontrivial automorphism groups of $\X$; when $\X$ is tame, 
\begin{equation}\label{eq:genus}
    g(\X) = g(X) + \frac{1}{2}\sum_{x\in\X(k)} \left (1 - \frac{1}{|G_{x}|}\right )\deg\pi(x), 
\end{equation}
where $G_{x}$ is the automorphism group at $x$ and $\pi \colon \X\rightarrow X$ is the coarse moduli map. The formula in the wild case is more complicated \cite[Cor.~7.3]{kob} but will not be used here. 

\subsection{Points on quotient stacks}
\label{subsec:prelim-points-pqr}

Let 
\begin{equation}\label{eq:Sdef}
    S\colonequals\Spec\Z[x,y,z]/(f(x,y,z))\smallsetminus \{x=y=z=0\}\subseteq\A^3_\Z\smallsetminus\{0\},
\end{equation} 
where $f(x,y,z)$ is a ternary polynomial cutting out a codimension $1$ subscheme in $\A_{\Z}^{3}\smallsetminus\{0\}$.  Let the multiplicative group $\G_{m}$ act on the coordinates of $\A_{\Z}^{3}$ with weights $a,b,c$ and suppose that the action restricts to a $\G_{m}$-action on $S$ (that is, $f$ is weighted homogeneous with weights $a,b,c$).  Then the quotient stack $\calY\colonequals [S/\G_m]\subseteq \P(a,b,c)$ is a proper $1$-dimensional Deligne--Mumford stack with generic automorphism group $\mu_{d}$, where $d = \gcd(a,b,c)$. When $S$ is smooth and irreducible (i.e., $f$ has everywhere-nonvanishing differential), such a stack is a gerbe over a stacky curve. In our main case of interest, when $(a,b,c) = (2,2,n)$ with $n$ odd, $\calY$ will always be a genuine stacky curve. 

The presentation of \(\Y\) as a quotient stack gives the following description of its points. For a scheme \(T\), the groupoid \(\Y(T)\) consists of objects which are \(\G_m\)-torsors over \(T\) with \(\G_m\)-equivariant maps to \(S\). Thus we can represent an object in \(\Y(T)\) as a tuple \((L_T, x, y, z)\) such that \(L_T\) is a line bundle on \(T\), and \(x,y, z\) are sections of \(L_T^{\otimes a}, L_T^{\otimes b}, L_T^{\otimes c}\) respectively satisfying \(f(x,y,z) = 0\).  In particular, if \(T = \Spec R\) for a PID \(R\), we may represent an object in \(\Y(T)\) as a point \([x:y:z] \in \P(n,n,2)(R)\) satisfying the weighted homogeneous equation \(f(x,y,z) = 0\). 

For ease of notation, we will drop the the line bundle from the tuple $(L_{T},x,y,z)$ when it is clear from context.

The following proposition will be helpful when considering morphisms of stacky curves.
\begin{proposition}[{\cite[Lemma~3.5]{kob}}]  
\label{prop:eqCategories}  
    If $F \colon \calX \to \calY$ is a functor between categories fibred in groupoids over the category of schemes, then $F$ is an equivalence of categories fibred in groupoids if and only if for each scheme $T$, the functor $F_T\colon \calX(T)\to \calY(T)$ is an equivalence of categories.
\end{proposition}

\begin{proof}
    Follows from \cite[{\href{https://stacks.math.columbia.edu/tag/003Z}{Tag 003Z}}]{stacks-project}.
\end{proof}

\subsection{Generalized Fermat equations}
\label{subsec:gen-fermat}

A {\it generalized Fermat equation} is a Diophantine equation of the form 
\begin{equation}\label{eq:generalizedFermat}
    Ax^p + By^q = Cz^r,
\end{equation}
where $A, B, C\in\Z$ and $p, q, r \in \Z_{\ge 2}$. The integer solutions to~\eqref{eq:generalizedFermat} have been studied extensively in the literature; see~\cite{PSS} for a survey in the case $A = B = C = 1$ (and a complete description of the solutions in the case $(p,q,r) = (2,3,7)$) and \cite{DarmonGranville,Beukers} for a general treatment. 

Of principal interest are the {\it primitive solutions}, i.e.~solutions $(x,y,z)$ with $\gcd(x,y,z) = 1$. There may be a finite or an infinite number of primitive solutions to~\eqref{eq:genfermat} depending on whether $\chi = \frac{1}{p} + \frac{1}{q} + \frac{1}{r} - 1$ is positive, in which case there are either no primitive solutions or infinitely many \cite{Beukers}, or $\chi$ is negative, in which case there are at most finitely many primitive solutions \cite{DarmonGranville}. In many cases when $\chi < 0$, the full set of solutions has been found; see \cite{wilcox-grechuk24} for a recent survey. The case $\chi = 0$ also occurs --- in fact only when $(p,q,r) = (2,3,6),(2,4,4)$, or $(3,3,3)$, or permutations of these --- and reduces to classical descent techniques; see~\cite[Sec.~6]{DarmonGranville}. 

Set $d = \gcd(qr,pr,pq)$. Using \eqref{eq:generalizedFermat}, we can define the scheme $S$ as in~\eqref{eq:Sdef} so that $S(\Z)$ is the set of nontrivial primitive integer solutions to~\eqref{eq:generalizedFermat}. Each choice of parameters determines a $1$-dimensional stack 
\begin{equation}\label{eq:Ypqr}
    \Y \colonequals [S/\G_{m}] \subset [\A^{3}\smallsetminus\{0\}/\G_{m}] = \P\left(\frac{qr}{d},\frac{pr}{d},\frac{pq}{d} \right),
\end{equation}
One can check that $\Y$ is smooth over $\Z\left[\frac{1}{pqr}\right]$; it follows that $\Y$ is a relative stacky curve over this ring. 

\begin{remark}
    Note that there is an additional action by roots of unity in each coordinate of $\A^{3}$: setting $H = \G_{m} \cdot (\mu_{p}\times\mu_{q}\times\mu_{r})$ and taking the quotient of $S$ by this group action gives us another stacky curve $\X = [S/H]$. When $p,q$ and $r$ are pairwise coprime, the two quotient stacks are isomorphic, but in general $\Y$ can be a nontrivial cover of $\X$. For instance, when \((p,q,r)= (2,2,n)\) for \(n\) odd, the map $\Y\rightarrow\X$ is a degree $2$ \'{e}tale cover, under which the two stacky points of \(\Y\) map to a single integral point in \(\X\). Thus the problem for \(\X\) is not as interesting. 
    
    The main practical difference between these stacks is that $\X$ has geometric fundamental group $D_{2n}$, the dihedral group with $2n$ elements, while $\Y$ has geometric fundamental group $\mu_{n}$. Our stacky approach to  
    \eqref{eq:generalizedFermat}, adapted from \cite{PSS}, is much simpler over $\Y$ since its \'{e}tale covers are all abelian. 
\end{remark}

\subsubsection{Integrality of points on weighted projective stacks}
\label{subsec:prelin-integral-pts}

Let \(R\) be a Dedekind domain (not necessarily a PID) with fraction field \(F\), let \(\p\) a prime ideal of \(R\) and let $F_\p$ denote the completion of $F$ at $\p$ with normalized discrete valuation $v_\p$. Let \(\P(a, b, c)\) be the weighted projective stack with weights $a,b,c$. 

An \emph{$R$-point} (resp.~\emph{$F_\p$-point}) of $\P(a,b,c)$ is a morphism of stacks $\varphi\colon\Spec R\to\P(a,b,c)$ (resp. $\Spec F_\p\to\P(a,b,c)$). A \emph{representative} of such a point is an object in the isomorphism class of $\varphi$ in the groupoid $\P(a,b,c)(R)$, which can be identified with a tuple of $(x,y,z)$ defined over \(R\). We write $\varphi = [x : y : z]$, to be thought of as weighted homogeneous coordinates. 

A point \([x:y:z] \in \P(a,b,c)(F_\p)\) is called \emph{\(\p\)-integral} if
\[
    \min \left\{ \frac{v_{\p}(x)}{a}, \frac{v_{\p}(y)}{b}, \frac{v_\p(z)}{c}\right\} \in \Z
\]
for any representative $(x,y,z)$. Further, a representative $(x,y,z)$ is called \emph{\(\p\)-minimal} if the above minimum is in the interval \([0,1)\). In particular, if $R$ is a PID and $(x,y,z)$ is a fixed representative of a point in \(\P(a,b,c)(R)\) with \(x, y, z \in R\), then it is \(p\)-minimal if and only if \(p\nmid \gcd(x,y,z)\). 

Returning to the stacky curve $\Y = [S/\G_{m}]$ from~\eqref{eq:Ypqr}, for a PID \(R\), a point $[x : y : z]\in\Y(R)$ is by definition $\p$-integral at all primes $\p$. Taking a representative which is $\p$-minimal at all $\p$ yields a \emph{primitive} solution, i.e.~we have $(x,y,z)\in R^{3}$ with $\gcd(x,y,z) = 1$. 

\begin{remark}\label{rem:primitivity}
    The property of $\p$-integrality belongs to a point of $\P(a,b,c)$, i.e.\ the equivalence class. A representative of that class can be taken to be $\p$-minimal. As we are chiefly interested in the case where $R$ is a PID, namely the cases $R=\Z$, $\Z_p$, or $\O_{K,\invertedprimes}$ for $K$ a quadratic number field, we will use ``$[x:y:z] \in \P(a,b,c)(R)$" to mean that $(x,y,z)$ is a particular primitive representative.
\end{remark}

\subsubsection{Signature \((2,2,n)\) stacky curves}
\label{subsec:prelim-22n}

Let $\Y_{B,C}$ denote the stacky curve defined as in~\eqref{eq:Ypqr} from the generalized Fermat equation
\begin{equation}\label{eq:Y_BC}
    x^2 + By^2 = Cz^n.
\end{equation}
Throughout the paper we assume that $n$ is odd and that $B,C\not = 0$. Further, we will assume that \(-B\) is not a square in Sections~\ref{sec:global-descent},~\ref{sec:mainthm},~and~\ref{sec:cascade}.

Indeed, geometrically, the stacky curve $(\Y_{B,C})_{\bar{\Q}}$ has two stacky points $\sizedbracket{\pm \sqrt{-B}:1:0}$, each of degree $\frac{1}{n}$. These points are in $\Y_{B,C}(\Z)$ if and only if $-B$ is a square in $\Z$. When $-B$ is not a square, these points are not defined over $\Q$, so since \(n\) is odd, $(\Y_{B,C})_{\Q}$ has no stacky $\Q$-points. The relative stacky curve \(\Y_{B,C} \to \Spec \Z\), on the other hand, can have fibers with stacky points not visible over the generic point. If for instance \(p \mid C\), then the fiber \((\Y_{B,C})_{\F_p}\) contains the stacky \(\mu_2\) point \([0:0:1]\). If \(p \mid B\), then the fiber  \((\Y_{B,C})_{\F_p}\) contains the stacky \(\mu_2\) point \([0:1:0]\).

\begin{remark}
\label{rem:groupoids}
    Throughout this paper, we consider various integral models of our stacky curve $\Y_{B,C}$, which is a priori defined over $\Spec\Z$. When \(-B\) is not a square, $\Y_{B,C}(\Z)$ is already (equivalent to) a set. We will denote by $\Y_{B,C}(\Z;\bullet)$ the full subgroupoid of $\Y_{B,C}(\Z)$ supported on points satisfying condition ($\bullet$). Most of the time, these subgroupoids will be represented by sets, and we will make clear when groupoids are involved. 
    
    Setting $K = \Q\left (\sqrt{-B}\right )$ and passing to $\Y_{B,C/\orb_{K}}(\orb_{K})$, the stacky points $[\pm\sqrt{-B} : 1 : 0]$ are now defined over $K$. In fact, they are defined over any subring $\orb$ of $\O_K$ containing $\Z[\sqrt{-B}]$, so the existence of a point in $\Y_{B,C/\O}(\O)$ is not sufficient to guarantee an integer solution to~\eqref{eq:Y_BC}. This is one of the main difficulties that requires the delicate descent theory arguments in Sections~\ref{sec:global-descent}~and~\ref{sec:mainthm}.
\end{remark}

By formula~\eqref{eq:genus}, the genus of $(\Y_{B,C})_{\Q}$ is
\begin{equation}\label{eq:YBCgenus}
    g(\Y_{B,C}) = g(\mathbb{P}^1) + \frac{1}{2}\cdot \left(1 - \frac{1}{n}\right)\cdot 2 = \frac{n-1}{n}.
\end{equation}
In particular, since its genus is less than \(1\), $(\Y_{B,C})_{\Q}$  is spherical and so \(\Y_{B,C}\) either has no integral points or infinitely many by \cite[Thm.~1.2]{Beukers}.

\begin{remark}\label{rmk:geometric-genus}
    The genus/Euler characteristic calculated in~\eqref{eq:YBCgenus} is that of the generic fiber of \(\Y_{B,C}\). The Euler characteristic of the relative curve \(\Y_{B,C} \to \Spec \Z\) is not as well behaved, and can vary on fibers. 
\end{remark}

\subsection{Preliminaries from algebraic number theory}
\label{subsec:prelim-algebraic-number-theory}

We record some basic facts from algebraic number theory regarding ideal class groups of orders in number fields beginning with the fundamental exact sequence describing the ideal class group of the maximal order \(\O_K\) of \(K\). 
\begin{equation}\label{seq:ClK}
    \begin{tikzcd}
        1 \arrow[r] & K^\times/\O_K^\times \arrow[r, "\div"] & \oplus_\p \Z \arrow[r, "{[\cdot]}"]               & \Cl(\O_K) \arrow[r]  & 0 
    \end{tikzcd}
\end{equation}
Here $\p$ runs over all prime ideals of \(\O_K\). We denote by $\mathrm{div}$ the map sending an element $\alpha \in K^\times$ to the prime ideal decomposition of the principal fractional ideal $\alpha\O_K$, and use $[I]$ to denote the ideal class of a fractional ideal $I$ of $\O_K$.

The sequence~\eqref{seq:ClK} can be generalized to localizations. In particular, suppose $\invertedprimes$ is a finite set of primes of $\O_K$ such that $\Cl(\O_K) = \langle [\p]\,\colon\,\p \in \invertedprimes \rangle$. Set $R_K$ to be the localization of $\O_K$ obtained by inverting all elements supported at $\invertedprimes$. Then a diagram chase gives the following exact sequence.
\begin{equation}
\label{seq:ClK_with_RK}
    \begin{tikzcd}
        1 \arrow[r] & R_K^\times / \O_K^\times \arrow[r, "\div"] & \oplus_{\p \in \invertedprimes} \Z \arrow[r, "{[\cdot]}"] & \Cl(\O_K) \arrow[r] & 0
    \end{tikzcd}
\end{equation}

Sometimes it will be convenient to refer to the group of fractional \(\O_K\)-ideals of $K$, $I(\O_K)$, which we identify with $\oplus_{\p} \Z$. Making this identification, in~\eqref{seq:ClK} and~\eqref{seq:ClK_with_RK} we replace $\div$ with the inclusion $\alpha \mapsto \alpha \O_K$. Similarly, we may identify $K^\times / \O_K^\times$ with $P(\O_K)$, the group of principal fractional ideals of $K$.

Taking~\eqref{seq:ClK} and~\eqref{seq:ClK_with_RK} and applying the multiplication by $n$ map, we obtain the following exact sequences.

\begin{align}
    \label{seq:ClK_mod_n}
    0 \rightarrow {\Cl(\O_K)[n]} \rightarrow P(\O_K)/nP(\O_K) \rightarrow I(\O_K)/nI(\O_K) \rightarrow \Cl(\O_K) / n\Cl(\O_K) \rightarrow 0\\[1ex]
    \label{seq:ClK_with_RK_mod_n}
    0 \rightarrow {\Cl(\O_K)[n]} \rightarrow R_K^\times / ((R_K^\times)^n, \O_K^\times) \rightarrow \oplus_{\p \in \invertedprimes} \Z/n\Z \rightarrow \Cl(\O_K) / n\Cl(\O_K) \rightarrow 0
\end{align}

Let $\O \subset \O_K$ be an order of conductor $\f$. In general, $\O$ may fail to be a Dedekind domain. In order to generalize~\eqref{seq:ClK_mod_n} to this situation, we make the following definitions using \cite[Section~7]{Cox}:
\begin{itemize}
    \item $I(\O)$ denotes the group of invertible fractional $\O$-ideals;
    \item $I(\O, \f)$ denotes the group of fractional $\O$-ideals with norm prime to $\f$;
    \item $P(\O)$ and $P(\O, \f)$ denote subgroups of $I(\O)$ and $I(\O, \f)$, respectively, consisting of principal fractional $\O$-ideals; 
    \item $\Cl(\O) = I(\O) / P(\O)$ is the ideal class group of $\O$.
\end{itemize}

\begin{lemma}[See e.g.\ {\cite[Prop.\ 7.20]{Cox}, \cite[Thm. 3.8]{conrad_orders}}]
\label{lem:IOKf=IOf}
    We have a norm-preserving bijection 
    \[I(\O_K, \f) \to I(\O, \f)\]
    given by $\fraka \mapsto \fraka \cap \O$, with inverse given by $\fraka \mapsto \fraka\O_K$.
\end{lemma}
We can present $\Cl(\O)$ by only those ideals prime to its conductor $\f$.

\begin{lemma}[See e.g.\ {\cite[Prop.\ 7.19]{Cox}, \cite[Thm.~5.2]{conrad_orders}}]
\label{lem:ClO_primetof}
    Let $\O$ be an order in a number field $K$. For any integer $M$ and ideal class $[\fraka] \in \Cl(\O)$, there exists an ideal $\mathfrak{b} \in I(\O, M)$ such that $[\mathfrak{b}] = [\fraka]$. In particular, we have
    \[\Cl(\O) \simeq I(\O, \f) / P(\O, \f).\]
\end{lemma}

This leads to a useful exact sequence (see e.g.\ \cite[(5.6)]{conrad_orders}) that relates the class groups of $\O$ and $\O_K$.
\begin{equation}
\label{seq:ClO_to_ClK}
    \begin{tikzcd}
        1 \arrow[r] & \O_K^\times / \O^\times \arrow[r] & (\O_K/\f)^\times / (\O / \f)^\times \arrow[r] & \Cl(\O) \arrow[r] & \Cl(\O_K) \arrow[r] & 1
    \end{tikzcd}
\end{equation}

Combining~\eqref{seq:ClK_mod_n},~\eqref{seq:ClK_with_RK_mod_n}, and the analogous sequence for an order $\O$, we can assemble the following diagram with exact rows.
\begin{equation}
\label{seq:thebigdiagram}
\adjustbox{scale=0.8,center}{
    \begin{tikzcd}[column sep=small]
    & 1 \arrow[r]                               & {\Cl(\O)[n]} \arrow[r] \arrow[d]                                                  & {P(\O, \f)/nP(\O,\f)} \arrow[r] \arrow[d]                         & {I(\O,\f)/nI(\O,\f)} \arrow[r] \arrow[d, "\simeq"] & \Cl(\O)/n\Cl(\O) \arrow[r] \arrow[d, two heads] & 0 \\
    & 1 \arrow[r]                               & {\Cl(\O_K)[n]} \arrow[r] \arrow[d, "="]                                           & {P(\O_K, \f)/nP(\O_K,\f)} \arrow[r] \arrow[d]                     & {I(\O_K,\f)/nI(\O_K,\f)} \arrow[r] \arrow[d, hook] & \Cl(\O_K)/n\Cl(\O_K) \arrow[r] \arrow[d, "="]   & 0 \\
    & 1 \arrow[r]                               & {\Cl(\O_K)[n]} \arrow[r]                                                          & P(\O_K)/nP(\O_K) \arrow[r]                                        & I(\O_K)/nI(\O_K) \arrow[r]                         & \Cl(\O_K)/n\Cl(\O_K) \arrow[r]                  & 0 \\
    1 \arrow[r] & {\Cl(\O_K)[n]} \arrow[r] \arrow[ru, "="'] & {R_K^\times / ((R_K^\times)^n, \O_K^\times)} \arrow[r] \arrow[r] \arrow[ru, hook] & \oplus_{\p \in \invertedprimes} \Z/n\Z \arrow[ru, hook] \arrow[r] & \Cl(\O_K)/n\Cl(\O_K) \arrow[r] \arrow[ru, "="']    & 0                                               &  
\end{tikzcd}
}
\end{equation}

From~\eqref{seq:thebigdiagram}, we extract the following diagram.
\begin{equation}
\label{seq:def_psi}
    \begin{tikzcd}
        & {P(\O, \f)/nP(\O,\f)} \arrow[d, "\psi_\O"] \\
    R_K^\times / (R_K^\times)^n \arrow[r] & {P(\O_K)/nP(\O_K)}  
    \end{tikzcd}
\end{equation}
The map $\psi_\O$ is the composition of maps in the third column of~\eqref{seq:thebigdiagram}, while the horizontal map arises from precomposing the inclusion with the natural map to $R^\times_K / ((R_K^\times)^n, \O_K^\times)$.

\begin{remark}[Real quadratic fields]
    Let $K/\Q$ be a real quadratic field and $\O \subset \O_K$ an order. Along with the \textit{ideal} class group $\Cl(\O)$, there is the \textit{narrow} class group $\Cl^+(\O)$, obtained by taking the quotient of $I(\O)$ by the totally positive principal ideals.  Either these groups coincide or we have $[\Cl^+(\O) : \Cl(\O)] = 2$. When $n$ is odd, in either case we find
    \begin{align*}
        &\Cl^+(\O) / n\Cl^+(\O) \simeq \Cl(\O)/n\Cl(\O)& &\text{and}& &\Cl^+(\O)[n] \simeq \Cl(\O)[n].&
    \end{align*}
    In the present paper, we are concerned only with $n$ odd and with quotients of $\Cl(\O)$ by $n$, so we need not distinguish between these two notions of class group.
\end{remark}

We conclude by recording a class number formula for (relative) orders in quadratic fields.

\begin{lemma}[See {\cite[Corollary 7.28]{Cox}\footnote{The result in \cite{Cox} is stated for class groups of binary quadratic forms with negative discriminant, but the formula~\eqref{eq:class_number_formula} holds for ideal class groups of general quadratic fields; see also \cite[Thm.~5.3]{conrad_orders}.}}]
\label{lem:class_number_formula}
    Let $K$ be a quadratic number field and \(\orb' \subset \orb \) denote orders contained in \(\orb_K\) of conductors \(\f'\) and \(\f\) respectively, with relative conductor $\r = \f'/\f$. Then,
    \begin{equation}\label{eq:class_number_formula}
        \frac{\#\Cl(\O')}{\#\Cl(\O)}[\O^\times : {\O'}^\times ] = \r \prod_{p \mid \r} \left(1 - \left(\frac{\Disc(\O)}{p}\right)\frac1p\right).
    \end{equation}
\end{lemma}

\section{Local solubility}\label{sec:local-solubility}

Let $p$ be a prime number.  In this section, we provide a characterization of when $\Y_{B,C}(\Z_p)$ is empty.  The main result here is the following proposition. 

\begin{proposition}\label{prop:localTest}
    Let $p$ be a prime and write $B = p^kB'$ and $C = p^\ell C'$ for $p \nmid B', C'$. Then $\Y_{B,C}(\Z_p) = \emptyset$ if and only if $\ell$ is odd and either
	\begin{enumerate}[label = (\roman*)]
		\item\label{part:loc-1} $k$ is odd and $\ell < k < n+\ell$; or
		\item\label{part:loc-2}
  $p \neq 2$, $k$ is even, $k < n+\ell$, and $\left(\frac{-B'}{p}\right) = -1$; or
            \item\label{part:loc-3} $p=2$, $k$ is even, $k < \ell$, and $B' \equiv 3 \pmod{8}$; or
		\item\label{part:loc-4} $p=2$, $k$ is even, $\ell < k < n+\ell-2$, and $B' \equiv 1, 3, 5 \pmod{8}$; or
		\item\label{part:loc-5} $p=2$, $k = n+\ell - 2$, and $B' \equiv 1,5 \pmod{8}$.
	\end{enumerate}
\end{proposition}

\begin{remark}
    In particular, when \(B' \equiv 7 \mod 8\) and $p = 2$, $\Y_{B,C}(\Z_{2})\not = \emptyset$ for any $k,\ell$. 
\end{remark}

We will use repeatedly use the following version of Hensel's lemma. For \(a \in \Z_p\), recall that \(|a|_p = p^{-v_p(a)}\) is the \(p\)-adic absolute value of \(a\). For \(\mathbf{a} = (a_1, a_2, \dots, a_d) \in \Z_p^d\), let \(\|\mathbf{a}\|_p = \max_i \{|a_i|_p\}.\)

\begin{lemma}[Hensel's lemma]
    Let \(f \in \Z_p[X_1, X_2 \ldots , X_d ]\) be a multivariate polynomial. Suppose \(\mathbf{a} \in \Z_p^d\) satisfies
    \[
    |f(\mathbf{a})|_p < \|\nabla f (\mathbf{a}))\|_p^2
    \]
    Then there is an \(\mathbf{\alpha} \in \Z_p^d\) such that \(f(\mathbf{\alpha}) = 0 \) and \(\|\mathbf{\alpha} - \mathbf{a}\|_p < \|\nabla f (\mathbf{a}))\|_p\).
\end{lemma}

\begin{proof}
    See e.g.\ \cite[Thm.~2.1]{conrad_Hensel}.
\end{proof} 

We first prove some useful lemmas. Recall  from Remark~\ref{rem:primitivity} that for a PID \(R\), when we write $[x: y: z] \in \mathcal{Y}_{B,C}(R)$, we assume the point is primitive. 

\begin{lemma} \label{lem:loc-bijection} 
    Let $R$ be any PID and suppose $p \in R$ is a prime element. If $B,C \in R$ and \(p^2\mid B, C\), then the map
    \begin{align}\label{eq:reduce_by_2}
        \Y_{B,C}(R) &\to \Y_{B\cdot p^{-2}, C\cdot p^{-2}}(R) \\
        \nonumber [x:y:z] & \mapsto [x/p : y : z]
    \end{align}
    is an isomorphism of groupoids.
\end{lemma}

\begin{proof}
    The map on objects is well-defined since $p \mid x$ is necessary for $[x:y:z] \in \Y_{B,C}(R)$. To show that the inverse map $[x' : y' : z'] \mapsto [px' : y' : z']$ is well-defined, note that $[px' : y' : z']$ is primitive if and only if $p \nmid y'$ or $p \nmid z'$. If \(p\mid y'\) and \(z'\), then necessarily \(p\mid x'\), contradicting the primitivity of $[x' : y' : z']$. 
    
    The map defines an isomorphism of groupoids since each stacky point is mapped to a stacky point and its automorphism group is preserved.
\end{proof}

We will use Lemma~\ref{lem:loc-bijection} in this section with $R=\Z_p$. In Section~\ref{sec:cascade}, we will set $R=\Z$. 

\begin{lemma}
\label{lem:possible_lemma}
    Maintaining the notation from Proposition~\ref{prop:localTest}, suppose \(1 < k < n+1\) and  \(\ell =1\). Then $[x:y:z] \in \Y_{B,C}(\Z_p)$ if and only if one of the following holds:
    \begin{enumerate}[label=(\roman*)]
        \item $k$ is even, $p^{k/2} \mid x$, $p \mid z$, and $\sizedbracket{\frac{x}{p^{k/2}}:y:\frac{z}{p}}$ satisfies
        \begin{equation}\label{eq:aux_even}
            x^2 + B'y^2 = p^{n+1-k}C'z^n;
        \end{equation}
        
        \item $k$ is odd, $p^{(k+1)/2} \mid x$, $p \mid z$, and $\sizedbracket{\frac{x}{p^{(k+1)/2}}:y:\frac{z}{p}}$ satisfies
        \begin{equation}\label{eq:aux_odd}
            px^2 + B'y^2 = p^{n+1-k}C'z^n.
        \end{equation}
    \end{enumerate}
\end{lemma}

\begin{proof}
    Since $p \mid B, C$, we have $p \mid x$. By the hypotheses $k > 1$ and $\ell = 1$ we must also have $p \mid z$. Thus $p^{n+1}$ divides the right-hand-side of~\eqref{eq:Y_BC}, and the hypothesis $1 < k < n+1$ forces $p^k \mid x^2$, which implies $p^{k/2} \mid x$ if $k$ is even, and $p^{(k+1)/2} \mid x$ if $k$ is odd.

    Taking $(x', y', z') = (x/p^{\lceil k/2 \rceil}, y, z/p)$ and rearranging~\eqref{eq:Y_BC}, we obtain
    \[p^{2\lceil k/2 \rceil - k}x'^2 + B'y'^2 = p^{n+1-k} C'z'^n,\]
    which specializes to~\eqref{eq:aux_even} and~\eqref{eq:aux_odd} when $k$ is even or odd, respectively.
\end{proof}

\begin{lemma}\label{lem:local-test-reduction}
    Maintaining the notation from Proposition~\ref{prop:localTest}, we have $\Y_{B,C}(\Z_p) \neq \emptyset$ whenever $\ell$ is even or $\ell$ is odd and $k \geq n + \ell$. 
\end{lemma} 

\begin{proof}
    If $\ell$ is even, then $[p^{\ell/2} C'^{(n+1)/2} : 0 : C'] \in \Y_{B,C}(\Z_p)$. Assume now that $\ell$ is odd and $k \geq n + \ell$. By Lemma~\ref{lem:loc-bijection}, it suffices to prove $\Y_{B,C}(\Z_p) \neq \emptyset$ for $\ell = 1$ and $k \geq n+1$.

    If \(k>n+1\), then the point \([C'^{(n+1)/2}: 1: C']\) satisfies an auxiliary equation
   \begin{equation}\label{eq:auxiliary_k_odd_big}
		x^2 + p^{k-(n+1)}B'y^2 = C'z^n.
	\end{equation}
    modulo \(p\). Lifting this solution using Hensel's lemma gives us a \(\Z_p\) solution \([x':y':z']\) to~\eqref{eq:auxiliary_k_odd_big}, with \(y' \in \Z_p^{\times}\). The point \([p^{(n+1)/2}x': y': pz']\) then satisfies the equation for $\Y_{B,C}$ and is $p$-integral.

    On the other hand, if \(k = n + \ell = n+1\), then since \(B',C' \in \Z_p^{\times}\), we have that
    \[
    \left[ 0 : \frac{C'^{(n+1)/2}}{B'^{(n+1)/2}} : \frac{pC'}{B'} \right] \in \Y_{p^kB', pC'}(\Z_p)
    \]
    so that \(\Y_{B,C}(\Z_p) \neq \emptyset\) in this case.
\end{proof}

Now we are ready to prove Proposition~\ref{prop:localTest}. 

\begin{proof}[Proof of Proposition~\ref{prop:localTest}] 

By Lemma~\ref{lem:local-test-reduction}, we find $\Y_{B,C}(\Z_p) \neq \emptyset$ whenever $\ell$ is even or $k \geq n + \ell$. Thus we assume $\ell$ is odd and $k< n + \ell$. \bigskip
   
Proof of~\ref{part:loc-1}: suppose $k$ is odd. We first claim that if \(k \le \ell\), then \(\Y_{B,C}(\Z_p) \neq \emptyset\). Indeed, since both \(k\) and \(\ell\) are odd, we may invoke Lemma~\ref{lem:loc-bijection} to assume that \(k=1\). Then 
\[\left[0 : p^{(\ell - 1)/2} \frac{C'^{(n+1)/2}}{B'^{(n+1)/2}} : \frac{C'}{B'}\right] \in \Y_{B,C}(\Z_p).\] 
Suppose now that \(\ell<k< n+ \ell\) and that \(k\) is odd. We may assume by Lemma~\ref{lem:loc-bijection} that \(\ell =1\) and $1 < k < n+1$. Invoking Lemma~\ref{lem:possible_lemma}, if $[x:y:z] \in \Y_{B,C}(\Z_p)$, then we have $p^{(k+1)/2} \mid x,\ p \mid z$, and $[x/p^{(k+1)/2} : y : z/p]$ satisfies~\eqref{eq:aux_odd}. However, the latter implies $p \mid y$, contradicting the primitivity of $[x:y:z]$. Therefore $\Y_{B,C}(\Z_p) = \emptyset$.
\bigskip

Proof of~\ref{part:loc-2}: suppose $p$ is an odd prime and $k < n+\ell$ is even. We have two cases to consider: $k < \ell$ and $k > \ell$. 

If $k < \ell$, then by Lemma~\ref{lem:loc-bijection} we may assume $k=0$. Since $p \nmid B'$ and $p > 2$, the equation~\eqref{eq:Y_BC} has a nontrivial solution modulo $p$ if and only if the Legendre symbol $\left(\frac{-B'}{p}\right) = 1$. Moreover, this solution may be lifted to $\Z_p$ by Hensel's Lemma.

If $\ell < k < n + \ell$, then again by Lemma~\ref{lem:loc-bijection}, we may assume that $\ell = 1$, so $1 < k < n+1$. By Lemma~\ref{lem:possible_lemma}, any solution $[x:y:z] \in \Y_{B,C}(\Z_p)$ satisfies $p^{k/2} \mid x,\ p \mid z$, such that $[x/p^{k/2} : y : z]$ satisfy~\eqref{eq:aux_even}. The latter has a solution modulo $p$ if and only if $\left(\frac{-B'}{p} \right) = 1$, and if it exists, such a solution may be lifted to a point in $\Y_{B,C}(\Z_p)$ by Hensel's Lemma. 
\bigskip

Proof of~\ref{part:loc-3}: Suppose \(p=2\) and \(k < \ell\) is even. By Lemma~\ref{lem:loc-bijection}, we may assume \(k =0\). We claim that in this case, \(\Y_{B,C}(\Z_2) = \emptyset\) if and only if \(B' \equiv 3 \pmod 8\). Consider first the case \(\ell=1\). Note that if  \([x:y:z] \in \Y_{B',2C'}(\Z_2)\), then \(x\) and \(y\) must both be in \(\Z_2^{\times}\), since the alternative would violate primitivity. In particular, 
\begin{equation}\label{eq:loc-test-mod2}
    x^2 + B'y^2 \equiv 1 + B' \pmod 8.
\end{equation}
On the other hand, \(2C'z^n\) is \(2\) or \(6 \pmod 8\) if  \(z \in \Z_2^{\times}\), and \(0 \pmod 8\) if \(z \in 2 \Z_2\). Thus \(B' \equiv 1, 5, 7 \pmod 8\). Conversely, if \(B' \equiv 1,5,7 \pmod 8\), a solution to~\eqref{eq:loc-test-mod2} with \(x, y\) odd can be lifted to \([x:y:z] \in \Y_{B',2C'}(\Z_2)\) via Hensel's lemma, with \(x, y \in \Z_2^{\times}.\)

For \(\ell \ge 3\), we have that \([x:y:z] \in \Y_{B', 2^{\ell}C'}(\Z_2)\) must satisfy \(x^2 + B'y^2 \equiv 0 \mod 8\). If \(x, y \in \Z_2^{\times}\), then \(B' \equiv 7 \mod 8\). If either \(x, y \in 2\Z_2\), then they are both in \(2\Z_2\) and primitivity ensures that \(z \in \Z_2^{\times}\). Therefore \([x/2: y/2: z] \in \Y_{B', 2^{\ell-2}C'}(\Z_2)\). Repeating this process, we are reduced to the case \(\ell=1\) as in the previous paragraph and the rest of the proof follows as before. 
\bigskip

Proof of~\ref{part:loc-4} and~\ref{part:loc-5}: suppose \(p=2\), \(\ell < k < n+ \ell\). By Lemma~\ref{lem:loc-bijection}, we may assume that \(\ell=1\) and \(1 < k <  n+1\). By Lemma~\ref{lem:possible_lemma}, if $[x:y:z] \in \Y_{2^kB',2C'}$, then $2^{k/2} \mid x$ and $2 \mid z$ and $[x/2^{k/2} : y : z/2]$ satisfies the equation for $\Y_{B',2^{n+1-k}C'}$ and necessarily $y \in \Z_2^\times$. Now there are two cases depending on whether $k < n-1$ (part~\ref{part:loc-4}) or $k = n-1$ (part~\ref{part:loc-5}).

If $k < n-1$, then $\Y_{B',2^{n+1-k}C'}(\Z_2) \neq \emptyset$ implies $1 + B' \equiv 0 \pmod{8}$, whence $B' \equiv 7 \pmod{8}$. Indeed, whenever $B' \equiv 7 \pmod{8}$ we have $\Y_{B,C}(\Z_2) \neq \emptyset$, since $\Y_{B',2^{n+1-k}C'}$ has solutions modulo 2 that may be lifted by Hensel's Lemma. This concludes the proof of~\ref{part:loc-4}.

If $k = n-1$, then $\Y_{B',2^{n+1-k}C'}(\Z_2) = \Y_{B',4C'}(\Z_2)\neq \emptyset$ implies $B' \equiv 3 \pmod{4}$. Conversely, we see that $\Y_{B',4C'}(\Z/8\Z)$ has a point with $x,y$ odd whenever $B' \equiv 3 \pmod{4}$. In either case, such a solution may be lifted to a $\Z_2$-point $[x':y':z'] \in \Y_{B', 4C'}(\Z_2)$ with $x',y' \in \Z_2^\times$. Finally, $[2^{k/2}x' : y' : 2z'] \in \Y_{2^kB', 2C'}(\Z_2)$.
\end{proof}

\section{Descent for global points}\label{sec:global-descent}

In this section we characterize $\Z$-points on $\Y_{B,C}$. Our proof has two major components: the first step (\S\ref{subsec:descent}) is to introduce an auxiliary stacky curve $\Y'$ and perform a descent procedure for a torsor $\calC'\to\Y'$ over a localization of the ring of integers of a number field, namely $K = \Q\left (\sqrt{-B}\right )$. This gives us a set containing the integral points of \(\Y_{B,C}(\Z)\), but the containment might be strict. The second step (\S\ref{subsec:denominator-test}) is to examine the local conditions at each prime and identify conditions under which a point in the superset is indeed a point in \(\Y_{B,C}(\Z)\). This sets up the proof of the main theorems in Section~\ref{sec:mainthm}.

\subsection{Descent}
\label{subsec:descent}
Fix nonzero positive integers $B$ and $C$ such that $-B$ is not a square, and let \(\Y \colonequals \Y_{B,C}\). Also set $K \colonequals \Q\big(\sqrt{-B}\big)$ and denote by $\O_{K}$ its ring of integers. 
We introduce the following auxiliary stacky curve. 
\begin{equation}
\label{eqn:construction-for-Y'}
\Y' \colonequals [S'/\G_m],\text{ where }S' \colonequals \left(\Spec \Z[u,v,w]/(uv - Cw^n) \right)\smallsetminus \{(0,0,0)\} \subset \A_{\Z}^3. 
\end{equation}
Here, the action of \(\G_m\) on \(S'\) sends \((u,v,w) \mapsto (\lambda^nu, \lambda^n v, \lambda^2 w)\). Using the quotient stack structure as in \S\ref{subsec:prelim-points-pqr}, we may think of \(\Y'\) as being a codimension 1 substack of \(\P(n,n,2)\) cut out by the weighted homogeneous equation \(uv = Cw^n\); explicitly, if \(R\) is a PID, then the objects \([u:v:w]\in\Y'(R)\subseteq\P(n,n,2)(R)\) correspond to primitive $R$-solutions to \(uv = Cw^n\).

We will now construct a morphism of stacks
\(\YYprime \colon \Y \to \Y'\) defined over \({\orb_K}[1/2\sqrt{-B}]\). By the $2$-Yoneda lemma, it is sufficient to construct a map \(\YYprime_T\colon\Y(T) \to \Y'(T)\) for any \(\Spec \orb_K[1/2\sqrt{-B}]\)-scheme \(T\). Define\footnote{We abuse notation here by using coordinates to represent sections of line bundles. See \S\ref{subsec:prelim-points-pqr}.} 
\begin{equation}\label{eq:Y-Y'-iso}
    \begin{array}{cccc}\YYprime_T\colon&\Y(T) &\to &\Y'(T)\\
    &[
    x: y: z] & \mapsto &\left[ 
    x + \sqrt{-B}y: x - \sqrt{-B}y: z\right].
    \end{array}
\end{equation}
The stacky points
\([\pm \sqrt{-B}: 1 :0]\) map to the stacky points \([2\sqrt{-B}:0:0]\) and \([0:-2\sqrt{-B}:0]\) respectively. 

We claim that the morphism $\YYprime$ is an isomorphism of stacks over \(\orb_K[1/2\sqrt{-B}]\). Indeed by 
Proposition~\ref{prop:eqCategories}, we only need to check that \(\YYprime_T \colon \Y(T) \to \Y' (T) \) is an equivalence of groupoids for any \(\orb_K[1/2\sqrt{-B}]\)-scheme \(T\). This can be checked by constructing a quasi-inverse as follows
\begin{equation}\label{eq:Y-Y'-inverse}
     \begin{array}{cccc}
     \Y'(T) &\rightarrow &\Y(T)\\
     {[u:v:w]} &\mapsto &\displaystyle\left[\frac{u+v}{2}: \frac{u-v}{2\sqrt{-B}}: w \right].
 \end{array}
 \end{equation}

\begin{lemma}\label{lemma:Yprime-integer-points}
    Write $B=f^2B_0$ for $-1\ne B_0$ squarefree. Under the map $\YYprime$ from~\eqref{eq:Y-Y'-iso}, the image of $\Y(\Z)$ is contained in $\Y'\left (\orb_{K}\left[1/{2f}\right ]\right )$. In particular, if $B$ is squarefree, then the image of $\Y(\Z)$ is contained in $\Y'\left(\O_K\left[1/{2}\right]\right)$. 
\end{lemma}

\begin{proof}  
Let \([x:y:z] \in \Y(\Z)\) denote an integral point that is $p$-minimal at all rational primes $p$, and let $[u:v:w]\colonequals \YYprime([x:y:z]) \in \Y'(K)$. 
We have that $u,v,w \in \O_K$; we claim that for all primes $\p \subset \O_K[1/2f]$ we have $\min\{v_\p(u), v_\p(v), v_\p(w)\} = 0$. 
Using the equations for the quasi-inverse of $\YYprime_{T}$ in~\eqref{eq:Y-Y'-inverse}, we see
\begin{align}\label{eq:valuationsIneq}
        \val_{\p}(x) &= \val_{\p}\left(\frac{u+v}{2}\right) \ge \min\{\val_{\p}(u), \val_{\p}(v)\},\\
        \val_{\p}(y) &= \val_{\p}\left(\frac{u-v}{2\sqrt{-B}}\right) \ge \min\{\val_{\p}(u), \val_{\p}(v)\} -\val_{\p}\left(f\sqrt{-B_0}\right).\notag
    \end{align}
    We assume for the sake of contradiction that $u,v,w\in\p$ for some prime $\p\subseteq\O_K\left[1/{2f}\right]$.  In particular, $\p \mid z$.  If $\p$ is not ramified in $K$, then $\val_\p(f\sqrt{-B_0})=0$, so 
    the inequalities from~\eqref{eq:valuationsIneq} imply that $\Nm(\p) \mid x,y$, violating integrality.
    
    If $\p$ is ramified, then $\val_\p(\sqrt{-B_0})=1$. 
    In this case, $\val_\p(x)$ must be an even integer (so $\val_\p(x)\ge 2$). Assume for the sake of contradiction that  $\val_{\p}(y)=0$.  Thus
    \begin{equation*}
        \val_\p(u),\val_\p(v)= \min\left\{\val_\p(x),\val_\p\left(f\sqrt{-B_0}\right)+\val_\p(y)\right\} = 1. 
    \end{equation*}
          Then from the equation $uv = Cw^n$, we obtain
    \begin{equation*}
        2=\val_\p(u) + \val_\p(v) = \val_\p(uv) = \val_\p(Cw^n) \ge n \ge 3,
    \end{equation*} 
    a contradiction. 
\end{proof}

Next, we will perform explicit descent on $\Y'$ over a suitable open subscheme of \(\Spec \orb_K\) by describing a   cover $\pi'\colon\C'\to \Y'$ and its twists. We construct the curve \(\mathcal{C}'\) as follows. Let \[
S'' \colonequals \left(\Spec \Z[U,V,W]/(UV - CW) \right)\smallsetminus \{(0,0,0)\} \subset \A_{\Z}^3.
\]
Define
\begin{equation}\label{eqn:construction-for-C'}
    \mathcal{C}' \colonequals [S''/\G_m],
\end{equation}
where the action of \(\G_m\) on \(S''\) sends \((U,V,W) \mapsto (\lambda U, \lambda V, \lambda^2 W)\). This curve is embedded into  \(\P(1,1,2)\) by  construction. Note that $\C'$ does not intersect the stacky point of \(\P(1,1,2)\), so it is a (non-stacky) algebraic curve. Now, define a map \(\pi'  \colon \mathcal{C}' \to \Y'\) by specifying
\begin{equation}\label{CT-to-YT}
\begin{array}{cccc}
    \mathcal{C}'(T) &\rightarrow &\Y'(T)\\
    \left[U:V:W\right] &\mapsto &\displaystyle\left[\frac{U^n}{C^{(n-1)/2}}: \frac{V^n}{C^{(n-1)/2}}: W\right]
\end{array}
\end{equation}
where \(T\) is any $\Z[1/C]$-scheme. 

The map \(\pi'\colon \C' \to \Y'\) is an \'etale \(\mu_n\)-torsor over \(\Z[1/nC]\).  To see this, we first note that $\pi'$ is compatible with the \(\mu_n\) action by definition.  Also, we find that the stacky points are exactly the points fixed by the action of \(\mu_n\).  Indeed, \([\zeta U: \zeta^{-1}V: W]=[U: V: W]\) in \(\P(1,1,2)\) implies  \(\ \zeta U=\lambda U\), \(\ \zeta^{-1} V=\lambda V\) and \(\ W=\lambda^2 W\). Then  \(W = 0\), or $\zeta=\pm 1$. Since \(n\) is odd, \(-1 \notin \mu_n\), so we have $\zeta=1$.  Hence, the cover $\pi'$ is compatible with the \(\mu_n\)-action, and the fixed points of this action are precisely the stacky points, thus $\Y' = [\C' / \mu_n]$.  Finally, we invert $n$ to obtain an \'etale cover.  In total, this makes $\C' \to \Y'$ into a $\mu_n$-torsor over $\Z\big[\frac{1}{nC}\big]$.

\begin{notation}
\label{notation:RK}
    As before, we write $B=f^2 B_0$ with $-1\ne B_0$ squarefree.  Consider a finite set $\invertedprimes$ of primes in $K$ with the property that
    \begin{equation}
    \label{eq:invertedPrimes}
    \invertedprimes\text{ contains all primes above }2nCf,\text{ and } 
    \langle [\p] : \p \in \invertedprimes\rangle = \Cl(\O_K).
    \end{equation}
    Define $R_K$ as $\O_{K,\invertedprimes}$, the localization of $\O_K$ at the primes in $\invertedprimes$. In particular, \(R_K\) is a PID. 
\end{notation}

There is a natural inclusion of  $\Y'\big(\O_K\big[{1}/{2f}\big]\big)$ into $\Y'(R_K)$, so Lemma~\ref{lemma:Yprime-integer-points} implies that the image of $\Y(\Z)$ is contained in $\Y'(R_K)$. We now proceed to describe $\Y'(R_K)$ via descent.

Let \(\twist \in H^1(R_K, \mu_n)\), which has a representative in \(R_K^{\times}/(R_K^{\times})^n\) since \(R_K\) is a PID. Let \(\C_\twist'\subset \P(1,1,2)\) be the twist of $\C'$ defined by replacing $UV - CW$ in (\ref{eqn:construction-for-C'}) with \(UV - C\twist W\).
Then the twists \(\twistmap\colon \C_\twist'\to\Y'\) of the map \(\pi'\) are given by the maps:
\[
\left[U:V:W\right]\mapsto \left[\frac{U^n}{\twist C^{(n-1)/2}}:  \frac{V^n}{\twist^{n-1}C^{(n-1)/2}}:  W\right].
\]

\begin{lemma}
    Let $\twist,\tilde{\twist}\in H^1(R_{K},\mu_n)$.  Then $\twist$ and $\tilde\twist$ represent the same class if and only if $\twistmap\colon \C_\twist'\to \Y'$ and $\pi'_{\tilde{\twist}}\colon \C_{\tilde{\twist}}'\to \Y'$ are isomorphic twists. 
\end{lemma}

\begin{proof} 
The twists $\C_{\twist}'\colon UV = C\twist W$ and $\C_{\tilde{\twist}}'\colon \tilde{U}\tilde{V} = C\tilde{d}\tilde{W}$ are isomorphic as $\mu_n$-twists if and only there exists a map \(\phi\) making the following diagram commute:
\begin{center}
    \begin{tikzcd}
        \mathcal{C}_d' \arrow[rr, "\phi"] \arrow[dr, "\pi'_d", swap] & &\mathcal{C}'_{\tilde{d}} \arrow[dl, "\pi'_{\tilde{d}}"]\\
     & \mathcal{Y'}
    \end{tikzcd}
\end{center}
By explicitly writing down the coordinates, the map $\phi \colon [U:V:W] \mapsto [\tilde{U}:\tilde{V}:\tilde{W}]$ between $\C_{\twist}'$ and $\C_{\tilde{\twist}}'$ is an isomorphism of $\mu_n$-twists if and only if 
\[
    \left[\frac{U^n}{\twist C^{(n-1)/2}}:  \frac{V^n}{\twist^{n-1}C^{(n-1)/2}}:  W\right] = \left[\frac{\tilde{U}^n}{\tilde{\twist} C^{(n-1)/2}}:  \frac{\tilde{V}^n}{\tilde{\twist}^{n-1}C^{(n-1)/2}}:  \tilde{W}\right]
\] That is, there exists $\lambda\in R_K^{\times}$ such that
\[
    \frac{\tilde{d}}{d}\cdot \lambda^nU^n = \tilde{U}^n,\;\; \frac{\tilde{d}}{d}\cdot \lambda^nV^n = \tilde{V}^n,\;\; \lambda^2 W = \tilde{W}.
\] 
Hence, $\frac{\tilde{d}}{d} = t^n$ for some $t \in R_K^{\times}$. 
Conversely, if $\tilde{\twist} = t^n\twist$ for some $t \in R_K^\times$, then $[U : V : W] \mapsto [tU : t^{n-1}V : W]$ gives an isomorphism $\C'_\twist \to \C'_{\tilde{\twist}}$ defined over $R_K$.
\end{proof}

From the general theory of descent for \'{e}tale torsors, we have the following.

\begin{corollary}\label{cor:descent} With the notation above,
      \[
        \mathcal{Y}'(R_K) = \coprod_{\twist\in H^1(R_K, \mu_n)} \twistmap(\mathcal{C}'_\twist(R_K)).
    \] 
\end{corollary}

\begin{proof}
    See \cite[Lemma 2.4]{santens23}. 
\end{proof}

\subsection{Admissible twists}
\label{subsec:admissible-twists}

We next seek to sieve out the points of \(\Y'(R_K)\) that lie in the image of the map \(\YYprime \colon \Y(\Z) \to \Y'(R_K)\) and as such find the twists in \(H^1(R_K, \mu_n)\) that produce such points, as in Corollary~\ref{cor:descent}. In this subsection, we define a notion of admissible twists that serve this purpose for a given curve \(\Y_{B,C}\).

\begin{definition}\label{def:orders} 
Let $B = f^{2}B_{0}$, with $B_{0}$ squarefree, and suppose $\gcd(f,C) = 1$. Let $K = \Q\left (\sqrt{-B_0}\right )$ with ring of integers $\O_{K}$ and define an order $\O\subseteq\O_{K}$ as follows:  
\begin{equation*}
\O = \begin{cases}
    \Z\left [f\sqrt{-B_{0}}\right ], &\text{if } B_{0}\equiv 1,2\pmod{4};\\[1.5pt]
    \Z\left [f\sqrt{-B_{0}}\right ], &\text{if } B_{0}\equiv 3\pmod{8} \text{ and $C$ is odd};\\[2pt]
    \Z\left [f\frac{1 + \sqrt{-B_{0}}}{2}\right ], &\text{if } B_{0}\equiv 3\pmod{8} \text{ and $C$ is even; and}\\[5pt]
    \Z\left [f\frac{1 + \sqrt{-B_{0}}}{2}\right ], &\text{if } B_{0}\equiv 7\pmod{8}. 
\end{cases}
\end{equation*}
\end{definition}
Note that if \(B_0 \equiv 3 \pmod{8}\) and \(C\) is odd, then \(\O\) has conductor \(2f\); in all other cases, it has conductor \(f\).
Let $\invertedprimesCsplit \subset \invertedprimes$ denote the subset of inverted primes which both split in $K=\Q(\sqrt{-B})$ and divide $C$.  We define a set of \textit{admissible} twists of $\mathcal{C}' \to \Y'$ as follows.

\begin{definition}[$\admissible_{B,C}$]\label{def:admissible_twists}
Define $\admissible_{B,C} \subset R_K^\times / (R_K^\times)^n$ to consist of those $\twist$ satisfying
\begin{align}
\label{eq:admissible}    v_\p(d) &\equiv v_p(C)\frac{n\pm 1}{2}, \ v_{\pbar}(d) \equiv v_p(C)\frac{n\mp 1}{2}  \pmod{n} & \text{for all }\p \in \invertedprimesCsplit, \\
\nonumber   v_\p(d) &\equiv 0 \pmod{n} & \text{for all }\p \in \invertedprimes - \invertedprimesCsplit,
\end{align}
and $\div(\twist)\in\im\psi_{\O}$, where $\psi_\O$ is defined in (\ref{seq:def_psi}). 
\end{definition}

When $B_0 \equiv 3 \pmod{4}$, we will also employ a modified version of admissibility. Let $\p_2$ denote a prime above 2. 

\begin{definition}[$\widetilde{\admissible}_{B,C} $]\label{def:admissible_twists_tilde}
Define $\widetilde{\admissible}_{B,C} \subset R_K^\times / (R_K^\times)^n$ to consist of those $\twist$ satisfying
\begin{align}\label{eq:admissible_prime}
v_{\p_2}(d) & \equiv \pm \sizedparen{1 + \frac{n-1}{2}v_2(C)},\ v_{\pbar_2}(d) \equiv - v_{\p_2}(d) \pmod{n}, & \\ \nonumber
v_\p(d) &\equiv v_p(C)\frac{n\pm 1}{2}, \ v_{\pbar}(d) \equiv v_p(C)\frac{n\mp 1}{2} \!\!\! \pmod{n} &\hspace{-2.5cm} \text{for all }\p \in \invertedprimesCsplit \smallsetminus \sizedcurly{\p \mid 2\O_K}, \\\nonumber
v_\p(d) &\equiv 0 \pmod{n} &\hspace{-2.5cm}  \text{for all }\p \in \invertedprimes \smallsetminus (\invertedprimesCsplit \cup \sizedcurly{\p \mid 2\O_K}),
\end{align}
and $\div(\twist)\in\im\psi_{\O}$, where again $\psi_{\mathcal{O}}$ is defined in (\ref{seq:def_psi}). 
\end{definition}

Note that if 2 is ramified or inert in \(K\), then $\pbar_2 \equiv \p_2$, so $v_{\p_2}(d) = 0 \pmod{n}$ is enforced by the first condition. Thus actually in these cases we have $\admissible_{B,C} = \widetilde{\admissible}_{B,C}$.

\subsection{Divisibility conditions}
\label{subsec:div-lemmas}

We proceed to characterize $R_{K}$-points on $\C_{\twist}'$ for $\twist\in\admissible_{B,C}$ which determine $\Z$-points on $\Y_{B,C}$. In order to do so, we will define certain subgroupoids of \(\Y_{B,C}(\Z)\) that capture the conditions of Definitions~\ref{def:admissible_twists} and~\ref{def:admissible_twists_tilde}
and play a key role in Lemmas~\ref{lem:denominatortest_Bnot7mod8} and~\ref{lem:denominatortest_*prime}.

\begin{definition}[$\parallel$]
    Let $p$ be a rational prime and $u \in \O_K$. We write $p^r \parallel u$ to mean 
    \[u \in p^r\O_K \smallsetminus p^{r+1}\O_K.\]
    Notice that for $p$ inert (resp.\ $p = \p^2$ ramified), $p^r \parallel u$ is equivalent to $v_p(u) = r$ (resp.\ $v_\p(u) = 2r$ or $2r+1$). When $p = \p \pbar$ is split, $p^r \parallel u$ is equivalent to $\min(v_\p(u), v_{\pbar}(u)) = r$.
\end{definition}

\begin{definition}[$\annulus$]
\label{def:starp}
    Let $[x:y:z] \in \Y_{B,C}(\Z)$ and put $u = x + f\sqrt{-B_0}y$. For a prime number $p$, we say $[x:y:z]$ satisfies
    \begin{itemize}
        \item $\annulus_p^r$ if $p^{r} \parallel u$ for some nonnegative integer $r$. In particular, if $r = 0$, then $u \notin p\O_K$; 
        \item $\annulus^0$ if it satisfies $\annulus_p^0$ for all primes $p$.  
    \end{itemize} 
\end{definition}

If $p = \p\pbar$ is split, then $\annulus^0_p$ means that $u$ --- or its conjugate $\ubar$ --- lies in the $\p$-adic annulus $\O_{K,\p} \smallsetminus \pi_\p \O_{K,\p}$, where $\pi_\p$ is the uniformizer. 
Condition $\annulus^0$ will be key to characterizing points $\Y_{B,C}(\Z)$, while $*_p^r$ for $r > 0$ will be useful in Section~\ref{sec:cascade}.

For $p = 2$, we also need the following 
modification.

\begin{definition}[$\annulusTwo$]
\label{def:startildep}

    Let $[x:y:z] \in \Y_{B,C}(\Z)$ and $u = x + f\sqrt{-B_0}y$. We say $[x:y:z]$ satisfies
    \begin{itemize}
        \item ${\annulusTwo^{r}_2}$ if $2^{r+1} \parallel u$ and $2^{r} \parallel x$ for some \(r \ge 0\);
        \item $\annulusTwo^0$ if it satisfies $\annulusTwo^0_2$ and $\annulus^0_p$ for all odd primes $p$.
    \end{itemize}
\end{definition}

Recall the notation \(\Y_{B,C}(\Z; \bullet)\) from Remark~\ref{rem:groupoids}. By definition, for any \(r \ge 0\), \(\Y_{B,C}(\Z; \tilde{\annulus}^r_2) \subset \Y_{B,C}(\Z;{\annulus}^{r+1}_2)\).

\begin{lemma}
\label{lem:p_divides_u}
    Let $f,$ $B_0,$ and $C$ be nonzero integers with $B_0$ squarefree and $\gcd(f,C)=1$. Suppose $p$ is an odd prime. 
    If $[x:y:z] \in \Y_{B,C}(\Z;\annulus_p^r)$ for $r > 0$, then $p^r \mid x$ and exactly one of the following holds:
    \begin{enumerate}[label = (\alph*)]
        \item\label{part:lem-p-divides-u:C} $p^r \mid y$, $p^{2r} \mid C$, and $p \nmid z$;
        \item\label{part:lem-p-divides-u:f}  $p\nmid y$, $p^r \mid f$, and $p^{\lceil 2r/n \rceil} \mid z$.
    \end{enumerate}
    Moreover, in case~\ref{part:lem-p-divides-u:f}, if $n \nmid r$, then we have $p^r \parallel f$.
\end{lemma}

\begin{proof}
    By hypothesis, $u\in p^r\O_K$, so $\ubar \in p^r\O_K$ as well.
    Since $u + \ubar = 2x$ and $p$ is odd, we have $p^r \mid x$. To see that $p^r|fy$, we note that $u - \ubar = 2f\sqrt{-B_0}y\in p^r\O_K$.  If \(p \nmid B_0\), then $p^r \mid fy$.  If \(p \mid B_0\), then $p$ is ramified, $v_{\frakp}(\sqrt{-B_0})=1$, and $p\O_K=\frakp^2$ so that \(v_{\frakp}(u - \overline{u}) \ge 2r\) must be odd. In particular, $fy\O_K\in \frakp^{2r}$ and $p^r|fy$.

    Suppose that \(r > 0\) and that \(p \mid \gcd(f,y)\). Then $p \nmid C$, so we must have $p \mid z$. This contradicts $[x:y:z]$ being primitive, so we have either $p^r \mid y$ and $p\nmid f$ or $p^r \mid f$ and $p\nmid y$. 
    
    In case~\ref{part:lem-p-divides-u:C}, we have $p \nmid z$ is implied by primitivity of $[x:y:z]$. That $p^{2r} \mid C$ follows from the equation for $\Y_{B,C}$. 

    In case~\ref{part:lem-p-divides-u:f}, again we have $p^{2r} \mid Cz^n$, and by hypothesis that \(\gcd(f, C) = 1\) it must be $p^{2r} \mid z^n$. Since the latter is an $n$th power, we deduce that $p^{\lceil 2r/n \rceil} \mid z$.
    
    For the moreover statement, suppose that $n \nmid r$ and to the contrary that $p^{r+1} \mid f$. In this case, we use the fact that $\lceil 2r/n \rceil n > 2r$.  We can now examine the equation for $\Y_{B,C}$ satisfied by $[x:y:z]$, modulo $p^{2r+1}$:
    \[x^2 + f^2B_0 y^2 \equiv Cz^n \pmod{p^{2r+1}}.\]
    The right-hand-side vanishes by the observation we just made, and the $f^2B_0y^2$ term vanishes by hypothesis. This means $p^{2r+1} \mid x^2$, so $p^{r+1} \mid x$. However, this produces a contradiction, as now $p^{r+1} \mid u$.
\end{proof}
\begin{remark}
\label{rem:starp-fc}
 It follows from Lemma~\ref{lem:p_divides_u} that if $p\nmid 2fC$ and $\gcd(f,C)=1$ then \(\Y_{B,C}(\Z) = \Y_{B,C}(\Z; \annulus^0_p)\).  In particular, \(\Y_{B,C}(\Z) = \Y_{B,C}(\Z; \annulus^0_p \text{ for } p \nmid 2fC)\).
\end{remark}
The story for $p=2$ is different because $2^r \parallel u$ does not force $2^r \mid x,fy$.

\begin{lemma}
\label{lem:2_divides_u}
    Let $f,$ $B_0,$ and $C$ be nonzero integers with $B_0$ squarefree and $\gcd(f,C)=1$.  If $[x:y:z] \in \Y_{B,C}(\Z;\annulus_2^r)$ for $r > 0$, then $2^{r-1} \mid x$. Setting $s \colonequals \min\{v_2(x), r\}$, we have that exactly one of the following holds:
    \begin{enumerate}[label = (\alph*)]
        \item\label{part:s>02|y} $s > 0$, $2^s \mid y$, $2^{2s} \mid C$, and $2 \nmid z$;
        \item\label{part:s>02!|y} $s > 0$, $2\nmid y$, $2^s \mid f$, and $2^{\lceil 2s/n \rceil} \mid z$;
        \item\label{part:s=0} $s=0$ and $2 \nmid f,y$. 
    \end{enumerate}
    Moreover, in case~\ref{part:s>02!|y}, if $n \nmid s$,  then we have $2^s \parallel f$.
\end{lemma}

\begin{proof}   
    The proof is essentially the same as the one for Lemma~\ref{lem:p_divides_u}, except that $2^r \mid 2x$ implies only that $2^{r-1} \mid x$. If $2^{r-1} \parallel x$ then $s=r-1$; otherwise we have $s=r$. 
    
    First we show that \(2^s \mid fy\): if \(s=r-1\), then \(2^s \mid fy\). Suppose \(s=r\), so \(2^r \mid x\). If \(2 \nmid B_0\), then for \(\p\) above \(2\), we have $v_{\p}(u) \geq \min\{v_{\p}(x), v_{\p}(fy)\}$. Moreover, by assumption, we have $v_{\p_2}(u) \geq r$, so we must have $v_{\p_2}(fy) \geq v_{\p_2}(x) \geq r$. Hence, $2^r\mid fy$. If \(2 \mid B_0\), then \(2 \O_K = \p^2\)  and \(v_{\p}(u - \ol{u}) \ge 2r\) is odd, so that \(v_{\p}(fy) \ge 2r\). Thus \(2^s \mid fy\).
    
    Now, if \(2 \mid f\), then \(2 \nmid C\). Therefore \(2 \mid z\) and \(2 \nmid y\). So the statement for \(f\) and \(z\) in~\ref{part:s>02!|y} is true. If \(2 \mid y\), then \(2 \nmid z\), so \(2 \mid C\) and \(2\nmid f\). So we have \(2^s \mid y\) and \(2^{2s} \mid C\) as in~\ref{part:s>02|y}. 
    
    For~\ref{part:s=0}, when \(s=0\), we must have \(r=1\). In particular, \(2 \nmid x\) and \(2 \mid u\). For any prime \(\p\) above 2, \(v_{\p}(\sqrt{-B_0}fy) \ge \min \{v_{\p}(u), v_{\p}(x)\}\). Therefore,  \(v_{\p}(\sqrt{-B_0}fy) = 0\) and so \(2 \nmid fy\). (In particular, if \(2 \mid B_0\), then \(\Y_{B,C}(\Z;\annulusTwo^0_2)  =\emptyset\).)
    
    The proof for the moreover statement is exactly the same as in Lemma~\ref{lem:p_divides_u}.
\end{proof}

\begin{lemma}\label{lem:fCnotcoprime-no-star0-points}
    Let \(p, f,B_0,C\) be nonzero integers such that $p$ is prime, \(B_0\) is squarefree, \(p\mid \gcd(f,C)\), and \(p \parallel C\). Then \(\Y_{B,C}(\Z; \annulus^0_p) = \emptyset\) and if \(p=2\), then \(\Y_{B,C}(\Z; \annulusTwo^0_2) = \emptyset\). Further, if \([x:y:z] \in \Y_{B,C}(\Z; \annulus_p^r)\) for some \(r >0\), then the following holds.
    \begin{enumerate}[label = {(\alph*)}]
        \item \label{part:fCgcd-odd-divisibility}  If \(p\) is odd, then \(p^r \mid x\), \(p \nmid y\), \(p^r \mid f\), and \(p^{\lceil(2r-1)/n\rceil} \mid z\); in addition, if \(n \nmid 2r-1\), then \(p^r \parallel f\). If \(n \mid (2r-1)\) then either \(p^r \parallel f\) or \(p^{(2r-1)/n} \parallel z\).
        \item \label{part:fCgcd-even-divisibility} If \(p=2\), let \(s = \min\{v_2(x), r\}\). Then \(s \ge r-1\) and \(2^{s} \mid x\), \(2^{s} \mid fy\), and \(2^{\lceil(2s-1)/n\rceil} \mid z\). If \(s>0\), then \(2 \mid z\), so that \(2 \nmid y\) and \(2^s \mid f\). As above, if \(n \nmid 2s-1\), then \(p^s \parallel f\) and if \(n \mid (2s-1)\), then either \(p^s \parallel f\) or \(p^{(2s-1)/n} \parallel z\).       
    \end{enumerate}
\end{lemma}

\begin{proof}
 If \([x:y:z] \in \Y_{B,C}(\Z; \annulus^0_p)\), then assumption on \(f, C\) forces \(p \mid x\), which forces \(u = x + f\sqrt{-B_0}y \in p \orb_K\). So \(\Y_{B,C}(\Z; \annulus^0_p) = \emptyset\) and similarly, so is \(\Y_{B,C}(\Z; \annulusTwo^0_2)\). The rest follows by the same arguments as the proofs of Lemmas~\ref{lem:p_divides_u} and~\ref{lem:2_divides_u} (the relevant parts do not use \(\gcd(f,C)=1\)). To see the why the conditions on \(n\) dividing \(2r-1\) appear, observe that in part~\ref{part:fCgcd-odd-divisibility}, the \(n \nmid 2r-1\) assumption guarantees that \(\lceil \frac{2r-1}{n}\rceil n > 2r-1 \) so that \(Cz^n\) is divisible by \(p^{2r+1}\). If \(p^{2r+1} \mid Cz^n\), then we may use the same argument to conclude that \(p^r \parallel f\). Thus, in this case, if \(p^{r+1} \mid f\), then we must have that \(p^{2r} \parallel Cz^n\). The conditions on \(n\) dividing \(2s-1\) appear in part~\ref{part:fCgcd-even-divisibility} for a similar reason.
\end{proof}

\begin{remark}\label{rem:star0and3mod4}
    Suppose $B_0 \equiv 3 \pmod{8}$. Then by Definitions~\ref{def:starp}~and~\ref{def:startildep}, we have that
    \begin{align}
        \Y_{B,C}(\Z; \annulus^0) \neq \emptyset &\implies C \text{ odd},\\
        \Y_{B,C}(\Z; \annulusTwo^0) \neq \emptyset &\implies 4 \parallel C. 
    \end{align}
    When $B_0 \equiv 7 \pmod{8}$ and $C$ is odd, it is possible for $\Y_{B,C}(\Z)$ to have both points satisfying $\annulus^0$ and $\annulusTwo^0$, with the former having odd $z$-coordinate and the latter having even $z$-coordinate.
\end{remark}

\subsection{Detecting integral points}
\label{subsec:denominator-test}

In this subsection, we show that $\annulus^0$ points on $\Y_{B,C}$ are detected by admissible twists, i.e. are contained in the refined descent obstruction from \S\ref{subsec:intro-descent}.

\begin{lemma}
\label{lem:denominatortest_Bnot7mod8}
    Suppose $B=f^2B_0$ for $B_0 \ne -1$ squarefree and $\gcd(f,C) = 1$. Then under the map $\Y_{B,C}(\Z) \to \Y'(R_K)$ we have an inclusion of groupoids
    \[\Y_{B,C}(\Z; \annulus^0) \subseteq \coprod_{\twist \in A_{B,C}} \pi_\twist'(\C_\twist'(R_K)).\]
    In particular, if $A_{B,C} = \emptyset$ then $\Y_{B,C}(\Z;\annulus^0) = \emptyset$.
\end{lemma}

\begin{proof}
    Let $[x:y:z] \in \Y_{B,C}(\Z; \annulus^0)$. Let $[u:v:w]$ be its image in $\Y'(R_K)$ and note that $u = \bar{v}$. By Corollary~\ref{cor:descent},
    we have $[u:v:w] = \pi'_\twist([U:V:W])$ for some $\twist \in R_K^\times/(R_K^\times)^n$ and $[U:V:W] \in \C'_\twist(R_K)$. Thus for any $\lambda\in R_{K}^{\times}$,
    \[[u:v:w] = \left[ \frac{U^n}{\twist C^{\frac{n-1}{2}}} : \frac{V^n}{\twist^{n-1} C^{\frac{n-1}{2}}} : W\right] = \left[ \frac{\lambda^nU^n}{\twist C^{\frac{n-1}{2}}} : \frac{\lambda^nV^n}{\twist^{n-1}C^{\frac{n-1}{2}}} : \lambda^2 W\right].\] 
    Taking $\p$-adic valuations for primes $\p \in \invertedprimes$ and reducing modulo $n$, we have
    \begin{align}
    \label{eq:val_res_u}    v_\p(u) &\equiv -v_\p(\twist) - \frac{n-1}{2}v_\p(C) \pmod{n}, \\
    \label{eq:val_res_v}    v_\p(v) &\equiv v_\p(\twist) - \frac{n-1}{2}v_\p(C) \pmod{n},
    \end{align}
    which are well-defined since $\twist \in R_K^\times/(R_K^\times)^n$. We will show that these imply $\twist \in \admissible_{B,C}$.

    First, we consider $\p \in \invertedprimes$ ramified. Since $u = \bar{v}$ we have $v_\p(u) = v_\p(v)$. Subtracting~\eqref{eq:val_res_u} from~\eqref{eq:val_res_v} we have $2 v_\p(\twist) \equiv 0$, which implies $v_\p(\twist) \equiv 0 \pmod{n}$. 
    Similarly, if $\p = p\O_K$ is inert, we have $v_\p(u) = v_\p(v)$ and deduce that $v_\p(d) \equiv 0 \pmod{n}$. 

    Assume now that $\p \in \invertedprimessplit$, i.e.\ there exists a prime $\pbar\ne\frakp$ with $\p\pbar = p\O_K$. 
    Condition~$\annulus^0$ implies that $\min\{v_\p(u),v_{\pbar}(u)\}\equiv 0\pmod n$.  Hence, $\min\{v_\p(u),v_{\p}(v)\}\equiv 0\pmod n$.  Up to (possibly) swapping $\p$ and $\pbar$, we may assume $v_\p(v) \equiv 0\pmod n$. 

    If $p \nmid C$ then we have $v_\p(\twist) \equiv 0 \pmod n$. Otherwise we have $\p \in \invertedprimesCsplit$, so  
    $$v_\p(\twist) \equiv \frac{n-1}{2}v_\p(C) \equiv \frac{n-1}{2}v_p(C) \pmod{n};$$
    $$v_{\pbar}(\twist) \equiv -\frac{n-1}{2}v_{\pbar}(C) \equiv \frac{n+1}{2}v_p(C) \pmod{n}.$$ Swapping the roles of $\p, \pbar$ merely swaps the sign.  
    This shows $[x : y : z]\in \pi_{\twist}'(\C'(R_{K}))$ for some $\twist$ satisfying the conditions in~\eqref{eq:admissible}.

    To show $\div(d) \in\im(\psi_{\O})$, there is a slight difference depending on the residue of $B_{0}$ modulo ${8}$ and the parity of $C$. Here $\psi_{\mathcal{O}}$ is defined in (\ref{seq:def_psi}).  First, assume $B_{0} \not\equiv 3\pmod{8}$, so that the order $\O \subset \O_K$ defined in Definition~\ref{def:orders}
    has conductor $\f = f$. For $p \mid f$, the condition $*_p^0$ implies $u$ is coprime to $f$, so $u\O_K = \psi_\O(u\O)$ by the definition of $P(\O,f)$.   
    To see $\twist\O_K \in \im (\psi_\O)$, we observe that \([u:v:w]\in \im(\pi_\twist')\) implies
    \[
    u\twist = \frac{U^n}{C^{(n-1)/2}}
    \]
    for some \(U \in K^{\times}\).
    Since $\gcd(f,C) = 1$ and $\frac{n-1}{2}$ is coprime to $n$, we have
    \[d \O_K\equiv \psi_{\O}((C^{\frac{n-1}{2}}u\mathcal{O})^{-1}) \pmod{nP(\O_K)},\] 
    hence $\twist \in \admissible_{B,C}$.

   If $B_0 \equiv 3 \pmod{8}$, then by Remark~\ref{rem:star0and3mod4} the conclusion is vacuously satisfied for $C$ even. When $C$ is odd, $\O \subset \O_K$ has conductor $\f = 2f$. By the same argument as above, $u$ is coprime to $f$. Since $x,y$ are not both even by the $*^0$ condition, the congruence $x^2 + f^2B_0 y^2 \equiv 0 \pmod{8}$ has no solutions. Thus, $z$ is odd and $u$ is coprime to $2f$.

    This completes the inclusion of $\Y_{B,C}(\Z;\annulus^0)$  
    into $\coprod_{\twist \in \admissible_{B,C}} \pi'_\twist(\C'_\twist(R_K))$.   
    By hypothesis, $-B$ is not a square, so by Remark~\ref{rem:groupoids}, these groupoids are sets, so there is nothing extra to check.  
\end{proof}

Similarly, recall the definition of the subset $\widetilde{\admissible}_{B,C}$ from Definition~\ref{def:admissible_twists_tilde}.

\begin{lemma}
\label{lem:denominatortest_*prime}
    Suppose $B=f^2B_0$ for $B_0 \neq -1$ squarefree, $B_0 \equiv 3 \pmod{4}$, and $\gcd(f,C) = 1$. Then under the map $\Y_{B,C}(\Z) \to \Y'(R_K)$ we have an inclusion of groupoids
    \[\Y_{B,C}(\Z; \annulusTwo^0) \subset \coprod_{\twist \in \widetilde{\admissible}_{B,C}} \pi_\twist'(\C'(R_K)).\]
    In particular, if $\widetilde{\admissible}_{B,C} = \emptyset$ then $\Y_{B,C}(\Z;\annulusTwo^0) = \emptyset$.
\end{lemma}

\begin{proof}
    The proof is nearly identical to that of Lemma~\ref{lem:denominatortest_Bnot7mod8}, with nothing changing at the odd primes. When $B_0 \equiv 3 \pmod{8}$, by Remark~\ref{rem:star0and3mod4} the conclusion is vacuously satisfied unless $4 \parallel C$. Since 2 is inert we have $v_{2}(d) \equiv 0 \pmod{n}$, showing that $d$ satisfies~\eqref{eq:admissible_prime}. Since in this case $\O$ has conductor $\f = f$, necessarily odd by the hypothesis $\gcd(f,C) = 1$, we need not check that $u$ is coprime to 2, and the verification that the image of $d$ lies in that of $\psi_\O$ is the same as in the proof of Lemma~\ref{lem:denominatortest_Bnot7mod8}.

    Suppose now that $B_0 \equiv 7 \pmod{8}$. At the split prime 2, since our points satisfy $\annulusTwo^0_2$, we instead find that that $v_{\p_2}(u)$ or $v_{\p_2}(v)$ is congruent to 1 mod $n$. Assuming $v_{\p_2}(v) \equiv 1 \pmod n$, we have
    \[v_{\p_2}(\twist) \equiv 1 + \frac{n-1}{2} v_{\p_2}(C) \pmod{n}.\]
    Conjugating gives the value of $v_{\pbar_2}(\twist)$, revealing that $\twist$ satisfies the conditions in~\eqref{eq:admissible_prime}.
    
    The proof that $\div(\twist)\in\im(\psi_{\O})$ is similar to Lemma~\ref{lem:denominatortest_Bnot7mod8}; if $f$ is even, then $\Y_{B,C}(\Z; \annulusTwo^0) = \emptyset$ (see Lemma~\ref{lem:2_divides_u}\ref{part:s=0}), so there is nothing to check.
\end{proof}

\section{Main theorems}\label{sec:mainthm}
In the next two sections we prove the main results about the Hasse principle for integral points on the stacky curves $\Y_{B,C}$ (Theorems~\ref{thm:TFAE-intro} and \ref{thm:MainCascadeTheorem}). Proposition~\ref{prop:localTest} already characterizes the values of $B$ and $C$ for which $\Y_{B,C}$ is everywhere locally soluble, so it remains to describe the existence of global points on these stacky curves. To do this, we first characterize when $\Y_{B,C}(\Z; \annulus^0,\ \gcd(z,M) = 1)$ or $\Y_{B,C}(\Z; \annulusTwo^0, \gcd(z,M) = 1)$ are non-empty for $B=f^2B_0$ with $B_0 \ne -1$ squarefree, $\gcd(f,C) = 1$, and $M$ an arbitrary integer.  In Section~\ref{sec:cascade}, we will show how to reduce to one of these cases. 

For many cases, \cite[Prop.~8.1]{DarmonGranville} is a good template for detecting local and global points on $\Y_{B,C}$. However, it fails when certain hypotheses are removed, as the next example shows. 

\begin{ex}
\label{ex:relorderfactors}
\cite[Prop.~8.1]{DarmonGranville} is false when $B_{0}\not\equiv 1\pmod{4}$.
Explicitly, if there is some~$C$ which factors as $C\O_K = \J_{+}\J_{-}$ with $[\J_{+}]\in 3\Cl(\O_{K})$ but $[\J_{+}\cap\O]\not\in 3\Cl(\O)$, then \cite[Prop.~8.1]{DarmonGranville} would predict that $x^{2} + B_{0}y^{2} = Cz^{3}$ has primitive $\Z$-solutions, but Theorem~\ref{thm:TFAE-intro} correctly predicts that there are no such solutions. 
For example, consider the equation $x^2+339y^2 = 29z^3$. Let $K \coloneqq \Q(\sqrt{-339})$, so Definition~\ref{def:orders} defines the order $\O = \Z[\sqrt{-339}]$.  A Magma computation shows that $\Cl(\O_{K}) \cong \Z/6\Z$ and $\Cl(\O) \cong \Z/3\Z\times\Z/6\Z$.
The ideal $[\J_{+}]$ has order $2$ in $\Cl(\O_K)$ and is therefore in $3\Cl(\O_K)$. Note that $[\J_{+}\cap\O]$ has order $6$ in $\Cl(\O)$, but no element of order $6$ is in $3\Cl(\O)$.
\end{ex}

Of course, one cannot expect the conditions from \cite{DarmonGranville} to work for general \(B\), since \(x^2 + By^2\) is not necessarily the norm form of the maximal order. Thus, in this section, we use the order \(\O\) in $\O_{K}$ as in Definition~\ref{def:orders}. When \(B_0 \equiv 1, 2 \pmod{4}\), this order has the desired norm form, which allows us to obtain a criterion similar to that of \cite{DarmonGranville}. For an admissible \(\twist\), we show \(\div(\twist\O_K)\) can be decomposed into a linear combination of coprime ideals \(\J_{+}\) and \(\J_{-}\) satisfying the conditions that the classes $[d\O_K]$ and \([\J_{+}]\) differ by an \(n\)th power in \(\Cl(\O_K)\) and that \([\J_{+} \cap \O]\) belongs to a prescribed coset of \(n \Cl(\O)\). Furthermore, \(C\O_K\) factors as $\J_+\J_- \r^2$ where \(\J_{\pm}\) are supported at primes split in \(\O_K\) and \(\r\) is supported at ramified primes. The condition that \([\J_{+} \cap \O]\) belongs to a prescribed coset of \(n \Cl(\O)\) produces a candidate \(u = x + \sqrt{-B}y\) such that \(x,y \in \Z\) and \(\Nm(u) = Cz^n\) with \([x:y:z] \in \Y_{B,C}(\Z; \gcd(z,M)=1).\) One then checks that $u$ can be chosen such that $[x:y:z]$ satisfies $\annulus^0$. When $B_0 \equiv 3 \pmod{4}$, extra complications arise due to the behavior of the prime 2. 

\subsection{Base cases}

In what follows, $K = \Q\left (\sqrt{-B_{0}}\right )$ with ring of integers $\O_{K}$. For any fractional ideal \(\fraka\), \(\overline{\fraka}\) will denote its Galois conjugate. 

\begin{theorem}
\label{thm:TFAEunified}
Suppose $B = f^{2}B_{0}$ with $B_{0} \neq -1$ squarefree and $\gcd(f,C) = 1$. If $B_0 \equiv 3 \pmod{4}$,  
additionally assume $C$ is odd. Let \(\O\) be the order defined in Definition~\ref{def:orders}. Then the following are equivalent: 
\begin{enumerate}[label = (\roman*)]
    \item\label{part:TFAEi} for all $M\in\Z$, $\Y_{B,C}(\Z;\annulus^{0},\gcd(z,M) = 1)\not = \emptyset$;
    \item\label{part:TFAEii} $p \nmid C$ for all inert primes $p$, $v_p(C) \leq 1$ for all ramified primes $p$, and $\admissible_{B,C}\not = \emptyset$; 
    \item\label{part:TFAEiii} $C\O_{K} = \J_{+}\J_{-}\r^{2}$ as ideals in $\O_{K}$, with $\J_{\pm}$ coprime and supported on split primes, satisfying $\J_{+} = \overline{\J_{-}}$ and $[\J_{\pm}\cap\O]\in n\Cl(\O)$, and with $\r$ a product of ramified primes. 
\end{enumerate}
If $B_0 \equiv 3 \pmod{4}$ and $C$ is even, we have $\Y_{B,C}(\Z;\annulus^{0}) = \emptyset$.
\end{theorem}

\begin{proof}
   \ref{part:TFAEi} $\implies$~\ref{part:TFAEii} It suffices to take $M=1$. The statement $\admissible_{B,C}\not = \emptyset$ follows from Lemma~\ref{lem:denominatortest_Bnot7mod8}. Let $[x:y:z] \in \Y_{B,C}(\Z; \annulus^0)$. Suppose $p \mid C$ for an inert prime $p$, necessarily coprime to $f$ and necessarily odd by our assumptions. Then $x^2 + f^2B_0y^2 \equiv 0 \pmod{p}$, but for $p$ inert this congruence has only the trivial solution; this forces $p \mid x,y$, contradicting the $*^0$ condition. Next, suppose that $p^2 \mid C$ for a ramified prime $p$. If $p \mid B_0$, then we have $p \mid x$, which implies $p \mid fy$; since $p \nmid f$, we have $p \mid y$, again violating $*^0$. When $B_0 \equiv 1 \pmod{4}$, the prime $2$ is also ramified; in this case, examining $x^2 + f^2 B_0 y^2 \equiv 0 \pmod{4}$ reveals that $2 \mid x,y$, again contradicting $*^0$.
       
   \ref{part:TFAEii} $\implies$~\ref{part:TFAEiii} We construct the ideals $\J_\pm$ and $\r$ as follows: since $\twist \in \admissible_{B,C}$, for all primes $\p \in \invertedprimesCsplit$ we have $v_\p(\twist) \equiv \frac{n\pm 1}{2}v_\p(C) \pmod{n}$. Let 
    \begin{align*}
        &\J_+ \colonequals \prod_{\substack{\p \in \invertedprimesCsplit\\v_\p(\twist) \equiv \frac{n + 1}{2}v_\p(C) \ (n)}}\p^{v_\p(C)}&\text{and }\hspace{0.8cm}
        &\J_- \colonequals \prod_{\substack{\p \in \invertedprimesCsplit\\v_\p(\twist) \equiv \frac{n - 1}{2}v_\p(C) \ (n)}}\p^{v_\p(C)}.
    \end{align*} 
    We then define 
    \[
        \r = \prod_{\substack{\p \text{ ramified in }K \\ \p \mid C}} \p. 
    \] 
    By hypothesis, we have $C\O_K = \J_+\J_-\r^2$, 
    with $\J_+$ coprime and conjugate to $\J_-$, and satisfying the desired support conditions.

    It remains to show $[\J_\pm \cap \O] \in n \Cl(\O)$. By construction, we have $\div(d\O_K) = \frac{n+1}{2} \div(\J_+) + \frac{n-1}{2} \J_- \pmod{n}$.
    Let $\widetilde{d}\O \in P(\O, f)/nP(\O,f)$ satisfy $\psi_{\mathcal{O}}(\widetilde{d}\O) = d\O_{K}$, where $\psi_{\mathcal{O}}$ is as in (\ref{seq:def_psi}). Then the commutativity of~\eqref{seq:thebigdiagram} and the injectivity of the map $I(\O,f)/nI(\O,f) \to I(\O_K)/nI(\O_K)$ reveal that $\div(\widetilde{d}\O) = \frac{n+1}{2} (\J_+ \cap \O) + \frac{n-1}{2} (\J_- \cap \O)$. By taking into account that $[\J_+] = -[\J_-]$, we find that $\div(\widetilde{\twist}\O) = [\J_+ \cap \O]$, so by the exactness of the top row of~\eqref{seq:thebigdiagram}, we have $[\J_+ \cap \O] \in n \Cl(\O)$.

   \ref{part:TFAEiii} $\implies$~\ref{part:TFAEi} Write $C\O_K = \J_+\J_-\r^2$ as in the hypothesis. 
    The argument is slightly different depending on $B_{0}$ modulo ${8}$ and the parity of $C$, so first assume 
    $B_0 \not \equiv 3 \pmod{4}$. Since $\gcd(f, C) = 1$, we necessarily have $\J_\pm \cap \O, \r \cap \O \in I(\O,\f)$. 
    Moreover, $[\r \cap \O] \in \Cl(\O)[2]$, so since $n$ is odd, $[\r \cap \O] \in n \Cl(\O)$.
    Thus $[\J_+\r \cap \O] \in n \Cl(\O)$.

    Let $\z$ be an $\O$-ideal such that $(\J_+\r \cap \O)\z^n$ is principal, say $(\J_+\r \cap \O)\z^n = u\O$ for some $u = x + f\sqrt{-B_0}y \in \O$. By Lemma~\ref{lem:ClO_primetof}, we may further assume that $\z$ is coprime to $2fB_0CM$. By Lemma~\ref{lem:IOKf=IOf}, the norms of $\z\O_K$ and $\z\O$ coincide for \(\z \in \O\) relatively prime to  \(\mathfrak{f}\), so setting $z = \Nm(\z)$ we have
    \[\Nm(u) = x^2 + f^2B_0y^2 = \Nm(\J_+\r\z^n) = Cz^n\]
    for integers $x,y,z$ with $\gcd(z,2fB_0CM)=1$. 

    It remains to show that condition $*^0$ is satisfied for some such $u$; this implies primitivity of $[x:y:z]$. We claim that if $u \in p\O_K$ for a rational prime $p$, then $\z \in p\O_K$. We can then replace $u$ and $\z$ by $u/p^n$ and $\z/p$, respectively, until $u \notin p\O_K$ for any rational prime $p$. Thus we have $[x:y:z]$ is primitive and satisfies $*^0$.

    If $u\in p\O_K$ and $p \nmid C$, then we already have $p \mid \z$, so suppose $p \mid C$. If $p = \p\pbar$ splits in $K$, then since $\J_+$ and $\J_- = \overline{\J_+}$ are coprime by hypothesis, and both \(\p, \pbar \mid u\), we must have one of $\p$ or $\pbar$ dividing $\z$. However, we chose $\z$ above to be coprime to $C$, a contradiction. If $p = \p^2$ is ramified in $K$, then again we find that $\p \mid \z$, contradicting our choice of $\z$ above. Therefore in all cases, if $u \in p\O_K$ we have $\z \in p\O_K$, as desired.

    We now address $B_0 \equiv 3 \pmod{4}$, provided $C$ is odd. If $B_0 \equiv 3 \pmod{8}$, then $\O \subset \O_K$ has conductor $\f =2f$. Thus $\J_\pm \cap \O, \r \cap \O \in I(\O, \f)$ and the same argument as above applies to produce a point $[x:y:z] \in \Y_{B,C}(\Z; \annulus^0,\ \gcd(z,M) = 1)$.

    If $B_{0}\equiv 7\pmod{8}$,  
    let $\O' = \Z[f\sqrt{-B_0}]$ be the order in $\O$ of relative conductor $2$ (absolute conductor $2f$ in $\O_K$). By Lemma~\ref{lem:class_number_formula}, we have $\#\Cl(\O')/\#\Cl(\O) \in \{1,2\}$, so in either case,  $n\Cl(\O') \simeq n\Cl(\O)$, induced by the isomorphism $I(\O', 2f) \to I(\O, 2f)$ on ideals coprime to $2f$.
    Thus $[\J_+ \r \cap \O'] \in n\Cl(\O')$, so we find an ideal $\z$ in $\O'$ coprime to $2fB_0CM$ such that 
    \[(\J_+ \r \cap \O')\z^n = u\O'.\]
    Taking norms produces
    \[Cz^n = x^2 + f^2B_0y^2\]
    with $\gcd(z,M) = 1$. The choices of $\J_\pm$ and $\z$ ensure that $[x:y:z]$ is primitive, via the same 
    argument, so once again $\Y_{B,C}(\Z;\annulus^{0},\gcd(z,M) = 1)$ is nonempty. This completes the proof of~\ref{part:TFAEiii} $\implies$~\ref{part:TFAEi} in all cases. 

    Finally, if $B_0 \equiv 3 \pmod{4}$ and $C$ is even, then $f$ is odd. Thus for any $[x:y:z] \in \Y_{B,C}(\Z)$, $x$ and $y$ must have the same parity. In either case they do not satisfy $\annulus_2^0$, so we have $\Y_{B,C}(\Z; \annulus_2^0) = \emptyset$; see Remark~\ref{rem:star0and3mod4}. 
\end{proof}

Next, 
define a subset of integers $\Z_{B,C}$ according to 
$$
\Z_{B,C} = \begin{cases}
    \Z, &\text{if } B_{0}\not\equiv 7\pmod{8};\\
    \Z, &\text{if } B_{0}\equiv 7\pmod{8} \text{ and } 2\nmid C \text{ or } 8\mid C;\\
    2\Z + 1, &\text{if } B_{0}\equiv 7\pmod{8} \text{ and } 2\mid C \text{ but } 8\nmid C. 
\end{cases}
$$

\begin{theorem}
\label{thm:TFAEunified_prime}
Suppose $B = f^{2}B_{0}$ with $-1 \neq B_0 \equiv 3 \pmod{4}$ squarefree and  $\gcd(f,C) = 1$. Let \(\O\) be the order defined in Definition~\ref{def:orders}. Then the following are equivalent: 
\begin{enumerate}[label = $\widetilde{(\roman*)}$]
    \item\label{part:TFAE2i} for all $M\in\Z_{B,C}$, $\Y_{B,C}(\Z;\annulusTwo^{0},\ \gcd(z,M) = 1)\not = \emptyset$;
    \item\label{part:TFAE2ii} $p \nmid C$ for all odd inert primes $p$, $v_{p}(C) \leq 1$ for all ramified primes $p$, $4\parallel C$ if $B_0\equiv 3\pmod{8}$, and $\widetilde{\admissible}_{B,C}\not = \emptyset$;
    \item\label{part:TFAE2iii} $C\O_{K} = 2^{\ell}\J_{+}\J_{-}\r^{2}$ as ideals in $\O_{K}$, with $\ell = 2$ when $B_0\equiv 3\pmod{8}$,  $\J_{\pm}$ coprime to $2$ and each other and supported on split primes, satisfying $\J_{+} = \overline{\J_{-}}$ and $[\J_{+}\cap\O]\in \pm[\p_{2}^{2 - \ell}] + n\Cl(\O)$, and with $\r$ a product of ramified primes. 
\end{enumerate}
Moreover, if $B_0 \equiv 3\mod{8}$ and $v_{2}(C)\not = 2$, then $\Y_{B,C}(\Z;\annulusTwo^0) = \emptyset$ (recall Remark~\ref{rem:star0and3mod4}).
\end{theorem}

\begin{proof}
    First, note that if $C$ is odd, then any $[x : y : z]\in\Y_{B,C}(\Z;\annulusTwo^{0})$ must have $z$ even, since $2 \parallel u$ implies $2\mid Cz^{n}$. Thus when $C$ is odd, it is necessary to take $M$ odd in~\ref{part:TFAE2i}. 

    Similar to the proof of Theorem~\ref{thm:TFAEunified},~\ref{part:TFAE2i} $\implies$~\ref{part:TFAE2ii} follows from Lemma~\ref{lem:denominatortest_*prime} and Remark~\ref{rem:star0and3mod4}.

   \ref{part:TFAE2ii} $\implies$~\ref{part:TFAE2iii} When $\twist\in\widetilde{\admissible}_{B,C}$, tracing through~\eqref{eq:admissible_prime} reveals 
    \[0 \equiv [\twist\O_K] \equiv [\J_+] \pm [\p_2^{2-\ell}] \pmod{n},\]
    where $\ell = v_{2}(C)$ and where swapping $\p_2$ and $\pbar_2$ merely
    changes the sign. Since both $\J_+$ and $\p_2$ are coprime to $f$ (for the latter, see Lemma~\ref{lem:2_divides_u}\ref{part:s=0}), the conductor of $\O$, and $d\O_K \in \im\psi_\O$, where $\psi_{\mathcal{O}}$ is as in (\ref{seq:def_psi}), we obtain 
    \[[\J_+ \cap \O] \in \pm[\p_{2}^{2-\ell} \cap \O] + n \Cl(\O).\]

   \ref{part:TFAE2iii} $\implies$~\ref{part:TFAE2i}  We begin with $B_0 \equiv 7 \pmod{8}$. First, suppose $C$ is odd. Without loss of generality, suppose that $[\J_+ \cap \O] \in [\p_2^2 \cap \O] + n \Cl(\O)$. Here we find $\z$ coprime to $2BCM$ such that $(\p_2^{n-2} \J_+ \r \cap \O) \z^n = u\O$ for $u \in \O$.
   Taking norms, we obtain
    \[2^{n-2}Cz^n = \Nm(u),\]
    for $u = a + bf\frac{1 + \sqrt{-B_0}}{2}$ with \(a,b \in \Z\). Note that $u \in \p_2\O_K \smallsetminus 2\O_K$, from which it follows that $b$ is odd. 
    Replacing $u$ by $2u$ and rearranging, we have
    \[C(2z)^n = \Nm(2u) = x^2 + f^2B_0y^2,\]
    where $x = 2a + bf$ and $y = b$. Since $b$ is odd, this ensures that $[x:y:z]$ satisfies $\tilde{*}_2^0$. The same arguments ensure it is primitive and satisfies $*^0_p$ at odd primes $p$.

    When $C$ is even and $M$ is odd, we can follow a similar argument to construct points $[x : y : z]\in\Y_{B,C}(\Z;\annulusTwo^0,\gcd(z,M) = 1)$ with \emph{even} $z$-coordinate. The only difference here is that we find $\z$ coprime to $2BCM$ satisfying $(\p_{2}^{n + \ell - 2}\J_{+}\r\cap\O)\z^{n} = u\O$. Taking norms, multiplying by $4$ and moving a factor of $2^{n}$ into $\Nm(\z)^n$ produces the desired point with even $z$-coordinate. 

    When $v_2(C) \geq 3$ and $M$ is even, we also need to be able to construct points with \textit{odd} $z$-coordinate. Starting with 
    \[(\p_2^{\ell-2}\J_+\r \cap \O)\z^n = u\O = \left(a + bf\frac{1+\sqrt{-B_0}}{2}\right)\O,\]
    we recognize that $b$ is necessarily odd; either $a$ is even, or $a,b$ have the same parity, but not both, since the left-hand-side is divisible by only $\p_2$ and not $\pbar_2$. Note that here we are using that $\ell > 2$. Multiplying by 2 at this stage, we have 
    \[(\p_2^{\ell-1}\pbar_2\J_+\r \cap \O)\z^n = 2u\O = \left((2a + bf) + bf\sqrt{-B_0}\right)\O,\]
    and taking norms yields
    \[Cz^n = (2a+bf)^2 + (f^2B_0)b^2 = x^2 + (f^2B_0)y^2.\]
    Since $b$ {and \(f\) are} odd, $2 \nmid \gcd(x,y)$, and the justification that $p \nmid \gcd(x,y)$ for odd primes $p$ follows the same argument as in the proof of Theorem~\ref{thm:TFAEunified}.  

    Lastly, consider $B_0 \equiv 3 \pmod{8}$. Since 2 is inert, $[\J_+ \cap \O] \in n \Cl(\O)$, so we find $\z$ such that $(\J_+\r \cap \O)\z^n = u\O$ for $u \in \O$ and $\z$ coprime to $2BCM$. Taking norms we have 
    \[\frac{C}{4}z^n = \Nm(u).\]
    Multiplying both sides by 4, we obtain $[x:y:z] \in \Y_{B,C}(\Z; \annulusTwo^0,\ \gcd(z,M) = 1)$.    
\end{proof}

Theorem~\ref{thm:TFAEunified} specializes to \cite[Prop.~8.1]{DarmonGranville} when $B_0\equiv 1 \mod 4$, $f = 1$ and $C$ is odd and coprime to $B$, since in this case $\O = \O_K$ and $C$ is not divisible by any ramified primes.

The following examples illustrate how the main theorems determine the existence of primitive solutions to generalized Fermat equations. 

\begin{ex}[{\cite[\S8]{DarmonGranville}}]
\label{ex:29,3}
For the generalized Fermat equation $x^{2} + 29y^{2} = 3z^{3}$, we have $K = \Q\left (\sqrt{-29}\right )$ which has cyclic class group of order $6$ generated by a prime lying over $3$, say $\p_{3} = \left (3,\alpha\right )$ where $\alpha = 1 + \sqrt{-29}$. Therefore the only primes we need to invert are $2$ and $3$, so we may take \(R_K = \O_K[1/6]\).
In this case, one can check that $A_{29,3}$ is empty, meaning there are \emph{no admissible twists} of $\C'\to\Y'$. Then Theorem~\ref{thm:TFAEunified}(ii) says $x^{2} + 29y^{2} = 3z^{3}$ has no nontrivial primitive $\Z$-solutions, confirming the example in \cite[Sec.~8]{DarmonGranville}. 
\end{ex}

\begin{ex}
\label{ex:29,19}
Let $B = 29$ again, and take $C = 19$.  In this case, $R_K = \O_K\left [\frac{1}{114}\right ]$ and the admissible twists $\C_{d}'\to(\Y_{29,19})_{R_{K}}$ are parametrized by 
\begin{align*}
    \admissible_{29,19} = \{ & 418  - 19\sqrt{-29},
    -418  + 19 \sqrt{-29},
    3059  - 1140 \sqrt{-29},
    -3059  + 1140 \sqrt{-29},\\ &
    3610  + 1083 \sqrt{-29},
    -3610  - 1083 \sqrt{-29},
    159182  + 17575 \sqrt{-29},\\&
    -159182  - 17575 \sqrt{-29},
    164255  - 15884 \sqrt{-29},
    -164255  + 15884 \sqrt{-29},\\&
    650522  - 534641 \sqrt{-29},
    -650522  + 534641\sqrt{-29}\}.
\end{align*}
In particular, $A_{B,C}\not = \emptyset$ so by Theorem~\ref{thm:TFAEunified}, $x^{2} + 29y^{2} = 19z^{3}$ has primitive $\Z$-solutions, including $(7,4,3)$, $(8,47,15)$ and $(22,1,3)$. 
\end{ex}

\begin{ex}\label{ex:3,31}
    Consider the equation $x^2 + 3y^2 = 31z^3$. Setting $K = \Q\left (\sqrt{-3}\right )$ and $\O = \Z\left [\sqrt{-3}\right ]$, we see that $\Cl(\O) = \Cl(\O_K) = 1$, so Theorem ~\ref{thm:TFAEunified} implies that there exist primitive $\Z$-solutions. Moreover, we claim that \(3\mid y\) for any such solution. Factoring 
    \[C\O_K = 31\O_K = \J_+\J_- = (2 + 3\sqrt{-3})(-2 - 3\sqrt{-3}),\] 
    we notice that both $\J_\pm\cap\O$ are contained in the order $\O_3 \subset \O$ of relative conductor 3. Suppose that $[x:y:z] \in \Y_{3,31}(\Z)$ with $u = x + \sqrt{-3}y$. Then $u = \J_+ \mathfrak{z}^3$, where $\mathfrak{z} \subset \O$ is an ideal with norm $z$. Since $\O$ has trivial class group, we must have $\z = (a + b\sqrt{-3})$, so $\z^3 = (a + b \sqrt{-3})^3$. By the binomial formula, $\z^3 \cap \O_3$ is principal. Thus, $u$ is an element of $\O_3$. By taking norms, we find that $3 \mid y$. That is, \emph{every} primitive solution $[x : y : z]$ must have $3\mid y$. 

    This foreshadows the discussion in \S \ref{subsec:points-with-restricted-y-coordinate}, where we explore how to determine whether or not $\Y_{B,C}(\Z; *^0,\ p \nmid y)$ is nonempty. In particular, Lemma \ref{lem:p-part-kernel} generalizes the approach taken in this example. See also the examples in \S \ref{subsubsec:examples-death-case}, \ref{subsec:examples-brute-force}.
\end{ex}

\subsection{Cohomological obstructions to the Hasse principle}
\label{subsec:coho-descent}

Theorems~\ref{thm:TFAEunified} and~\ref{thm:TFAEunified_prime} show that the refined descent obstruction, as defined below, is the only obstruction to the integral Hasse principle for \(\Y_{B,C}(\Z; \annulus^0)\) (resp. \(\Y_{B,C}(\Z; \annulusTwo^0)\)).

In \cite[Thm.~1.1]{santens23}, Santens proved that the finite descent obstruction is the only obstruction to strong approximation for tame stacky curves $X$ over a number field \(k\) of genus $g<1$ and therefore the only obstruction to the Hasse principle for all the integral models $\mathcal{X}$ of $X$. That is, there exists a finite \'etale group scheme $G$ and a $G$-torsor $\pi \colon \mathcal{C} \to \mathcal{X}$ such that the descent along $\pi$ locus 
\[
    \mathcal{X}(\A_U)^{\pi} = \coprod_{\sigma\in H^1(U, G)} \pi_{\sigma}(\mathcal{C}_{\sigma}(\A_U)) \neq \emptyset \quad\Longrightarrow\quad \mathcal{X}(U) \neq \emptyset
\] for a Dedekind domain $U$ with fraction field \(k\). Here $\A_U$ is the subring of integral ad\`eles of $\A_k$, and 
\[
    \mathcal{X}(\A_U) \coloneqq \prod_{v\in S} \X(k_v) \times \prod_{v\not\in S}\mathcal{X}(\O_v)
\] endowed with the product topology
(see \cite[Prop. 13.0.2]{christensen2020}),
where $S$ is the finite set of archimedean places of $k$ and the product is over places of \(U\).  

Although descent along $\pi$ captures the failure of the integral Hasse principle for a stacky curve $\Y_{B,C}$, it is hard in practice to use it to show whether  $\Y_{B,C}$ satisfies the integral Hasse principle over $U = \Z$, since it is difficult to construct \'etale covers of $\Y_{B,C}$ over \(\Z\). We are able to construct \'etale covers over \(\Z[1/2n]\), and to make the twists of this cover easy to compute, we pass to the ring \(R_K\).  
We note that 
\[
    \Y_{B,C}(\Z) \subseteq \Y_{B,C}\Big(\Z\Big[\frac{1}{2n}\Big]\Big) \subseteq \Y_{B,C}(R_K),
\] 
and \(\Y_{B,C}(\A_{\Z}) \subset \Y_{B,C}(\A_{R_K}) \neq \emptyset\). Descent along \(\pi\colon \mathcal{C}' \to (\Y_{B,C})_{R_K}\) now only allows one to describe the obstruction to the existence of $R_K$-points instead of $\Z$-points.  
Our main theorems thus provide a ``refined'' descent locus,
\[\Y_{B,C}^{\pi, \textnormal{refined}}(\A_{R_K}) \coloneqq \coprod_{\twist \in {\admissible}_{B,C}} \pi_\twist'(\C_d(\A_{R_K})) \subseteq \coprod_{\twist\in H^1(R_K, \mu_n)} \pi_{\twist}'(\C_d(\A_{R_K})) \] 
whose non-emptiness implies that of $\Y_{B,C}(\Z;\annulus^0)$.

\section{The ``cascade'': reduction to the base cases}\label{sec:cascade}
The goal of this section is to reduce to the case where the conditions of one of the main theorems (Theorems~\ref{thm:TFAEunified}~and~\ref{thm:TFAEunified_prime}) are satisfied. More precisely, we would like to prove that for any \(B,C\in\Z\smallsetminus\{0\}\), there exists a finite set of pairs \(\mathcal{P}=\{(B', C')\}\) such that
\[
\Y_{B,C}(\Z) \ne \emptyset \iff \Y_{B',C'}(\Z; \bullet) \neq \emptyset \text{ for some } (B',C')\in\mathcal{P}
\] 
for some condition \(\bullet \) depending on \(B\) and \(C\) that can be tested using the main theorems in  Section~\ref{sec:mainthm}. For a given \((B,C)\) with \(\gcd(f,C)=1\), the set of pairs \(\mathcal{P}\) depends on two kinds of primes: the first kind where \(p^2\mid \gcd(B,C)\), which are dealt with using Lemma~\ref{lem:loc-bijection}. The second kind are those for which \(\Y_{B,C}(\Z; \annulus^r_p)\) is non-empty for some \(r >0\). In~\S\ref{subsec:div-lemmas}, we proved some preliminary lemmas characterizing divisibility conditions for points in \(\Y_{B,C}(\Z;\annulus_p^r)\) for such primes. 
In~\S\ref{subsec:cascadefC1} and \S\ref{subsec:cascadepfC} we use these to prove isomorphisms between groupoids of the form \(\Y_{B,C}(\Z; \bullet)\) and \(\Y_{B',C'}(\Z; \bullet')\)  which in turn, help us compute the set \(\mathcal{P}\). We consider the case \(\gcd(f,C)=1\) in \S\ref{subsec:cascadefC1} and the case when primes \(p \mid \gcd(f,C)\) in~\S\ref{subsec:cascadepfC}. In \S\ref{subsec:points-with-restricted-y-coordinate}, we present a modification of Theorem~\ref{thm:TFAEunified} so that we can  construct a point $[x:y:z]$ with $y$ relatively prime to a fixed integer. 
In~\S\ref{subsec:cascadeRoadMap}, we provide an algorithm that takes any nonzero $B,C\in\Z$ as inputs and outputs whether \(\Y_{B,C}(\Z)\) is empty or not. Together, these results prove Theorem~\ref{thm:MainCascadeTheorem}.

\subsection{Cascade lemmas: \texorpdfstring{$\text{\textbf{gcd}}\bm{(f,C) = 1}$}{gcd(f,C)=1}}
\label{subsec:cascadefC1}
The  lemmas in this section fall naturally into two categories: Lemmas~\ref{lem:cascade_f},~\ref{lem:cascade_f_2} for primes dividing \(f\) and Lemmas~\ref{lem:cascade_y},~\ref{lem:cascade_y_2} for primes dividing \(C\).

\begin{observation}\label{observation:pnmidy}
    We begin with a simple, yet useful observation that follows from the definition of \(\annulus^0\) points. Let $f, B_0,$ and \(C\) be non-zero integers with \(B_0\) squarefree and let \(B = f^2B_0\). Suppose \(p \mid C\) but \(p \nmid f\). Then 
    \[
    \Y_{B,C}(\Z;\annulus^0_p) \cong \Y_{B,C}(\Z;\annulus^0_p, p \nmid y).
    \]
    Indeed, for any \([x:y:z] \in \Y_{B,C}(\Z;\annulus^0_p)\), since $p\mid C$, then \(p \mid (x^2 + f^2B_0 y^2)\). Because \(p \nmid f\) and \(B_0\) is squarefree, we have that \(p \mid x \) if and only if \( p \mid y\). The \(\annulus^0_p\) assumption now implies that \(p \nmid y\). Similarly, Lemma~\ref{lem:2_divides_u} implies that
    \[
    \Y_{B,C}(\Z; \annulusTwo^0) =  \Y_{B,C}(\Z; \annulusTwo^0, 2 \nmid y).
    \]
\end{observation}

\begin{lemma}\label{lem:cascade_f}
     Let $f,$ $B_0,$ and $C$ be nonzero integers with $-1\ne B_0$ squarefree and $\gcd(f,C)=1$. Suppose $p$ is an odd prime such that $p^r \parallel f$ for some \(r>0\). Write \(r' = \lceil 2r/n\rceil n - 2r\). Then we have the following for any \(t \in \Z_{>0}\):
     \begin{align*}
         \Y_{B,C}(\Z; \annulus^t_p)
\begin{cases}
    = \emptyset & t>r;\\
    = \emptyset & 0<t<r \text{ and } n \nmid t;\\
    \cong \Y_{Bp^{-2t}, C}(\Z; \annulus^0_p, p \nmid y) & 0<t \le r \text{ and } n \mid t;\\
     \cong \Y_{Bp^{-2r}, Cp^{r'}}(\Z; \annulus^0_p) & t=r \text{ and } n \nmid t.
\end{cases}         
     \end{align*}
\end{lemma}

\begin{proof}
Let \(p\) be as in the statement of the lemma.
Suppose $t>r$. If $[x:y:z] \in \Y_{B, C}(\Z; *_p^t)$, using Lemma~\ref{lem:p_divides_u} we arrive at a contradiction since $p\nmid C$ and $p^r\parallel f$. On the other hand, if $0<t<r$ and \(n \nmid t\), then by Lemma~\ref{lem:p_divides_u} we again arrive at a similar contradiction.

For the remaining cases, consider the map
\[
\Y_{B,C}(\Z; \annulus^t_p) \rightarrow \Y_{B\cdot p^{-2t}, C \cdot p^{\lceil \frac{2t}{n}\rceil n - 2t}}(\Z; \annulus^0_p, p\nmid y); \qquad [x:y:z] \mapsto \left[x/p^t : y : z/p^{\lceil 2t/n \rceil}\right].
\]
This map is well defined because of the divisibility conditions imposed by the equations, as well as part~\ref{part:lem-p-divides-u:f} of Lemma~\ref{lem:p_divides_u}. It is an isomorphism: the inverse map is well defined since \(p \nmid y\). Notice that if \(n \nmid r\), then \(r' >0\), so that \(p \mid Cp^{r'}\) and the statement for that case follows from Observation~\ref{observation:pnmidy}. If \(n \mid r\), or \(n \mid t\) for \(0<t<r\), we cannot make any such reduction, and must indeed restrict to the subgroupoid with \(p \nmid y\).
\end{proof}

\begin{lemma}\label{lem:cascade_f_2}
    Let $f,$ $B_0,$ and $C$ be nonzero integers with $-1 \ne B_0$ squarefree and $\gcd(f,C)=1$. Suppose $2^r \parallel f$ for $r > 0$. If \(t > r+1\), then \( \Y_{B,C}(\Z; \annulus^t_2) = \emptyset\). Otherwise, we have the following.
\begin{align}
\label{eqn:YBC-tilde-2-cascade}
     \Y_{B,C}(\Z; \annulusTwo^{t-1}_2)&
     \begin{cases}
           = \emptyset & 0<t \le r;\\
             \cong \Y_{B \cdot 2^{-2r}, C \cdot 2^{\lceil \frac{2r}{n}\rceil n - 2r}}(\Z; \annulusTwo^0_2) & t = r+1;\\ 
     \end{cases}
     \\
     \label{eqn:YBC-tilde-2-complement-cascade}
     \Y_{B,C}(\Z; \annulusTwo^{t-1}_2)^c &
     \begin{cases}
           = \emptyset & 0<t<r \text{ and } n \nmid t;\\
             \cong \Y_{B \cdot 2^{-2t}, C}(\Z; \annulus^0_2, 2 \nmid y) & 0<t \le r \text{ and } n \mid t;\\
             \cong  \Y_{B \cdot 2^{-2r}, C \cdot 2^{\lceil \frac{2r}{n}\rceil n - 2r}}(\Z; \annulus^0_2) & t=r \text{ and } n \nmid t;\\
             = \emptyset & t = r+1          
     \end{cases}
\end{align}
where \(\Y_{B,C}(\Z; \annulusTwo^{t-1}_2)^c\) denotes the complement of \(\Y_{B,C}(\Z; \annulusTwo^{t-1}_2)\) inside \(\Y_{B,C}(\Z; \annulus^t_2)\).
\end{lemma}

\begin{proof}
 When $t>r+1$, the proof is similar to the analogous case in Lemma~\ref{lem:cascade_f}, using Lemma~\ref{lem:2_divides_u}.

For~\eqref{eqn:YBC-tilde-2-cascade}, recall that if \([x:y:z] \in \Y_{B,C}(\Z; \annulusTwo^{t-1}_2) \subset \Y_{B,C}(\Z; \annulus^{t}_2) \), then \(2^{t-1} \parallel x\) and \(2^t \parallel u\). The proof is now analogous to that of Lemma~\ref{lem:cascade_f} using the map
\[
[x:y:z ] \mapsto \left[\frac{x}{2^{t-1}}: y: \frac{z}{2^{\lceil \frac{2t-2}{n}\rceil}} \right]
\]
and the divisibility conditions imposed by Lemma~\ref{lem:2_divides_u} and Observation~\ref{observation:pnmidy}. 

For the proof of~\eqref{eqn:YBC-tilde-2-complement-cascade}, we observe that \(2^t \mid x\) and \(2^t \parallel u\) and follow a similar strategy, using Lemma~\ref{lem:2_divides_u} and Observation~\ref{observation:pnmidy} and the map
\[
[x:y:z ] \mapsto \left[\frac{x}{2^{t}}: y: \frac{z}{2^{\lceil \frac{2t}{n}\rceil}} \right].
\]
\end{proof}

\begin{lemma}\label{lem:cascade_y}
Let $f$, $B_0$, and $C$ be nonzero integers with $-1\ne B_0$ squarefree and $\gcd(f,C)=1$. Suppose $p$ is an odd prime. If $p \nmid f$, then for all \(0 < r \le \lfloor v_p(C)/2 \rfloor\)
    we have an isomorphism of groupoids
    \[\Y_{B, C}(\Z; *_p^r) \simeq \Y_{B, C \cdot p^{-2r}}(\Z; *_p^0, p \nmid z).\]
   Further, if \(r>\lfloor v_p(C)/2 \rfloor\),  then \(\Y_{B,C}(\Z; \annulus_p^r) = \emptyset\).
\end{lemma}

\begin{proof}
    Because $p\nmid f$, Lemma~\ref{lem:p_divides_u} implies that $p^r \mid y$ and $p \nmid z$.  The map 
    \[[x:y:z] \mapsto [x/p^r : y/p^r : z]\]
    is well-defined since $p\nmid z$, and since $p^r\parallel\left( x+f\sqrt{-B_0}y\right)$ implies $p\nmid (x+f\sqrt{-B_0}y)/p^{r}$.  The inverse map is well-defined because for all $[x':y':z']\in\Y_{B, C/p^{2r}}(\Z; *_p^0, p \nmid z)$, we have that $p \nmid z'$ and $p^r\parallel\left(x'p^r+f\sqrt{-B_0}y'p^r\right)$.
   If \(r>\lfloor v_2(C)/2 \rfloor\), the statement follows from part~\ref{part:lem-p-divides-u:C} of Lemma~\ref{lem:p_divides_u}.
\end{proof}

\begin{lemma}\label{lem:cascade_y_2}
    Let $B_0$ be squarefree and $\gcd(f,C)=1$. Suppose $2 \nmid f$ and that $0 < r \le \lfloor v_2(C)/2 \rfloor$.
    Then we have an isomorphism of groupoids
    \begin{align*}
        \Y_{B, C}(\Z; *_2^{r}) &\cong \Y_{B,C}(\Z; \tilde{*}_2^{r-1}) \sqcup \Y_{B,C}(\Z; \tilde{*}_2^{r-1})^c\\
        & \cong \Y_{B, C\cdot 2^{-2(r-1)}}(\Z; *_2^0, 2 \nmid z) \sqcup \Y_{B, C\cdot 2^{-2r}}(\Z; *_2^0, 2 \nmid z).
    \end{align*}
    Further, 
    \begin{align*}
        \Y_{B,C}(\Z; \annulus^r_2) = \emptyset \qquad &\text{ if } r>  \lfloor v_2(C)/2 \rfloor +1;\\
        \Y_{B,C}(\Z; \annulusTwo^{r-1}_2)^c = \emptyset \qquad &\text{ if } r =  \lfloor v_2(C)/2 \rfloor +1.
    \end{align*}
\end{lemma}

\begin{proof}
    The proofs are identical to that of Lemma~\ref{lem:cascade_y}, making use of Lemma~\ref{lem:2_divides_u}.
\end{proof}

\subsection{Cascade lemmas: \texorpdfstring{$\bm{p \mid } \text{\textbf{gcd}}\bm{(f, C)}$}{p|gcd(f,C)} and \texorpdfstring{$\bm{p \parallel C}$}{p||C}}
\label{subsec:cascadepfC}
\
We now consider the case when \(f\) and \(C\) and not coprime. By Lemma~\ref{lem:loc-bijection}, may assume that \(p \mid \gcd(f,C)\) implies \(p \parallel C\).

\begin{lemma}\label{lem:cascade_nnot|2r-1}
    Let $f,$ $B_0,$ and $C$ be nonzero integers with $-1 \ne B_0$ squarefree.  Let \(p\) be an odd prime with \(p \mid f\) and  \(p \parallel C\). Suppose \(p^r \parallel f\) with \(r>0\) and let \(r' = \lceil (2r-1)/n \rceil n - 2r\). Then we have the following
\begin{align}
 \label{eqn:fc-not-coprime-map}
    \Y_{B,C}(\Z; \annulus^t_p) 
    \begin{cases}
    = \emptyset & t >r;\\
    = \emptyset & 0<t<r \text{ and } n \nmid 2t-1;\\
    \cong \Y_{B\cdot p^{-2r} , C \cdot p^{r'}}(\Z; *_p^0) & t=r \text{ and } n \nmid 2t-1;\\
    \cong \Y_{B\cdot p^{-2t}, C\cdot p^{-1}}(\Z; *_p^0, p \nmid y)  & 0<t \le r \text{ and } n \mid 2t-1.
    \end{cases}
\end{align}
\end{lemma}
\begin{proof}
    The proof is similar to that of Lemma~\ref{lem:cascade_f} using the
map
\[
[x:y:z] \mapsto [x/p^{t}: y : z/p^{(2t-1)/n}]
\]
and by using Lemma~\ref{lem:fCnotcoprime-no-star0-points}.
\end{proof}
Similarly, for \(p=2\), we have the following.
\begin{lemma}\label{lem:p2:cascade_nnot|2r-1}
 Let \(f, B_0, C\) be as before with \(2^r \parallel f\) for some \(r >0\) and \(2 \parallel C\). Then, if \(t>r+1\), then \(\Y_{B,C}(\Z; \annulus^t_2) = \emptyset\). Moreover,
\begin{align}
     \Y_{B,C}(\Z; \annulusTwo^{t-1}_2)&
     \begin{cases}
     \label{eqn:YBC-tilde-2-fc-cascade}
     = \emptyset & 0 < t \le r;\\
      \cong \Y_{B \cdot 2^{-2r}, C \cdot 2^{\lceil (2r-1)/n\rceil n - 2r} }(\Z; \annulusTwo^0_2) & t=r+1. \\
     \end{cases}     
     \\
     \Y_{B,C}(\Z; \annulusTwo^{t-1}_2)^c &
     \begin{cases}
     \label{eqn:YBC-tilde-2-complement-fc-cascade}
             = \emptyset & 0<t\le r 
             \text{ and } n \nmid 2t-1;\\
             \cong \Y_{B \cdot 2^{-2t}, C\cdot 2^{-1}}(\Z; \annulus^0_2, 2 \nmid y) & 0<t < r \text{ and } n \mid 2t-1;\\
             \cong  \Y_{B \cdot 2^{-2r}, C \cdot 2^{\lceil (2r-1)/n\rceil n - 2r}}(\Z; \annulus^0_2) & t=r \text{ and } n \nmid 2t-1;\\
             \cong 
             \Y_{B \cdot 2^{-2r}, C\cdot 2^{-1}}(\Z; \annulus^0_2, 2 \nmid y) & t=r \text{ and } n \mid 2t-1;\\
             \cong \emptyset & t = r+1.  
     \end{cases}
\end{align}

\end{lemma}
\begin{proof}
    The proof is similar to that of Lemma~\ref{lem:cascade_f_2} using the
maps
\[
[x:y:z] \mapsto [x/2^{t-1}: y : z/2^{(2t-3)/n}] \text{ and } [x/2^{t}: y : z/2^{(2t-1)/n}]
\]
for~\eqref{eqn:YBC-tilde-2-fc-cascade} and~\eqref{eqn:YBC-tilde-2-complement-fc-cascade}, respectively and by using Lemma~\ref{lem:fCnotcoprime-no-star0-points}.
\end{proof}

\begin{ex}\label{ex:243,93}
Consider the equation $x^2 + 243y^2 = 93z^3$. It is straightforward to see that we must have $3 \mid x$, $3 \mid z$, and $3 \nmid y$. In fact, we must have $3^2 \mid x$ and $3 \parallel z$; if the latter is not satisfied, we force $3^3 \mid x$, and therefore $3 \mid y$, contradicting the primitivity of $[x:y:z]$. Therefore, we consider the cascade map
\[
    \Y_{243,9}(\Z;\annulus_3^2) \to \Y_{3,31}(\Z;\annulus_3^0) \colon [x:y:z] \mapsto [x/3^2 : y : z/3], 
\] where points in the image of this map satisfy $3\nmid y$ by Lemma~\ref{lem:cascade_nnot|2r-1}. Then it suffices to find points on the target groupoid $\Y_{3,31}(\Z)$ with $3\nmid y$. However, Example ~\ref{ex:3,31} says that any $[x:y:z] \in \Y_{3,31}(\Z)$ must satisfy $3\mid y$. This implies that
  \[
    \Y_{3^5, 3\cdot 31}(\mathbb{Z}) = \Y_{3^5, 3 \cdot 31}(\mathbb{Z}; \annulus_3^2) \overset{\sim}{\rightarrow} \Y_{3, 31}(\mathbb{Z}; \annulus_3^0, 3 \nmid y) = \emptyset.
  \]
Therefore there are no primitive $\Z$-solutions to this generalized Fermat equation. 
\end{ex}

\subsection{Points with restricted \texorpdfstring{\(y\)}{y}-coordinates}
\label{subsec:points-with-restricted-y-coordinate}

In several of the results in the previous section, the target groupoids are of the form $\Y_{B', C'}(\Z; \annulus_p^0, p\nmid y)$ for $p$ an odd prime or $\Y_{B', C'}(\Z; \annulusTwo_2^0, 2\nmid y)$. When Observation~\ref{observation:pnmidy} applies, we can conclude that these are equal to $\Y_{B', C'}(\Z; \annulus_p^0)$ or $\Y_{B', C'}(\Z; \annulusTwo_2^0)$, respectively. In general, as illustrated in Example~\ref{ex:3,31}, it is possible to have \(B',C'\) with \(p \nmid C'\), but \(p \mid y\) for every \([x:y:z] \in \Y_{B',C'}(\Z)\). 

Examining the the proofs of Theorems \ref{thm:TFAEunified} and \ref{thm:TFAEunified_prime}, we deduce the following. Here, for an integer $N \geq 1$, we denote by $\O_N \subset \O$ the order of relative conductor $N$. 

\begin{corollary}\label{cor:p_not_dividing_y}
    Suppose $\gcd(f,C) = 1$ and fix an integer $N\geq 1$. We have that 
    \[\Y_{B,C}(\Z; \annulus^0,\ \gcd(y,N) = 1) \neq \emptyset \quad (\text{resp}.\ \Y_{B,C}(\Z; \annulusTwo^0,\ \gcd(y,N) = 1) \neq \emptyset)\]
    is equivalent to the existence of an ideal $\z \subset \O_N$ such that $\J_+ \r (\z \O)^n$ is principal in $\O$ (resp.\ $\p_2^{k} \J_+ \r (\z \O)^n$ is principal in \(\O\), for the appropriate exponent $k$ prescribed by Theorem~\ref{thm:TFAEunified_prime}), but $(\J_+ \r(\z \O)^n) \cap \O_p$ is not principal in \(\O_p\) for any $p \mid N$ (resp.\ $(\p_2^{k} \J_+ \r (\z \O)^n) \cap \O_p$ is not principal in \(\O_p\) for any $p \mid N$).
\end{corollary}

When $\gcd(N,n)=1$, it is relatively straightforward to construct such $\z$ using similar methods as in \S\ref{sec:mainthm}, yielding the proposition below.

\begin{proposition}\label{prop:cascade-with-gcds}
    Let $\gcd(f,C)=1$ and fix an integer $N$ satisfying $\gcd(N,nC)=1$. Then for any integer \(M\),
    \[\Y_{B,C}(\Z; \annulus^0, \gcd(z,M)=1) \neq \emptyset \iff \Y_{B,C}(\Z; \annulus^0,\ \gcd(y,N)=1,\ \gcd(z,M)=1) \neq \emptyset\]
    and
    \[\Y_{B,C}(\Z; \annulusTwo^0, \gcd(z,M)=1) \neq \emptyset \iff \Y_{B,C}(\Z; \annulusTwo^0,\ \gcd(y,N)=1,\ \gcd(z,M)=1) \neq \emptyset.\]
\end{proposition}

\begin{proof}
    The reverse implication is immediate. For the forward implication, it suffices to take ~$N$ positive and squarefree. The conclusion is vacuous when $N = 1$, so it suffices to proceed by induction on the number of prime divisors of $N$. Equivalently, if we have a point \([x:y:z]\) with $\gcd(y,N) = 1$ and a prime $p$ satisfying $p \nmid nCN$, we must find a point \([x':y':z']\) with $\gcd(y',Np)=1$. We will first focus on the case of $B_0 \equiv 1,2 \pmod{4}$, then note the necessary adjustments for that of $B \equiv 3 \pmod{4}$.

    Suppose we have $[x:y:z] \in \Y_{B,C}(\Z; \annulus^0,\ \gcd(y,N)=1,\ \gcd(z,M)=1)$ and a prime $p \nmid nCN$. Recall $\O = \Z[f\sqrt{-B_0}]$ and $\O_p, \O_N, \O_{Np} \subset \O$ denote the orders of relative conductor $p,N,Np$, respectively. By Theorem~\ref{thm:TFAEunified}, we have $C\O_K = \J_+ \J_- \r^2$ for $[\J_\pm \cap \O] \in n\Cl(\O)$, satisfying the appropriate splitting conditions. By the proof of Theorem~\ref{thm:TFAEunified}, letting \(u = x + \sqrt{-B}y\), we have 
    \begin{equation}\label{eqn:factorization-nm-eqn}(\J_+\r \cap \O)\z^n = u\O\end{equation}
    and since  $\gcd(y,N)=1$, we must have $u \notin \O_N$.
  
    If additionally $u \notin \O_p$, we are done, since $p \nmid y$ in this case. Therefore, assume $u \in \O_p$; the fact that \(p\) is coprime to \(C\) along with~\eqref{eqn:factorization-nm-eqn} imply that  $[\J_\pm \cap \O_p] \in n \Cl(\O_p)$. We now have two cases: $p \mid B$ and $p \nmid B$. In both cases, our goal is to modify $u = x + f\sqrt{-B_0}y$ by an element $w =a + Nf\sqrt{-B_0}b$ with $p \nmid b$, so the resulting product $uw$ is in $\O \smallsetminus (\O_N \cup \O_p)$, and hence produces a point $[x' : y': z'] \in \Y_{B,C}(\Z)$ with $\gcd(y',Np)=1$.
    
    Consider first the case of $p \mid B$. From~\eqref{seq:ClO_to_ClK} we deduce 
    \[\frac{\#\Cl(\O_p)}{\#\Cl(\O)} [\O^\times : \O_p^\times] = p - \left(\frac{-B}{p}\right) = p.\]
    
    If $[\O^\times : \O_p^\times] = p$, then $[\O_N^\times : \O_N^\times\O_p^\times] = p$ by the second isomorphism theorem. Therefore there exists a unit $w \in \O_N^\times \smallsetminus \O_N^\times\O_p^\times$. Writing $w = a + Nf\sqrt{-B_0}b$, we have $p \nmid b$, and we can simply replace the ideal $u\O$ by $uw\O$ without changing the left-hand-side of~\eqref{eqn:factorization-nm-eqn}, and $uw \in \O \smallsetminus (\O_N \cup \O_p)$. 

    If $[\O^\times : \O_p^\times] = 1$, then $\O^\times = \O_p^\times$, which implies $\O_N^\times = \O_{Np}^\times$. Now by Lemma~\ref{lem:class_number_formula} we have
\begin{align}\label{eqn:kernel-ONp-to-Op}
    \ker\left(\Cl(\O_{Np}) \to \Cl(\O_N)\right) \simeq \Z/p\Z.
\end{align}
    Since $p \nmid n$, multiplication by $n$ is an isomorphism of this kernel, allowing us to find a nontrivial $[\mathfrak{w}] \in \ker\left(\Cl(\O_{Np}) \to \Cl(\O_N)\right)$. We see that $(\J_+ \r \cap \O)(\z (\mathfrak{w}\O))^n = uw\O$, where $w = \mathfrak{w}^n \in \O_N \smallsetminus \O_{Np}$. Once again, after taking norms the resulting point has $\gcd(y',Np) = 1$. 
 
    Finally, suppose $p \nmid B$. Recall that $\Cl(\O) \simeq \Cl(-4B)$, the class group of forms of discriminant $-4B$,\footnote{We abuse notation here. If \(B_0>0\), i.e. in the imaginary quadratic case, there is an isomorphism between the class group of primitive positive definite quadratic forms of discriminant \(-4B\) and the class group of the corresponding order. When \(B_0<0\), the isomorphism is between the class group of primitive quadratic forms of discriminant \(-4B\) and the narrow class group (see for instance \cite[\S5.2]{Cohen-ANT}). We use the relevant isomorphism in each case, and noting that the class group and the narrow class group only differ by \(2\)-torsion, and thus statements about \(n\)th powers are the same in both.} in which $x^2 + By^2$ corresponds to the identity. Similarly, $\Cl(\O_p) \simeq \Cl(-4p^2B)$, with $x^2 + p^2By^2$ the identity class. The class of $p^2x^2 + By^2$ is 2-torsion in $\Cl(-4p^2B)$, and hence in $n \Cl(-4p^2B)$. Thus, we can find an ideal $\mathfrak{w}$ in $\O_p$ such that $\mathfrak{w}^n$ corresponds to the class of $p^2x^2 + By^2$. Moreover, we can choose $\mathfrak{w}$ coprime to $C$, $\z$, and $Np$. By the theory of quadratic forms (see e.g.~\cite[Thm.~7.7(iii)]{Cox}), we have that the norm of $\mathfrak{w}^n \O_p$ is $p^2a'^2 + Bb'^2$ for some \(a',b' \in \Z\) with $\gcd(N,a')=1$ and \(p \nmid b'\). Using the analogue of Lemma~\ref{lem:IOKf=IOf}, we have norm preserving bijections between \(I(\O_{Np}, Np), I(\O_p, Np)\) and \(I(\O, Np)\). In particular, \(\mathfrak{w}^n \O\) has norm \(p^2a^2 + NBb^2\) with \(\gcd(a,N)=1\) and \(p \nmid b\).

    The coprimality conditions on \(a\) and \(b\), together with the assumption \(p \nmid B\), imply that the ideals \((pa + Nf\sqrt{-B_0}b)\O\) and \((pa - Nf\sqrt{-B_0}b)\O \) are coprime. In particular, \((pa + Nf\sqrt{-B_0}b)\O\) (resp. \((pa - Nf\sqrt{-B_0}b)\O \))  is an \(n\)th power of some ideal in \(\O\), since their product is an \(n\)th power. Therefore, $\mathfrak{w'}^n\O = (pa + Nf\sqrt{-B_0}b)\O$ for some \(\mathfrak{w'}\). From this, we deduce
    \begin{align*}
        (\J_+\r \cap \O)(\z(\mathfrak{w}'\O))^n &= (x + f\sqrt{-B_0}y)(pa + Nf\sqrt{-B_0}b) \\
        &= (pax - Nf^2B_0by) + f\sqrt{-B_0}(pay + Nbx)\\ 
        &= x' + f\sqrt{-B_0}y'.
    \end{align*}
    Since $p \mid y$, we have $p \nmid y'$. Similarly, $\gcd(N, y') = 1$. Taking norms gives the desired point and the same arguments from the proof of Theorem~\ref{thm:TFAEunified} ensure $[x:y:z]$ is primitive.
   
    When $B_0 \equiv 3 \pmod{8}$ and $C$ is odd, we still have $\O = \Z[f\sqrt{-B_0}]$, but this order has conductor $2f$ and discriminant \(-4B\). No further adjustments are necessary to the proof of the statement about $\annulus^0$; that for $\annulusTwo^0$ holds vacuously since both groupoids are empty (Remark~\ref{rem:star0and3mod4}). When $B_0 \equiv 3 \pmod{8}$ and $4 \parallel C$, we have $\O = \Z[f \frac{\sqrt{-B_0} + 1}{2}]$, but the relevant $u$ is in $\Z[f\sqrt{-B_0}]$, so the same arguments apply to prove the statement for $\annulusTwo^0$.

    When $B_0 \equiv 7 \pmod{8}$, the relevant order in Theorem~\ref{thm:TFAEunified} is $\O = \Z[f\frac{\sqrt{-B_0} + 1}{2}]$. However, by Lemma~\ref{lem:class_number_formula}, $\Cl(\O_2)/n\Cl(\O_2) \simeq \Cl(\O)/n\Cl(\O)$ in this case, so when \(C\) is odd, one may run the same argument by replacing \(\O\) by \(\O_2 = \Z[f\sqrt{-B_0}]\), whose discriminant is \(-4B\). If $C$ is even, then $\Y_{B,C}(\Z; *^0) = \emptyset$, so the first statement is trivially satisfied.
    
    To prove the second statement we follow essentially the same approach, now using Theorem~\ref{thm:TFAEunified_prime}. Note that points in \(\Y_{B,C}(\Z; \annulusTwo^0)\) have odd $y$ coordinate, so we may assume $N,p$ are odd. Given a point $[x:y:z] \in \Y_{B,C}(\Z; \annulusTwo^0, \gcd(z,M) = 1)$, we have
    \[(\p_2^{v_2(C)-1}\pbar_2\J_+\r \cap \O)\z^n = u\O = x + f\sqrt{-B_0} y,\]
    with $\gcd(y,N)=1$. If $p \mid y$, it suffices to find one of the following:
    \begin{itemize}
        \item a unit $w = a + Nf\sqrt{-B_0}b \in \O_{2N}^\times \smallsetminus \O_{2Np}^\times$, or
        \item an ideal $\mathfrak{w} \subset \O_{2Np}$ such that $\mathfrak{w}^n\O_{2N}$ is principal, but $\mathfrak{w}^n$ is not.
    \end{itemize}
  The proof of the existence of such a \(w\) or \(\mathfrak{w}\) is similar to the above, noting that \(p\) and \(N\) are both odd.  Then modifying the equation above, we find
    \[(\p_2^{v_2(C)-1}\pbar_2\J_+\r\I \cap \O)(\z (\mathfrak{w}\O))^n = uw\O = (x + f\sqrt{-B_0} y)(a + Nf\sqrt{-B_0}b) = x' + f\sqrt{-B_0}y',\]
    with divisibility conditions forcing $\gcd(y', Np) = 1$. Taking norms yields the desired point.
\end{proof}

In the proof above, it is important that \(\gcd(N,n)=1\). If \(p \mid n\), then multiplication by \(n\) is no longer an isomorphism on the kernel in~\eqref{eqn:kernel-ONp-to-Op}. And indeed, one can produce counterexamples to this proposition if \(\gcd(N,n) \neq 1\) (see Examples~\ref{ex:3,31} and~\ref{ex:243,93}). 

To illustrate this in the case $\gcd(N,n) > 1$, suppose $N=p$ for a prime $p \mid n$. Recall from Lemma \ref{lem:class_number_formula} we have
\[
\frac{\#\Cl(\O_{p})}{\#\Cl(\O)} [\O^{\times}: \O_p^{\times}] = p\left(1 - \left(\frac{\Disc(\O)}{p} \right)\frac1p\right).
\]
If $[\O^\times : \O_p^\times] > 1$ or $p \nmid \Disc(\O)$, then we can apply the techniques of the proof of Proposition \ref{prop:cascade-with-gcds} to show
\[\Y_{B,C}(\Z; \annulus^0) \neq \emptyset \iff \Y_{B,C}(\Z; \annulus^0,\ p \nmid y) \neq \emptyset\]
(and similarly for $\annulusTwo^0$ where appropriate). 

Suppose then that $[\O^\times : \O_p^\times] = 1$ and $p \mid \Disc(\O)$. Let $q\colon \Cl(\O_p) \to \Cl(\O)$ denote the natural surjection. Lemma \ref{lem:class_number_formula} gives 
\[\ker(q) \simeq \Z/p\Z.\]
Let $\eta \colon \Cl(\O_p) \to \Cl(\O)$ denote the composition $[n] \circ q = q \circ [n]$ and consider the diagram \eqref{eq:eta_diagram}.

\begin{equation}\label{eq:eta_diagram}
    \begin{tikzcd}
    \ker(\eta) \arrow[dr] & \ker(q) \arrow[d] \arrow[r, "0"] & \ker(q) \arrow[d]\\
    \Cl(\O_p)[n] \arrow[r] \arrow[d] &\Cl(\O_{p})\arrow[d, two heads,"q" ]\arrow[r, "{[n]}"] \arrow[dr, "\eta"] & \Cl(\O_{p}) \arrow[d, two heads, "q"] \\
    \Cl(\O)[n] \arrow[r] & \Cl(\O) \arrow[r, "{[n]}"] & \Cl(\O).
    \end{tikzcd}
\end{equation}

\begin{lemma}\label{lem:p-part-kernel}
    Suppose \(\gcd(f,C)=1\) and $p \mid n$, $p \nmid C$. If \(\Y_{B,C}(\Z; \annulus^0) \neq \emptyset\) (resp.\ $\Y_{B,C}(\Z; \annulusTwo^0) \neq \emptyset$), let \(\J\) satisfy $[\J] = -[\J_+]$ provided by Theorem \ref{thm:TFAEunified} (resp.\ $[\J] = -[\p_2^k \J_+]$ as in Corollary \ref{cor:p_not_dividing_y}) satisfying $[\J] \in n\Cl(\O)$. Then, \(\Y_{B,C}(\Z; \annulus^0, p\nmid y) \neq \emptyset\) (resp.\ $\Y_{B,C}(\Z; \annulusTwo^0,\ p \nmid y) \neq \emptyset$) if and only if either  \([\J \cap \O_p] \notin n \Cl(\O_p)\) or \( \ker(\eta) \not\subseteq \Cl(\O_p)[n]\).
\end{lemma}

\begin{proof}
 Note that since \(\eta = q \circ [n]\), we already have that \(\Cl(\O_p)[n] \subseteq \ker(\eta)\).  Let \(\J\) be as in the theorem. By the assumption, we have that \([\J \cap \O] \in n\Cl(\O)\) and since \(p \mid n \), \([\J \cap \O]\) is in the image of the bottom map \([n] \colon \Cl(\O) \to \Cl(\O)\) of \eqref{eq:eta_diagram}. The preimage \(q^{-1}([\J \cap \O])\) contains the class \([\J \cap \O_p]\). 

By Corollary~\ref{cor:p_not_dividing_y}, \(\Y_{B,C}(\Z;\annulus^0, p \nmid y) \neq \emptyset\) (resp.\ \(\Y_{B,C}(\Z;\annulusTwo^0, p \nmid y) \neq \emptyset\)) if and only if there is a \([\z] \in \Cl(\O_p)\) such that \(n[\z] \neq [\J \cap \O_p]\) and \(\eta([\z]) = [\J \cap \O]\).
    
If \([\J \cap \O_p] \notin n\Cl(\O_p)\), then we are done. If not, suppose \(n[\z] = [\J \cap \O_p]\) so that \(\eta([\z]) = [\J \cap \O]\). Then for any \(\alpha \in \ker(\eta) \setminus \Cl(\O_p)[n]\), \(\eta([\z] + \alpha) = [\J \cap \O]\) but \(n([\z] + \alpha) \neq [\J \cap \O_p]\). Conversely, if \([\J \cap \O_p] \in n\Cl(\O_p)\), then any \([\z] \in \Cl(\O_p)\) with \(\eta([\z]) = [\J \cap \O]\) and \(n[\z] \neq [\J \cap \O_p]\) produces a non-trivial element in \(\ker(\eta) \setminus \Cl(\O_p)[n]\). 
\end{proof}

\subsection{Road map to the cascade method}
\label{subsec:cascadeRoadMap}

The lemmas presented in \S\ref{subsec:cascadefC1} and~\S\ref{subsec:cascadepfC} cover most cases for modifying non-zero $B,C\in \Z$ prime-by-prime to reduce to the cases covered in our main theorems (Theorems~\ref{thm:TFAEunified}~and~\ref{thm:TFAEunified_prime}). In this section, we outline the strategy to apply these lemmas to any non-zero \(B,C\), proving Theorem~\ref{thm:MainCascadeTheorem}. 

We start with the following observation, which follows from definitions and the independence of behaviors at different primes.

\begin{observation}
\label{obs:iterate_cascade}
    The isomorphisms of Lemmas~\ref{lem:cascade_f},~\ref{lem:cascade_f_2}, and~\ref{lem:cascade_nnot|2r-1} are compatible with choosing conditions at multiple primes simultaneously; that is, if distinct primes $p$ and $q$ satisfy $p^{r_p}\parallel f$ and $q^{r_q}\parallel f$ with $r_p, r_q\not\equiv 0\pmod{n}$, then the following diagram commutes.
    \begin{center}
    \begin{tikzcd}
    {\Y_{B,C}(\mathbb{Z}; *_p^{r_p}, *_q^{r_q})} \arrow[d, "\sim"'] \arrow[r, "\sim"] & {\Y_{B'',C''}(\mathbb{Z}; *_p^{r_p}, *_q^{0})} \arrow[d, "\sim"] \\
    {\Y_{B',C'}(\mathbb{Z}; *_p^{0}, *_q^{r_q})} \arrow[r, "\sim"']           & {\Y_{B''',C'''}(\mathbb{Z}; *_p^{0}, *_q^{0})}
    \end{tikzcd}
    \end{center}
    Similar diagrams commute when using the isomorphisms from Lemmas~\ref{lem:cascade_y},~\ref{lem:cascade_y_2}, and~\ref{lem:p2:cascade_nnot|2r-1}, or those involving restrictions on the \(y\)-coordinates, including when combining isomorphisms from all lemmas. 
\end{observation}

\noindent
Let \(B,C \) be non-zero integers.
\begin{enumstep}
    \item\label{step:cascade-remove-squares}
    For all primes $p$ with \(p^2 \mid \gcd(B,C)\), use Lemma~\ref{lem:loc-bijection} to conclude that 
    $$\Y_{B,C}(\Z)\ne \emptyset \iff \Y_{B\cdot p^{-2}, C\cdot p^{-2}}(\Z)\ne\emptyset.$$  This reduces to the case in which \(\gcd(B,C)\) is squarefree, i.e., to the case where \(p~\mid~\gcd(f,C) \) implies \(v_p(C) \le 1 \).\medskip

We note that for any prime \(p\), a point in \(\Y_{B,C}(\Z)\) lies in a \(\annulus^r_p\) set for some \(r \ge 0\). So in fact, for a fixed \(p\), \(\Y_{B,C}(\Z) = \bigcup_{t \ge 0} \Y_{B,C}\left(\Z; \annulus^t_p\right)\). For any prime \(p\), let \(T_p\) denote the set of \(t \in \Z_{\ge 0}\), such that \(\Y_{B,C}(\Z;\annulus^t_p)\) is potentially non-empty. By \S\ref{subsec:cascadefC1}~and~\S\ref{subsec:cascadepfC}, this set is finite and can be computed explicitly. In particular, 
\[
\Y_{B,C}(\Z) =  \coprod_{t \in T_p} \Y_{B,C}(\Z; \annulus^t_p).
\]
Further, note that for \(t \geq 1\), \(\Y_{B,C}(\Z; \annulus^t_2) = \Y_{B,C}(\Z; \annulusTwo^{t-1}_2) \sqcup \Y_{B,C}(\Z; \annulusTwo^{t-1}_2)^c \), and the above decomposition still makes sense for \(p=2\).

\begin{notation}
    For convenience, for any subgroupoid $S \subset \Y_{B,C}(\Z)$, let \(S(\annulus^t_p)\) (resp. \(S(\annulusTwo^t_2)\) denote the subgroupoid of points in \(S\) satisfying \(\annulus^t_p\) (resp. \(\annulusTwo^t_2\)).
\end{notation} 

\medskip

\item\label{step:cascade-clear-gcd(f,C)} 

If \(\gcd(f,C)=1\), then set $\calP\colonequals\{\Y_{B,C}(\Z)\}$. Otherwise, write \(\{ p : p \mid \gcd(f,C)\}= \{p_1, p_2 \ldots p_k\}\). Compute \(T_{p_i}\) according to \S\ref{subsec:cascadepfC}.  

By Observation~\ref{obs:iterate_cascade}, write
\[
\Y_{B,C}(\Z) = \coprod_{(t_1, t_2, \ldots , t_k) \in T_{p_1} \times T_{p_2} \times \ldots \times T_{p_k}} \Y_{B,C}(\Z; \annulus^{t_1}_{p_1}, \ldots,\annulus^{t_k}_{p_k})
\]
and apply the cascade (Lemma~\ref{lem:cascade_nnot|2r-1}~or~\ref{lem:p2:cascade_nnot|2r-1})  for all the \(p_i\)'s together to each set in the disjoint union. 
This gives us a set of \(\{\left(B_{((p_i),(t_i))}, C_{((p_i),(t_i))}\right)\colon (t_i) \in \prod T_{p_i}\}\) such that \(C_{((p_i),(t_i))}\) is coprime to the square part of \(B_{((p_i),(t_i))}\) and  \(\Y_{(B,C)}(\Z)\) is possibly nonempty
if and only if one of the groupoids in 
\begin{equation}
\label{equ:SetP}
\calP\colonequals \{\Y_{B_{((p_i),(t_i))}, C_{((p_i),(t_i))}}(\Z; (\annulus^0_{p_i})_i)\}
\end{equation}
is nonempty (with appropriate adjustments at \(2\) and possible divisibility restrictions on $y$). Perform the following steps for each of these target groupoids in \(\calP\) until $\calP$ is empty. 
\medskip

\item\label{step:cascade-apply-main-theorem}
Let \(S = \Y_{B', C'}(\Z; \bullet) \in\calP\). Apply Theorem~\ref{thm:TFAEunified} (or Theorem~\ref{thm:TFAEunified_prime} if \(B_0 \equiv 3 \mod 4\)) to decide if \(S(\annulus^0)=\Y_{B',C'}(\Z;\annulus^0)\) (or \(S(\tilde{\annulus}^0)=\Y_{B',C'}(\Z;\tilde{\annulus}^0)\)) is nonempty. If \(S(\annulus^0) \neq \emptyset\)  (or \(S(\annulusTwo^0) \neq \emptyset\))
then the original curve has an integral point, so one can stop. If not, then
\begin{enumalg}
    \item If $S = S(\annulus^0) = \emptyset$ (or $S = S(\annulusTwo^0) = \emptyset$), remove $S$ from $\calP$ and return to the start of Step~\ref{step:cascade-apply-main-theorem}; or,

    \item Let $\{p_i\}_i$ denote the (finite) set of primes such that points in $S$ satisfy \((\annulus^0_{p_i})_i\) or \((\tilde\annulus^0_{p_i})_i\) (note that this set of primes may differ from the set of primes in~\eqref{equ:SetP}).
    If it exists, pick $p \mid f'C'$ with $p\notin\{p_i\}_i$ and go to Step~\ref{step:cascade-iteration}.  Otherwise, replace $S$ by $S(\annulus^0)$ (or $S(\annulusTwo^0)$) in $\calP$ and return to the start of Step~\ref{step:cascade-apply-main-theorem}.
\end{enumalg} 
\medskip

\item\label{step:cascade-iteration} 
With the notation from Step~\ref{step:cascade-apply-main-theorem}, remove $S$ from $\calP$ and write
\[
\Y_{B',C'}(\Z) = \coprod_{t \in T_{p}} \Y_{B',C'}(\Z; \annulus^{t}_{p}).
\]
For each $t\in T_p$, do the following and then return to the start of Step~\ref{step:cascade-apply-main-theorem}.  If \(p \mid f'\) (resp. \(p \mid C'\)), apply Lemma~\ref{lem:cascade_f} or Lemma~\ref{lem:cascade_f_2} if $p=2$ (resp. Lemma~\ref{lem:cascade_y} or~\ref{lem:cascade_y_2} if \(p=2\)) to $\Y_{B',C'}(\Z; \annulus^{t}_{p})$ to reduce to the case \(\Y_{B'(p,t),C'(p,t)}(\Z;\bullet,\annulus^0_{p})\) 
or \(\Y_{B'(p,t),C'(p,t)}(\Z;\bullet,\annulusTwo^0_{p})\), possibly with divisibility restrictions on \(y\) and \(z\) as well.  Add that new subgroupoid to~$\calP$.

\begin{remark}
    This process terminates, since it iterates over primes dividing \(fC\) and for each such prime \(p\), the set \(T_{p}\) is finite. While in principle, this process seems extremely expensive (we could have \(2^{2^{w(fC)}}\) sets to test), in practice, it appears to end much more quickly. 
\end{remark}

\end{enumstep}

\subsection{A family of examples in detail}
\label{subsec:detailedExample}

To illustrate some of the main features of the cascade, we will describe the integral Hasse principle on $\Y_{B,C}$ corresponding to the following family of generalized Fermat equations: 
\begin{equation}\label{eq:intro-ex}
x^2 + 3^{2r}\cdot 83y^{2} = 3^{r'}\cdot 23z^{3},
\end{equation}
where $r$ and $r'$ are positive integers.
Note that if $r > 0$ and $r' > 1$, we can replace $(B,C)$ by $(B/9,C/9)$, so assume without loss of generality that $r'\leq 1$ if $r > 0$. 

First, by Proposition~\ref{prop:localTest}, $\Y_{B,C}$ has $\Z_{p}$-points for all primes $p$; that is, (\ref{eq:intro-ex}) is everywhere locally soluble. We proceed to analyze the Hasse principle for a few curves in this family by computing the set of $\Z$-points. Magma code used in computing the following examples can be found in \url{https://github.com/juanitaduquer/stacky22n/tree/main/Examples} \cite{githubRepo}.

\subsubsection{Coprime case}\label{subsubsec:coprimecase}
When $r = r' = 0$, we get the pair of coefficients $(B,C) = (83,23)$ which are squarefree, coprime integers. Set $K \colonequals \Q\left (\sqrt{-83}\right )$ with ring of integers $\O_{K} = \Z[\alpha]$, where $\alpha = \frac{1 + \sqrt{-83}}{2}$. Notice that $C\O_{K} = \J_{+}\J_{-} = (1+\alpha)(2-\alpha)$ is a product of principal prime ideals. Let $\O \colonequals \Z\left [\sqrt{-83}\right ]$, which has class group $\Cl(\O) \cong \Z/9\Z$. Here, a Magma computation shows that the ideal $23\O$ factors into a pair of non-principal ideals $\J_{+}\cap\O$ and $\J_{-}\cap\O$, both of order $3$ in $\Cl(\O)$. Since any such ideal class is a cube in $\Cl(\O)$, we can write $(\J_{+}\cap\O)\z^{3} = u\O$ for some $u = x + \sqrt{-83}y\in\O$. Then $\Nm(u) = x^{2} + 83y^{2} = 23\Nm(\z)^{3}$. This shows there are primitive\footnote{If $p\mid u$ for some integer prime $p$, replace $u$ and $\z$ by $u/p^{3}$ and $\z/p$, respectively, repeating if necessary.} solutions to this generalized Fermat equation, therefore infinitely many by \cite[Thm.~1.2]{Beukers}. In this example, there are exactly 18 admissible twists detecting all the $\Z$-points on $\Y_{83,23}$ satisfying~$\annulus^0$.

For $r = 2$ and $r' = 0$, we get $(B,C)=(6723,23)$. Set $\O_{9} \colonequals \Z\left [9\sqrt{-83}\right ]$, which has class group $\Cl(\O_{9}) \cong \Z/3\Z\times\Z/18\Z$. For one of the ideals $\J_{+}$ dividing $23\O_{K}$ from above, $[\J_{+}\cap\O_{9}]$ is in $3 \Cl(\O_{9})$, namely it is the class corresponding to $(0,3)\in\Z/3\Z\times\Z/18\Z$. Hence, by the proof of Theorem~\ref{thm:TFAEunified}, there is some $u = x + 9\sqrt{-83}y\in\O_{9}$ producing a primitive $\Z$-solution $\Nm(u) = 23z^{3} = x^{2} + 6723y^{2}$ to~\eqref{eq:intro-ex}. On the descent side, there $6$ admissible twists $\mathcal{C}_{d}'\to(\Y_{6723,23})_{R_K}$ parametrizing infinitely many points in $\Y_{6723,23}(\Z)$.

\subsubsection{Coprime case with no \(\annulus^0\) points}
\
When \(r=0\) and \(r' = 2\), we have \((B,C) = (83,207)\) and we find that that there are no admissible twists. Further, by Lemma \ref{lem:p_divides_u} and Remark~\ref{rem:starp-fc}, \(\Y_{B,C}(\Z; \annulus^0) = \Y_{B,C}(\Z; \annulus_3^0 )\), and thus the latter is empty. However, by Lemma~\ref{lem:cascade_y}, we see that
\[
\Y_{83,207}(\Z; \annulus^1_3) \cong \Y_{83,23}(\Z; \annulus_3^0, 3 \nmid z),
\]
which we find is non-empty, since \(A_{83,23} \neq \emptyset\).

\subsubsection{Non-coprime case}
    For $r = 2$ and $r' = 1$, one can show that $\Y_{6723,69}(\Z;\annulus_{3}^{t}) = \emptyset$ for all $t\not = 2$ (see Lemma~\ref{lem:cascade_nnot|2r-1}). When $t = 2$, there is an isomorphism 
    $$
    \Y_{6723,69}(\Z;\annulus_{3}^{2}) \xlongrightarrow{\sim} \Y_{83,23}(\Z;\annulus_{3}^{0},3\nmid y). 
    $$
    Since we already found such points in this groupoid above, $\Y_{6723,69}(\Z;\annulus_{3}^{2})\not = \emptyset$ and we have solutions to~\eqref{eq:intro-ex}.

    \subsubsection{An example resolved by Lemma~\ref{lem:p-part-kernel}}
    \label{subsec:examples-brute-force}
    For $r = 4$ and $r' = 1$, one has $\Y_{544563,69}(\Z;\annulus_{3}^{t}) = \emptyset$ for all $t\not = 2,4$. When $t = 4$, Lemma~\ref{lem:cascade_nnot|2r-1} gives
    $$
    \Y_{544563,69}(\Z;\annulus_{3}^{4}) \xlongrightarrow{\sim} \Y_{83,{207}}(\Z;\annulus_{3}^{0}), \quad [x : y : z] \longmapsto [x/81 : {y} : {z/27}]. 
    $$
    We saw above that $\Y_{83,{207}}(\Z;\annulus_{3}^{0}) = \emptyset$, so $\Y_{544563,69}(\Z;\annulus_{3}^{4}) = \emptyset$.  
    For $t = 2$, Lemma \ref{lem:cascade_nnot|2r-1} gives instead an isomorphism 
    $$
    \Y_{544563,69}(\Z;\annulus_{3}^{2}) \xlongrightarrow{\sim} \Y_{6723,23}(\Z;\annulus_{3}^{0},3\nmid y), \quad [x : y : z] \longmapsto [{x/9} : {y} : {z/3}]. 
    $$
    Then Lemma~\ref{lem:p-part-kernel} predicts $\Y_{6723,23}(\Z;\annulus_3^0,3\nmid y)\not = \emptyset$; explicitly, 
    \begin{align*}
        \Cl(\O) = \Cl(\Z[3^2\sqrt{-83}]) &\simeq \Z/3\Z \times \Z/18\Z, \\
        \Cl(\O_3) = \Cl(\Z[3^3\sqrt{-83}]) &\simeq \Z/3\Z \times \Z/54\Z,
    \end{align*}
    and $\ker(\eta) \not\subseteq \Cl(\O_3)[3]$ (e.g.\ it contains $(0,6) \in \Z/3\Z \times \Z/54\Z$). Indeed, we have $[61:2:11]\in \Y_{6723,23}(\Z;\annulus_{3}^{0}, 3 \nmid y)$. Note that this is not guaranteed by Proposition~\ref{prop:cascade-with-gcds} since $n=3$; this is explored further in the next example. Therefore $\Y_{544563,69}(\Z;\annulus_{3}^{2})\not = \emptyset$.

\subsubsection{An example with no points resolved by Lemma \ref{lem:p-part-kernel}}
\label{subsubsec:examples-death-case}
Let $r = 3$ and $r' = 1$. By Lemma \ref{lem:cascade_nnot|2r-1} we have
\begin{align*} 
    \Y_{60507,69}(\Z) &= \Y_{60507,69}(\Z; \annulus_3^2) \sqcup \Y_{60507,69}(\Z; \annulus_3^3)\\
    &= \Y_{747,23}(\Z; \annulus_3^0,\ 3 \nmid y) \sqcup \Y_{83,69}(\Z; \annulus_3^0).
\end{align*}
By Lemma~\ref{lem:p_divides_u} and Remark~\ref{rem:starp-fc}, we have that \(\Y_{83,69}(\Z; \annulus^0 )= \Y_{83,69}(\Z; \annulus_3^0)\). Thus an application of Theorem \ref{thm:TFAEunified} shows $\Y_{83,69}(\Z; \annulus_3^0) = \emptyset$.

An initial computer search reveals that there exist points $[x' : y' : z']\in\Y_{747,23}(\Z;\annulus_{3}^{0})$ but they all appear to satisfy $3\mid y'$, which suggests $\Y_{60507,69}(\Z;\annulus_{3}^{2}) = \emptyset$ and therefore there are no primitive solutions to~\eqref{eq:intro-ex} at all when $r = 3,r' = 1$. 
This is confirmed by Lemma \ref{lem:p-part-kernel} as follows. We have
\begin{align*}
    \Cl(\O) = \Cl(\Z[3\sqrt{-83}]) &\simeq \Z/18\Z, \\
    \Cl(\O_3) = \Cl(\Z[3^2\sqrt{-83}]) &\simeq \Z/3\Z \times \Z/18\Z,
\end{align*}
with $q \colon \Cl(\O_3) \to \Cl(\O)$ given by projection. Thus the map $\eta = q \circ [3]$ has kernel precisely $\Cl(\O_3)[3]$. Moreover, writing $23\O_K = \J_+\J_-$ as in \S \ref{subsubsec:coprimecase}, we must have $[\J_+ \cap \O_3] \in 3 \Cl(\O_3)$, since there do exist points in $\Y_{747,23}(\Z; *_3^0,\ 3 \mid y)$. Therefore by Lemma~\ref{lem:p-part-kernel}, $\Y_{747,23}(\Z;\annulus_3^0,3\nmid y) = \emptyset$ and it follows that $\Y_{60507,69}(\Z) = \emptyset$. 

\subsubsection{Summary}

The cases described above are summarized in Table~\ref{tab:detailedExample}. For each $\Y_{B,C}(\Z;\annulus_{3}^{t})$ which is mapped into some $\Y_{B',C'}(\Z)$, the set of admissible twists $\{\mathcal{C}_{d}'\to(\Y_{B',C'})_{R_K}\}$, i.e.~those which detect $\Z$-points on $\Y_{B',C'}$, is denoted by $\admissible_{B',C'}$. 

\begin{table}[ht]
    \caption{Summary of results from \S\ref{subsec:detailedExample} for members of the family in~\eqref{eq:intro-ex}.}
    \label{tab:detailedExample}
    \centering
\begin{tabular}{|c|c|c|c|c|l|}
    \hline
    $(r,r')$ & $(B,C)$      & $t$ & $(B',C')$   & \((\annulus^0)\)-points?   & \multicolumn{1}{|c|}{Reason}\\
    \hline
    $(0,0)$ & $(83,23)$     & $0$ & $(83,23)$   & yes       & $\admissible_{83,23}\not = \emptyset$ (Theorem~\ref{thm:TFAEunified})\\
    \hline
     $(0,2)$ & $(83,207)$ & $0$ & $(83,207)$   & no        & $\admissible_{83,207} = \emptyset$ (Theorem~\ref{thm:TFAEunified})\\
            &               & $1$ & $(83,23)$  & yes       & $\admissible_{83,23}\not = \emptyset$ (Theorem~\ref{thm:TFAEunified}) and\\
            &               &     &             &           & $\Y_{83,23}(\Z;\annulus_{3}^{0},3\nmid z)\not = \emptyset$ (Lemma~\ref{lem:cascade_y})\\
            \hline
    $(2,0)$ & $(6723,23)$   & $0$ & $(6723,23)$ & yes       & $\admissible_{6723,23}\not = \emptyset$ (Theorem~\ref{thm:TFAEunified})\\
    \hline
    $(2,1)$ & $(6723,69)$   & $2$ & $(83,23)$   & yes       & $\admissible_{83,23}\not = \emptyset$ (Theorem~\ref{thm:TFAEunified}) and\\
            &               &     &             &           & $\Y_{83,23}(\Z;\annulus_{3}^{0},3\nmid y)\not = \emptyset$ (Lemma~\ref{lem:cascade_nnot|2r-1})\\
            \hline
    $(4,1)$ & $(544563,69)$ & $4$ & $(83,{207})$   & no        & $\admissible_{83,{207}} = \emptyset$ (Theorem~\ref{thm:TFAEunified})\\
            &               & $2$ & $(6723,23)$  & yes       & $\admissible_{6723,23}\not = \emptyset$ (Theorem~\ref{thm:TFAEunified}) and\\
            &               &     &             &           & $\Y_{6723,23}(\Z;\annulus_{3}^{0},3\nmid y)\not = \emptyset$ (Lemmas~\ref{lem:cascade_nnot|2r-1}, \ref{lem:p-part-kernel})\\
            \hline
    $(3,1)$ & $(60507,69)$  & $3$ & $(83,69)$   & no        & $\admissible_{83,69} = \emptyset$ (Theorem~\ref{thm:TFAEunified})\\
            &               & $2$ & $(747,23)$  & no & $\Y_{747,23}(\Z;\annulus_{3}^{0},3\nmid y) = \emptyset$ (Lemma~\ref{lem:p-part-kernel})\\
    \hline
\end{tabular}
\end{table}

\begin{remark}
    Replacing $p = 3$ with $p = 2$ in equation (\ref{eq:intro-ex}), one can similarly analyze a family of stacky curves for which Theorem~\ref{thm:TFAEunified_prime} and Lemma~\ref{lem:p2:cascade_nnot|2r-1} are required. Note that some of these curves will fail to have $\Z_{2}$-points. 
\end{remark}

\subsection{Arbitrary coefficients}\label{subsec:Acoeff}

The cascade strategy also allows us to determine the existence of global points on stacky curves $\Y_{A,B,C}$ constructed similarly as in~\eqref{eq:Y_BC}, corresponding to solutions to 
\begin{equation*}
    Ax^{2} + By^{2} = Cz^{n}
\end{equation*}
for non-zero squarefree $A$. For a general integer \(A\), write $A =A_0 A_\Box^2$ with $A_0$ squarefree. There is an isomorphism of groupoids
 \begin{align*}
        \Y_{A,B,C}(\Z[1/A_{\Box}]) &\to \Y_{1,A_0B,A_0C}(\Z[1/A_\Box])\\
        [x:y:z] &\mapsto [A_0A_\Box x : y : z],\\
        \sizedbracket{\frac{x'}{A_0A_{\Box}} : y' : z'} & \mapsfrom [x' : y' : z'].
    \end{align*}
Indeed, if \(p \nmid A_{\Box}\) divides \(y\) and \(z\), then \(p^2 \mid Ax^2\). Since \(p\) does not divide \(A_{\Box}\) and \(A_0\) is squarefree, \(p\) must divide \(x\), which contradicts primitivity on the source. The inverse map is well-defined, since we can change the representative \([x': y':z']\) to \([A_{\Box}^nx': A_{\Box}^n y': A_{\Box}^2z']\).
Thus, when \(A\) is squarefree, i.e. \(A_{\Box} =1\), this map induces an isomorphism of groupoids of \(\Z\) points and we can use the methods of this paper to decide non-emptiness. 

If \(A_{\Box} \neq 1\), the problem becomes trickier and requires the ability to detect points with restricted \(x\)-coordinates, similar to the case discussed in \S\ref{subsec:points-with-restricted-y-coordinate}. We do not do this here, but we anticipate that methods from Sections~\ref{sec:mainthm} and~\ref{sec:cascade} can be generalized to these cases.

\section{Statistics}
\label{sec:statistics}

Maintaining the notation from the previous sections, let $n$ be an odd integer.  For $T>0$, define counting functions
\begin{align*}
    N_n(T) &\colonequals \#\{(B,C) \in \Z^2 : |B|, |C| < T,\ \Y_{B,C}(\Z) \neq \emptyset \},\\
    N^{\loc}_n(T) &\colonequals \#\{(B,C) \in \Z^2 : |B|, |C| < T,\ \Y_{B,C}(\Z_p) \neq \emptyset \text{ for all primes } p\}.
\end{align*}
These quantities capture how often the stacky curve $\Y_{B,C}$ has integral points and integral points everywhere locally, respectively. In this section we describe and compare their rates of growth, to get at how often $\Y_{B,C}$ satisfies the integral Hasse principle.

\begin{theorem}
\label{thm:local_stats}
    For all odd $n$ we have
    \[N_n^\loc(T) \leq \frac{4\pi^4 \constFI T^2}{9(1-\alpha)^{1/2}\sqrt{\log T}} + o\left(\frac{T^2}{\sqrt{\log T}}\right),\]
    where $\constFI > 0$ is a positive constant and $\frac12 < \alpha < 1$. On the other hand, we have 
    \[N_n^{\loc}(T) \geq \frac{16 \constFI \constTaub T^2}{5\sqrt{\log T}} + o\left(\frac{T^2}{\sqrt{\log T}}\right),\]
    where $\constTaub > 0$ is a positive constant.
\end{theorem}

We observe that although the constants in the upper and lower bounds presented in Theorem~\ref{thm:local_stats} do not depend on \(n\), we expect that an asymptotic formula for \(N_n^{\loc}(T)\) would indeed depend on it. Numerical estimates for $\constFI$ and $\constTaub$ are given in~\eqref{eq:constFI} and~\eqref{eq:constTaub}, respectively.

\begin{remark}
    \label{remark:local-densities}
    Theorem~\ref{thm:local_stats} implies that $N_n^\loc(T)/T^2 \to 0$ as $T \to \infty$, or 0\% of $\Y_{B,C}$ are everywhere locally soluble. This can also be observed by calculating local densities at each prime. Indeed let $\mu_p$ denote the $p$-adic Haar measure on $\Z_p^2$ normalized so $\mu_p(\Z_p^2) = 1$. Then a straightforward calculation using Proposition~\ref{prop:localTest} shows that for each odd prime \(p\)
\[\mu_p\left( \left\{(B,C) \in \Z_p^2 : \Y_{B,C}(\Z_p)\neq\emptyset \right\} \right)
=1 -  \frac{p^n + p^{n-2} + 2p^{n-3} - p + 2}{2p^{n-3}(p+1)(p^3+p^2 + p + 1)}.
\]
In particular, the product of these local factors over all primes $p$ diverges to 0. 
\end{remark}

On the global side of things, we have the following result for the $n=3$ case.
\begin{theorem}
\label{thm:pos_prop_HP_n=3}
    We have
    \[\liminf_{T \to \infty} \frac{N_3(T)}{N_3^\loc(T)} > 0.00988.\]
\end{theorem}
That is, when $n=3$, a positive proportion of everywhere locally soluble $\Y_{B,C}$ satisfy the Hasse principle for integral points. The approach of the proof is to bound $N_3(T)$ from below by considering only those squarefree $B > 0$ for which the imaginary quadratic extension $K=\Q(\sqrt{-B})$ has class group with trivial 3-part. Assuming the average rank of the $n$-part of $\Cl(\Q(\sqrt{\pm B}))$ (counted by $|B|$ rather than the discriminant) agrees with the prediction supplied by the Cohen--Lenstra heuristics (see Assumption~\ref{assumption:CL}), we can extend this argument to all odd primes $n$.
\begin{theorem}
\label{thm:pos_prop_HP_n>3}
	Let $n > 3$ be an odd prime and suppose Assumption~\ref{assumption:CL} holds. Then
	\[\liminf_{T \to \infty} \frac{N_n(T)}{N_n^\loc(T)} > 0.\]
\end{theorem}

\subsection{Proof of Theorem~\ref{thm:local_stats}}
\label{subsec:proof-of-local-stats}
\subsubsection{Upper bound}
\label{subsubsec:local-stats-upper-bound}
    For the upper bound, it suffices to count a weaker condition than that of Proposition~\ref{prop:localTest}: notice that if $\Y_{B,C}(\Z_p) \neq \emptyset$ for $p$ odd, we have that if $p$ divides the squarefree part of $C$, then either $p\mid B$ or $\left(\frac{-B}{p}\right) = 1$. Write $C = C_1C_2C_\Box^2$, where $C_1C_2$ is squarefree, such that $C_1 \mid B$ and $\gcd(C_2, B) = 1$. Note that while there are many ways of writing the squarefree part of \(C\) as a product of \(2\) integers, for a fixed \(B\), there is a unique way to factor it into a part that divides \(B\) and one coprime to it.
    
    We may assume without loss of generality that \(C_{\Box}, C_2 > 0\). We have
    \begin{align}
        \nonumber N_n^{\loc}(T) &\leq \sum_{1 \leq C_\Box \leq \sqrt{T}} \sum_{\substack{|C_1| \leq \frac{T}{C_\Box^2} \\ C_1 \text{ squarefree}}} \sum_{\substack{1 \leq C_2 \leq \frac{T}{C_\Box^2C_1}\\ C_2 \text{ squarefree}\\ \gcd(C_1, C_2) = 1}} \sum_{\substack{|B| \leq T \\ C_1 \mid B \\ \gcd(B,C_2) = 1}} \mathds{1}_{p \mid C_1C_2 \implies \left(\frac{-B}{p}\right) \in \{0,1\}}\\
        \label{eq:bigsum} &= \sum_{1 \leq C_\Box \leq \sqrt{T}} \sum_{{|C_1| \leq \frac{T}{C_\Box^2}}} \mu(C_1)^2 \sum_{\substack{1 \leq C_2 \leq \frac{T}{C_\Box^2C_1} \\ \gcd(C_1, C_2) = 1}} \mu(C_2)^2 \sum_{\substack{|B'| \leq T/C_1 \\ \gcd(B',C_2) = 1}} \prod_{\substack{p \mid C_2 \\ p \text{ odd}}} \frac12 \left( 1 + \left(\frac{-C_1B'}{p}\right)\right),
    \end{align}
    where $\mu$ denotes the M\"obius function. First, we argue that the contribution from $(B,C)$ pairs with $C_1C_\Box^2$ large is negligible.
    \begin{lemma}
    \label{lem:restricting-parameters}
        Let \(T, B, B', C, C_1, C_2, C_{\Box}\) be as above. Let \(0< U<T\) be a real parameter. Then,
        \[
        \sum_{U< |C_1 C_{\Box}^2| <T} \; \sum_{1 \le C_2 \le \frac{T}{C_1C_{\Box}^2}} \; \sum_{|B'| \le T/C_1} 1 \ll \frac{T^{5/2}}{U}, 
        \]
        where the implied constant is absolute. In particular, if \(U = T^{\frac{1}{2} + \epsilon}\) for some \(\epsilon >0\), then the above sum is \(o(T^2/\sqrt{\log(T)})\).
    \end{lemma}
    \begin{proof}
        The proof is a standard application of the Abel summation formula. 
    \end{proof}
    
    Moving forward, we look to bound~\eqref{eq:bigsum} with $1 \leq C_\Box \leq \sqrt{U}$ and $|C_1| \leq U/C_\Box^2$. Expanding the innermost product in~\eqref{eq:bigsum}, we have
    \begin{align*}
        \prod_{\substack{p \mid C_2 \\ p \text{ odd}}} \frac12 \left( 1 + \left(\frac{-C_1B'}{p}\right)\right) 
        &= \frac{1}{2^{\omega(C_2) - v_2(C_2)}} \sum_{\substack{D \mid C_2\\ D \text{ odd}}} \left(\frac{-C_1B'}{D}\right),      
    \end{align*}
    where $\omega$ denotes the arithmetic function which counts distinct prime divisors. We will treat the contributions from $D=1$ and $D > 1$ separately. 

    Fix an odd $D > 1$ such that \(D \mid C_2\). By the Polya--Vinogradov theorem \cite[Thm 13.15]{Apostol76}, we have
    \[\sum_{|B'| \leq T/C_1} \left(\frac{B'}{D}\right) \ll \sqrt{D} \log(D),\]
   where the implied constant is absolute. The relevant terms from~\eqref{eq:bigsum} become
 \begin{align*}
        & \sum_{1 \leq C_\Box \leq \sqrt{U}} \sum_{{|C_1| \leq \frac{U}{C_\Box^2}}} \mu(C_1)^2 \sum_{\substack{1 \leq C_2 \leq \frac{T}{C_\Box^2C_1} \\ \gcd(C_1, C_2) = 1}} \frac{\mu(C_2)^2}{2^{\omega(C_2) - v_2(C_2)}} \sum_{\substack{D \mid C_2\\ D>1 \text{ odd}}} \left(\frac{-C_1}{D}\right) \sum_{\substack{|B'| \leq T/C_1 \\ \gcd(B',C_2) = 1}} \left(\frac{B'}{D}\right) \\
        \ll& \sum_{1 \leq C_\Box \leq \sqrt{U}} \sum_{{|C_1| \leq \frac{U}{C_\Box^2}}} \mu(C_1)^2 \sum_{\substack{1 \leq C_2 \leq \frac{T}{C_\Box^2C_1} \\ \gcd(C_1, C_2) = 1}} \frac{\mu(C_2)^2}{2^{\omega(C_2)}} \sum_{\substack{ D \mid C_2\\ D>1 \text{ odd}}} \sqrt{D} \log(D) \\
        \ll& \sum_{1 \leq C_\Box \leq \sqrt{U}} \sum_{{|C_1| \leq \frac{U}{C_\Box^2}}} \mu(C_1)^2 \sum_{\substack{1 \leq E \leq \frac{T}{C_\Box^2C_1} \\ \gcd(C_1, E) = 1}}  \sum_{\substack{1 < D  < \frac{T}{C_{\Box}^2C_1 E} \\ \gcd(D, C_1E)=1} }  \mu(E)^2\mu(D)^2 \sqrt{D} \log(D) \\
         \ll& \sum_{1 \leq C_\Box \leq \sqrt{U}} \sum_{{|C_1| \leq \frac{U}{C_\Box^2}}} \mu(C_1)^2 \sum_{\substack{1 \leq E \leq \frac{T}{C_\Box^2C_1} \\ \gcd(C_1, E) = 1}}  \mu(E)^2 \sum_{1 < D < \frac{T}{C_{\Box}^2 C_1 E}} \frac{T^{1/2}\log(T)}{C_{\Box}\sqrt{C_1 E}}\\
         \ll& T^{3/2} \log(T) \sum_{1 \leq C_\Box \leq \sqrt{U}} \sum_{{|C_1| \leq \frac{U}{C_\Box^2}}} \frac{\mu(C_1)^2}{C_{\Box}^3 C_1^{3/2}} \sum_{\substack{1 \leq E \leq \frac{T}{C_\Box^2C_1} \\ \gcd(C_1, E) = 1}}  \frac{\mu(E)^2}{E^{3/2} } \\
        \ll & T^{3/2}\log(T).
    \end{align*}

    Suppose now that $D=1$. It suffices to bound the relevant terms of~\eqref{eq:bigsum} while ignoring some of the coprimality conditions:
    \begin{align*}
        & \sum_{1 \leq C_\Box \leq \sqrt{U}} \sum_{{|C_1| \leq \frac{U}{C_\Box^2}}} \mu(C_1)^2 \sum_{\substack{1 \leq C_2 \leq \frac{T}{C_\Box^2C_1} \\ \gcd(C_1, C_2) = 1}} \frac{\mu(C_2)^2}{2^{\omega(C_2) - v_2(C_2)}}\sum_{\substack{|B'| \leq T/C_1 \\ \gcd(B',C_2) = 1}} 1 \\
        \leq & 2T \sum_{1 \leq C_\Box \leq \sqrt{U}} \sum_{{|C_1| \leq \frac{U}{C_\Box^2}}} \frac{\mu(C_1)^2}{C_1} \sum_{\substack{1 \leq C_2 \leq \frac{T}{C_\Box^2C_1} \\ \gcd(C_1, C_2) = 1}} \frac{\mu(C_2)^2}{2^{\omega(C_2) - v_2(C_2)}} \\
        \leq & 4T \sum_{1 \leq C_\Box \leq \sqrt{U}} \sum_{{|C_1| \leq \frac{U}{C_\Box^2}}} \frac{\mu(C_1)^2}{C_1} \sum_{1 \leq C_2 \leq \frac{T}{C_\Box^2C_1}} \frac{\mu(C_2)^2}{2^{\omega(C_2)}} \\
        \leq & 4\constFI T \sum_{1 \leq C_\Box \leq \sqrt{U}} \sum_{{|C_1| \leq \frac{U}{C_\Box^2}}} \left( \frac{\mu(C_1)^2 T}{C_1^2C_\Box^2 \sqrt{\log\left(\frac{T}{C_1C_\Box^2}\right)}} + \frac{1}{C_1} E\left( \frac{T}{C_\Box^2 C_1}\right) \right).
    \end{align*}
    The last line follows from \cite[Lemma 1]{FriedlanderIwaniec} (applied with \(d=q=1\) and \(\chi\) trivial) where $\constFI$ is an explicit constant (see \S\ref{subsubsec:local-stats-constants}), and \(E\) is an error term that we will see shortly can be neglected. Continuing with the main term of the upper bound, we have
    \begin{align*}
        4\constFI T^2 \sum_{1 \leq C_\Box \leq \sqrt{U}} \sum_{{|C_1| \leq \frac{U}{C_\Box^2}}} \frac{\mu(C_1)^2}{C_1^2C_\Box^2 \sqrt{\log\left(\frac{T}{C_1C_\Box^2}\right)}} 
        & \leq \frac{4\constFI T^2}{\sqrt{\log (T/U)}} \sum_{1 \leq C_\Box \leq \sqrt{U}} \frac{1}{C_\Box^2} \sum_{{|C_1| \leq \frac{U}{C_\Box^2}}} \frac{\mu(C_1)^2}{C_1^2} \\ 
        & \leq \frac{8 \zeta(2) \constFI T^2}{\sqrt{\log (T/U)}} \sum_{1 \leq C_\Box \leq \sqrt{U}} \frac{1}{C_\Box^2} \\ 
        & \leq \frac{16 \zeta(2)^2 \constFI T^2}{\sqrt{\log (T/U)}} = \frac{4 \pi^2 \constFI T^2}{9 \sqrt{\log (T/U)}}.
    \end{align*}
    Setting $U = T^\alpha$ for $\frac12 < \alpha < 1$, we have
    \[\frac{4 \pi^2 \constFI T^2}{9 \sqrt{\log (T^{1-\alpha})}} = \frac{4 \pi^2 \constFI T^2}{9 (1-\alpha)^{1/2} \sqrt{\log T}},\]
    giving the claimed constant in the main term of the upper bound. This is minimized by taking $\alpha$ close to $1/2$. 
    
    The error term introduced by \cite[Lemma 1]{FriedlanderIwaniec} is
    \[E(X) = O\left(\frac{X}{(\log X)^{3/2}}\right) + O_\lambda \left(X(\log X)^{-\lambda}\right),\]
    where $\lambda > 0$ is any positive constant. Taking \(\lambda \ge 3/2\), a straightforward calculation shows
    \[\sum_{1 \leq C_\Box \leq \sqrt{U}} \sum_{{|C_1| \leq \frac{U}{C_\Box^2}}} \frac{1}{C_1} E\left(\frac{T}{C_\Box^2C_1}\right) \ll \frac{T}{\log(T/U)^{3/2}} \sum_{1 \leq C_\Box \leq \sqrt{U}} \sum_{{|C_1| \leq \frac{U}{C_\Box^2}}} \frac{1}{C_{\Box}^2C_1^2} =   o \left( \frac{T}{\sqrt{\log T}}\right),\]
    so this does not overtake the main term for \(U = T^{\alpha}\). We have thus established the upper bound of Theorem~\ref{thm:local_stats}.

    \subsubsection{Lower bound}
    \label{subsubsec:local-stats-lower-bound}
    For the lower bound, we may consider $B$ squarefree and $C$ odd, squarefree, and coprime to $B$. In this case, the criterion for $\Y_{B,C}(\Z_p) \neq \emptyset$ for all primes $p$ is precisely $p \mid C \implies \left(\frac{-B}{p}\right) = 1$. That is, we wish to bound from below
    \begin{equation}\label{eq:lowerboundsum}
        N_n^{\loc}(T) \geq \sum_{\substack{|B| \leq T \\ B \text{ squarefree}}} \sum_{\substack{|C| \leq T \\ \gcd(2B,C)=1}} \frac{\mu(C)^2}{2^{\omega(C)}} \prod_{p \mid C} \left( 1 + \left(\frac{-B}{p}\right)\right).
    \end{equation}
    Expanding the product, we again find that the contribution from the nontrivial Jacobi symbols is $o(T^2/\sqrt{\log T})$, hence negligible. 

    Applying \cite[Lemma 1]{FriedlanderIwaniec}, we have
    \begin{align*}
        \sum_{\substack{|B| \leq T \\ B \text{ squarefree}}} \sum_{\substack{|C| \leq T \\ \gcd(2B,C)=1}} \frac{\mu(C)^2}{2^{\omega(C)}} & =
        \frac{2\constFI T}{\sqrt{\log T}} \sum_{|B| \leq T} \left(\mu(B)^2 \prod_{p \mid 2B} \left( 1 + \frac1{2p}\right)^{-1}\right) + o\left(\frac{T^2}{\sqrt{\log T}}\right) \\
        &= \frac{16 \constFI \constTaub T^2}{5\sqrt{\log T}} + o\left(\frac{T^2}{\sqrt{\log T}}\right),
    \end{align*}
    where the last line follows by Lemma~\ref{lem:standard_tauberian} below. 
    
\begin{lemma}\label{lem:standard_tauberian}
    There exists a positive constant $\constTaub$ such that
    \[\sum_{1 \leq B < T} \mu(B)^2\prod_{p\mid B}\left(1 + \frac1{2p}\right)^{-1} \sim \constTaub T.\] 
\end{lemma}

\begin{proof}
    Consider the Dirichlet series given by
    \begin{align*}
       L(s) =  \sum_{B > 0} \frac{\mu(B)^2}{B^{s}} \prod_{p \mid B} \left(1 + \frac{1}{2p}\right)^{-1} &=  \sum_{B > 0} \mu(B)^2 \prod_{p \mid B} \left(1 + \frac{1}{2p}\right)^{-1}  p^{-s}\\
       &= \prod_{p} \left(1 + \left(1 + \frac{1}{2p}\right)^{-1} p^{-s}\right)\\
       &= \prod_{p} \left( 1 + p^{-s} \left( 1 - \frac{1}{2p} + \frac{1}{(2p)^2} \ldots \right) \right)\\
       &= \prod_p \left( 1 + p^{-s} - \frac1{2}p^{-s-1}\sum_{i \geq 0} \left(\frac{-1}{2p}\right)^i \right).
    \end{align*}
    Now note that 
    \[L(s) \zeta(s)^{-1} = L(s) \prod_p (1-p^{-s}) = \prod_p \left(1 - p^{-2s} - \frac12 p^{-s-1}(1 - p^{-s}) \sum_{i \geq 0} \left(\frac{-1}{2p}\right)^{i} \right) = M(s), \]
    which converges on the half-plane \(\mathrm{Re}(s) > 1/2\). Thus $L(s)$ converges for $\mathrm{Re}(s) > 1$ with a simple pole at $s=1$ and residue $\constTaub = M(1)$. The lemma follows from a standard application of a Tauberian theorem (see e.g.\ \cite{tauberian-ref} or \cite[Chapter 7]{bateman-diamond}).
   \end{proof}
   
  \subsubsection{Evaluation of the constants}
  \label{subsubsec:local-stats-constants}
    It is useful to estimate the constants $\constFI, \constTaub$. The constant $\constFI$ is given in \cite[Lemma 1]{FriedlanderIwaniec} by
    \[\constFI = \frac1{\sqrt{\pi}} \prod_p \left(1 + \frac1{2p}\right) \left(1 - \frac1p\right)^{1/2}.\]
    This is equivalent to 
    \[{\pi} \constFI^2 =  \prod_p \left(1 + \frac1{2p}\right)^2 \left(1 - \frac1p\right) = \prod_p \left(1 - \frac{3}{4p^2} - \frac1{4p^3}\right).\]

Since the Euler factors are decreasing in \(p\), we immediately get an upper bound by truncating the product after finitely many primes. On the other hand, observe that
\[
\left(1 - \frac{3}{4p^2} - \frac{1}{4p^3} \right) - \left( 1 - \frac{1}{p^2} \right) = \frac{1}{4p^2} - \frac{1}{4p^3} > 0.
\]
In particular, for any real \(X>0\), 
\[
\pi \constFI^2 \ge \prod_{p \le X}  \left(1 - \frac{3}{4p^2} - \frac{1}{4p^3} \right)\frac{\zeta(2)^{-1}}{ \prod_{p \le X}(1 - p^{-2})}.
\]
Computing this for \(X=50,000\), we find
\begin{align}
        \label{eq:constFI} 0.4581814 &\leq \constFI \leq 0.4581819.
\end{align}

For \(\constTaub = M(1)\) from Lemma~\ref{lem:standard_tauberian}, we use the same idea (the Euler factors are still decreasing), now comparing it to \(\zeta(1.95)\) instead of \(\zeta(2)\). This gives us
\begin{align}
        \label{eq:constTaub}  0.526859 &\leq \constTaub \leq 0.526861.
\end{align}

In particular, taking \(\alpha = 1/2 + 10^{-5}\), we see that
\[\liminf_{T \to \infty} \frac{N_n^\loc(T)}{\frac{T^2}{\sqrt{\log T}}} \leq \frac{4\pi^4 \constFI}{9(1-\alpha)^{1/2}} \le 28.05267255,\] and
 \[\limsup_{T \to \infty} \frac{N_n^\loc(T)}{\frac{T^2}{\sqrt{\log T}}} \geq  \frac{16 \constFI \constTaub}{5} \ge 0.772470381.\]

\subsection{Lower bounds for global points}
\label{subsec:global-statistics}

In this section, we prove Theorems~\ref{thm:pos_prop_HP_n=3} and~\ref{thm:pos_prop_HP_n>3}. We begin with some notation. Let $n$ be an odd prime and define $\delta_n^{\real}, \delta_n^{\imag}$ to be the density of squarefree $B$ for which $\Cl(\Q(\sqrt{B}))[n]$ is nontrival, given by the limits
\begin{align*}
    \delta_n^{\imag} &\colonequals \lim_{T \to \infty}\frac{\#\{1 \leq B \leq T : B \text{ squarefree}, \Cl(\Q(\sqrt{-B}))[n] \neq 0\}}{\#\{1 \leq B \leq T : B \text{ squarefree}\}}, \\
    \delta_n^{\real} & \colonequals \lim_{T \to \infty}\frac{\#\{1 \leq B \leq T : B \text{ squarefree}, \Cl(\Q(\sqrt{B}))[n] \neq 0\}}{\#\{1 \leq B \leq T : B \text{ squarefree}\}}.
\end{align*}
Davenport and Heilbronn showed in the family of imaginary (resp.\ real) quadratic fields $K = \Q(\sqrt{-B})$ counted by discriminant, the average size of $ {\Cl(\O_K)}[3]$ is 2 (resp.\ $\frac43$) \cite[Thm~3]{DavenportHeilbronn}. It follows from the fact that these averages are well behaved in residue classes \cite[Thm.~1]{NakagawaHorie} that the same average holds when counting instead by the squarefree $B$ instead of the discriminant.

For an odd prime $n > 3$, the Cohen--Lenstra heuristics (\cite[Section 9]{CohenLenstra84}) predict that the average size of ${\Cl(\O_K)}[n]$, when ordered by discriminant, is \(2\) for imaginary quadratic fields and ($1 + \frac1n$) for real quadratic fields. It is expected that these averages also behave well in residue classes of discriminants, leading to the same values when ordered by squarefree \(B\) as by discriminant. We record this as Assumption~\ref{assumption:CL} below.

\begin{assumption}[Modified Cohen--Lenstra]
\label{assumption:CL}
    Let \(n >3\) be an odd prime. Then 
   \begin{align*}
    \sum_{\substack{1 \le B \le T\\ B \text{ squarefree}}} \# \Cl(\Q(\sqrt{-B}))[n] &\sim \frac{2T}{\zeta(2)}\\
        \sum_{\substack{1 \le B \le T\\ B \text{ squarefree}}} \# \Cl(\Q(\sqrt{B}))[n] &\sim \frac{n+1}{n} \cdot \frac{T}{\zeta(2)}\\
   \end{align*}
\end{assumption}

\begin{lemma}
\label{lem:statistics-cohen-lenstra-delta-bounds}
    For \(n \ge 3\) an odd prime, under Assumption~\ref{assumption:CL} if $n > 3$, we have
  \begin{align*}
      \delta_n^\imag \leq \frac{1}{n-1}, \quad \text{ and } \quad \delta_n^\real \leq \frac{1}{n(n-1)}.
  \end{align*}
\end{lemma}

\begin{proof}
For the imaginary case, observe that,
    \begin{align*}
   \nonumber 2 &= \lim_{T \to \infty} \Bigg(\sum_{r \geq 0} \sum_{\substack{1 \leq B \leq T \\ B \text{ squarefree}\\ \mathrm{rk}_n \Cl(\Q(\sqrt{-B})) = r}} n^r\Bigg)\Bigg( \sum_{\substack{1 \leq B \leq T \\ B \text{ squarefree}}} 1\Bigg)^{-1} \\
   &\geq \lim_{T \to \infty} \Bigg({ \sum_{\substack{1 \leq B \leq T \\ B \text{ squarefree}\\ \Cl(\Q(\sqrt{-B}))[n]=0}} 1 + \sum_{\substack{1 \leq B \leq T \\ B \text{ squarefree}\\ \mathrm{rk}_n \Cl(\Q(\sqrt{-B}))[n] \neq 0}} n}\Bigg)\Bigg({ \sum_{\substack{1 \leq B \leq T \\ B \text{ squarefree}}} 1}\Bigg)^{-1} \\
   \nonumber & = 1 - \delta_n^\imag + n\delta_n^\imag,
\end{align*}
which shows that \(\delta_n^\imag \leq \frac{1}{n-1}.\) The proof for \(\delta_n^\real\) follows similarly.
\end{proof}

In order to give a lower bound on \(N_n(T)\), 
we will be interested in $\Cl(\Z[\sqrt{-B}])$ for \(B\) squarefree. Since this order is not always maximal, we define $\widetilde{\delta_n}^\imag$ (resp.\ $\widetilde{\delta_n}^\real$) to be the density of positive (resp.\ negative) squarefree integers $B$ such that $\Cl(\Z[\sqrt{-B}])[n] \neq 0$. By Lemma~\ref{lem:class_number_formula}, we have $\Cl(\Q(\sqrt{-B}))[n] = \Cl(\Z[\sqrt{-B}])[n]$ for all odd primes $n > 3$, so in this case we have $\widetilde{\delta_n}^\imag = \delta_n^\imag$ and $\widetilde{\delta_n}^\real = \delta_n^\real$.

When $n = 3$, an application of Lemma~\ref{lem:class_number_formula} shows $\Cl(\Q(\sqrt{-B}))[3] = \Cl(\Z[\sqrt{-B}])[3]$ whenever $B \not\equiv 3 \pmod{8}$; when $B \equiv 3 \pmod{8}$ we have $\Cl(\Z[\sqrt{-B}])[3] \neq 0$. By \cite{NakagawaHorie}, the average ranks are well behaved in residue classes, and we have
\begin{align}
    \label{eq:delta_tilde_3_imag}
    \widetilde{\delta}_3^{\imag} &= \lim_{T \to \infty} \Bigg({\sum_{\substack{1 \leq B \leq T \\ B \text{ squarefree} \\ B \not \equiv 3 \pmod{8} \\ \Cl(\Z[\sqrt{-B}])[n] \neq 0}} 1 + \sum_{\substack{1 \leq B \leq T \\ B \text{ squarefree} \\ B \equiv 3 \pmod{8}}} 1}\Bigg)\Bigg({\sum_{\substack{1 \leq B \leq T\\ B \text{ squarefree}}} 1}\Bigg)^{-1} = \frac{5}{6}\delta_3^\imag + \frac16 \leq \frac7{12},   \\
    \label{eq:delta_tilde_3_real}
    \widetilde{\delta}_3^{\real} &= \lim_{T \to \infty} \Bigg({\sum_{\substack{1 \leq B \leq T \\ B \text{ squarefree} \\ -B \not \equiv 3 \pmod{8} \\ \Cl(\Z[\sqrt{B}])[n] \neq 0}} 1 + \sum_{\substack{1 \leq B \leq T \\ B \text{ squarefree} \\ -B \equiv 3 \pmod{8}}} 1}\Bigg)\Bigg({\sum_{\substack{1 \leq B \leq T\\ B \text{ squarefree}}} 1}\Bigg)^{-1} = \frac{5}{6}\delta_3^\real + \frac16 \leq \frac{11}{36}.   
\end{align}

\begin{proof}[Proof of Theorem~\ref{thm:pos_prop_HP_n=3}]
    As in the proof of Theorem~\ref{thm:local_stats}, we achieve a lower bound by restricting our attention to $B$ and $C$ squarefree, with $C$ odd. Let
    \begin{align*}
        S^{\imag} &\colonequals \left\{ (B,C) \in \Z^2 : B \geq 1, \gcd(2B,C) = 1, BC \text{ squarefree}, \Cl(\Z[\sqrt{-B}])[3] = 0 \right\},\\
        S^{\real} &\colonequals \left\{ (B,C) \in \Z^2 : B \leq -1, \gcd(2B,C) = 1, BC \text{ squarefree}, \Cl(\Z[\sqrt{-B}])[3] = 0\right\},
    \end{align*}
    and define for $T > 0$
    \begin{align*}
        N_3^{\imag}(T) &\colonequals \#\left\{ (B,C) \in S^{\imag} : \max\{|B|,|C|\} \leq T,  \Y_{B,C}(\Z_p) \neq \emptyset \text{ for all }p\right\},\\
        N_3^{\real}(T) &\colonequals \#\left\{ (B,C) \in S^{\real} : \max\{|B|,|C|\} \leq T,  \Y_{B,C}(\Z_p) \neq \emptyset \text{ for all }p\right\}.
    \end{align*}
    We have
    \[N_3(T) \geq N_3^{\imag}(T) + N_3^{\real}(T),\]
    since when $\Cl(\Z[\sqrt{-B}])[3]$ is trivial, there is no obstruction to the Hasse principle by Theorem~\ref{thm:TFAEunified}. We can then produce lower bounds for $N_3^\imag(T)$ and $N_3^\real(T)$ by modifying our local solubility counts from the proof of Theorem~\ref{thm:local_stats}.

    Starting in the same manner as~\eqref{eq:lowerboundsum}, we have
    \begin{align*}
        N_3^{\imag}(T) &\geq \sum_{\substack{1 \leq B \leq T\\ B \text{ squarefree}\\ \Cl(\Z[\sqrt{-B}])[3] = 0}} \sum_{\substack{|C| \leq T \\ \gcd(2B,C) = 1}} \frac{\mu(C)^2}{2^{\omega(C)}} \prod_{p \mid C} \left( 1 + \left(\frac{-B}{p}\right) \right)\\ 
        & = \sum_{\substack{1 \leq B \leq T\\ B \text{ squarefree}\\ \Cl(\Z[\sqrt{-B}])[3] = 0}} \sum_{\substack{|C| \leq T \\ \gcd(2B,C) = 1}} \frac{\mu(C)^2}{2^{\omega(C)}} \sum_{D \mid C} \left(\frac{-B}{D} \right) \\
        & = \sum_{\substack{1 \leq B \leq T\\ B \text{ squarefree}\\ \Cl(\Z[\sqrt{-B}])[3] = 0}} \sum_{\substack{|C| \leq T \\ \gcd(2B,C) = 1}} \frac{\mu(C)^2}{2^{\omega(C)}}  +  \sum_{|C| \leq T, \text{ odd}} \frac{\mu(C)^2}{2^{\omega(C)}} \sum_{\substack{D \mid C\\ D>1}} \sum_{\substack{1 \leq B \leq T\\ B \text{ squarefree}\\ \Cl(\Z[\sqrt{-B}])[3] = 0\\ \gcd(B,C)=1}} \left(\frac{-B}{D} \right) \\
        & = \frac{2\constFI T}{\sqrt{\log T}} \left(\sum_{\substack{1 \leq B \leq T\\ B \text{ squarefree}\\ \Cl(\Z[\sqrt{-B}])[3] = 0}} \prod_{p \mid 2B} \left(1 + \frac1{2p}\right)^{-1} + o(T)\right)\\
        & = \frac{2\constFI T}{\sqrt{\log T}} \left(
        \sum_{1 \leq B \leq T} \prod_{p \mid 2B} \left(1 + \frac1{2p}\right)^{-1} 
        - \sum_{\substack{1 \leq B \leq T\\ B \text{ squarefree}\\ \Cl(\Z[\sqrt{-B}])[3] \neq 0}} \prod_{p \mid 2B} \left(1 + \frac1{2p}\right)^{-1} 
        + o(T)\right)\\
        &\geq \frac{2\constFI T}{\sqrt{\log T}} \left(
        \frac{4\constTaub T}{5} 
        - \sum_{\substack{1 \leq B \leq T\\ B \text{ squarefree}\\ \Cl(\Z[\sqrt{-B}])[3] \neq 0}} 1 
        + o(T)\right)\\
        &= \frac{2\constFI T^2}{\sqrt{\log T}} \left(
        \frac{4\constTaub}{5} 
        - \frac{6\widetilde{\delta_3}^{\imag}}{\pi^2}\right) + o\left(\frac{T^2}{\sqrt{\log T}}\right).\\
    \end{align*}
    Combining this with the estimates~\eqref{eq:constTaub} for $\constTaub$ and~\eqref{eq:delta_tilde_3_imag} for $\widetilde{\delta_3}^{\imag}$, we see that $\frac{4\constTaub}{5} - \frac{6\widetilde{\delta_3}^{\imag}}{\pi^2} > 0$. The same calculation shows 
    \[N_3^{\real}(T) \geq \frac{2\constFI T^2}{\sqrt{\log T}} \left(\frac{4\constTaub}{5} - \frac{6\widetilde{\delta_3}^{\real}}{\pi^2}\right) + o\left(\frac{T^2}{\sqrt{\log T}}\right).\]
Finally, using our lower bounds computed for $N_3(T)$ and our upper bounds for $N_3^{\loc}(T)$, we find 
    \[\liminf_{T \to \infty} \frac{N_3(T)}{N_3^{\loc}(T)} \geq \frac{4\constFI \left(\frac{4\constTaub}{5} - \frac{3(\widetilde{\delta_3}^\imag + \widetilde{\delta_3}^\real)}{\pi^2} \right)}{\frac{4\pi^4\constFI}{9(1-\alpha)^{1/2}}} = \frac{9(1-\alpha)^{1/2}}{\pi^4} \left( \frac{4\constTaub}{5} - \frac{3(\widetilde{\delta_3}^\imag + \widetilde{\delta_3}^\real)}{\pi^2} \right).\]
Taking $\alpha = \frac12 + 10^{-5}$ yields a value of at least 0.00988.
\end{proof}

\begin{proof}[Proof of Theorem~\ref{thm:pos_prop_HP_n>3}]

Given Assumption~\ref{assumption:CL}, we have $\widetilde{\delta_n}^\imag = \delta_n^{\imag} \leq \frac1{n-1}$ and $\widetilde{\delta_n}^\real = \delta_n^{\real} \leq \frac1{n(n-1)}$. Following the proof of Theorem~\ref{thm:pos_prop_HP_n=3} yields
\[\liminf_{T \to \infty} \frac{N_n(T)}{N_n^{\loc}(T)} \geq \frac{9(1-\alpha)^{1/2}}{\pi^4} \left( \frac{4\constTaub}{5} - \frac{3(n+1)}{\pi^2n(n-1)} \right) > 0.\]
In the limit as $n \to \infty$, the lower bound approaches 0.0275.
\end{proof}

\vfill
 
\bibliographystyle{alpha}
\bibliography{refs}

\end{document}

%% file: macros.tex
\usepackage{amsmath,amssymb,dsfont,mathrsfs}
\usepackage{amsthm}

\usepackage[dvipsnames]{xcolor}
\usepackage{pgf,tikz,pgfplots}
\pgfplotsset{compat=1.13}
\usetikzlibrary{arrows}
\usetikzlibrary{babel}
\usepackage{multicol}
\usepackage{bm}
\usepackage[all,cmtip]{xy}
\usepackage{tcolorbox}
\usepackage{tikz-cd}
\usepackage{adjustbox}
\usepackage{marvosym}
\usepackage{ stmaryrd }

\usepackage{mathtools}
\usepackage{colonequals}
\usepackage[letterpaper,top=1in,bottom=1in,left=1in,right=1in]{geometry}
\usepackage[normalem]{ulem}

\usepackage{enumitem}
\usepackage{soul}
\usepackage[hang,flushmargin]{footmisc}

\newenvironment{enumalg}
{\begin{enumerate}}
{\end{enumerate}}

\newenvironment{enumstep}
  {\begin{enumerate}[leftmargin=0pt, itemindent=1.8cm]
     \setcounter{enumi}{-1}%
     }
  {\end{enumerate}}

\usepackage{hyperref}
\hypersetup{colorlinks=true,urlcolor=magenta,citecolor=magenta,linkcolor=magenta}
\newcommand{\calC}{\mathcal{C}}
\newcommand{\C}{\mathcal{C}}

\newcommand{\Q}{{\mathbb{Q}}}
\newcommand{\A}{{\mathbb{A}}}

\newcommand{\R}{{\mathbb{R}}}

\newcommand{\G}{\mathbb{G}}

\newcommand{\Z}{{\mathbb{Z}}}

\newcommand{\F}{\mathbb{F}}

\newcommand{\f}{\mathfrak{f}}

\newcommand{\p}{\mathfrak{p}}
\newcommand{\pbar}{\bar{\mathfrak{p}}}

\newcommand{\frakp}{\mathfrak{p}}

\newcommand{\fraka}{\mathfrak{a}}

\newcommand{\ol}{\overline}
\newcommand{\loc}{\mathrm{loc}}

\newcommand{\I}{\mathfrak{i}}
\newcommand{\J}{\mathfrak{j}}
\renewcommand{\r}{\mathfrak{r}}
\newcommand{\z}{\mathfrak{z}}

\newcommand{\ubar}{\overline{u}}

\newcommand{\calX}{\mathcal{X}}
\newcommand{\X}{\mathcal{X}}
\newcommand{\calY}{\mathcal{Y}}
\newcommand{\Y}{\mathcal{Y}}
\newcommand{\calP}{\mathcal{P}}

\newcommand{\constFI}{\kappa_{1}}
\newcommand{\constTaub}{\kappa_{2}}
\newcommand{\real}{\mathrm{re}}
\newcommand{\imag}{\mathrm{im}}

\renewcommand{\div}{\ensuremath{\mathrm{div}}}

\DeclareMathOperator{\Nm}{Nm}

\DeclareMathOperator{\Cl}{Cl}

\DeclareMathOperator{\Disc}{Disc}

\DeclareMathOperator{\Spec}{Spec}

\DeclareMathOperator{\Aut}{Aut}
\DeclareMathOperator{\im}{im}
\newcommand{\val}{v}

\newcommand{\sizedbracket}[1]{\left[ #1 \right]}
\newcommand{\sizedparen}[1]{\left( #1 \right)}
\newcommand{\sizedcurly}[1]{\left\{ #1 \right\}}

\makeatletter
\newcommand{\YBCarg}[4][\@nil]{%
  \def\tmp{#1}  
   \ifx\tmp\@nnil
       \Y_{#2, #3} \left( \Z ; \gcd(z,f) = #4 \right)
    \else
       \Y_{#2, #3} \left( \Z ; \gcd(x,y) = #1,\ \gcd(z,f) = #4 \right)
    \fi
} \makeatother

\renewcommand{\O}{\ensuremath\mathcal{O}}
\newcommand{\orb}{\mathcal{O}}
\renewcommand{\P}{\ensuremath{\mathbb{P}}}

\newcommand{\annulus}{*}
\newcommand{\annulusTwo}{\tilde\annulus}
\newcommand{\invertedprimessplit}{\mathcal{S}_{\mathrm{split}}}

\usepackage{bm}

\usepackage{color}
\usepackage{phaistos}

\allowdisplaybreaks
\numberwithin{equation}{section}

\newtheorem{theorem}[equation]{Theorem}
\newtheorem{proposition}[equation]{Proposition}

\newtheorem{lemma}[equation]{Lemma}
\newtheorem{corollary}[equation]{Corollary}

\theoremstyle{definition}
\newtheorem{definition}[equation]{Definition}

\newtheorem{ex}[equation]{Example}

\newtheorem{notation}[equation]{Notation}
\newtheorem{observation}[equation]{Observation}
\newtheorem{assumption}[equation]{Assumption}

\newtheorem{remark}[equation]{Remark}

\theoremstyle{plain}
\newtheorem{thma}{Theorem} % regular numbered theorem environment
\newtheorem{thmb}{Theorem} % regular numbered theorem environment
\newtheorem{thmc}{Theorem} % regular numbered theorem environment
\newtheorem{thmd}{Theorem} % regular numbered theorem environment
\usepackage{xparse}
\usepackage{etoolbox}
\ExplSyntaxOn
\clist_new:N \l_my_colorcycle_clist
\clist_set:Nn \l_my_colorcycle_clist { c4 , teal }

\NewDocumentCommand{\colorizewords}{m}
 {
  \colorizewords_aux:n {#1}
 }

\cs_new:Nn \colorizewords_aux:n
 {
  \clist_set:Nn \l_tmpa_clist { #1 }
  \int_zero:N \l_tmpa_int
  \clist_map_inline:Nn \l_tmpa_clist
   {
    \int_incr:N \l_tmpa_int
    \int_set:Nn \l_tmpb_int { \int_mod:nn { \l_tmpa_int - 1 } { \clist_count:N \l_my_colorcycle_clist } + 1 }
    \textcolor{ \clist_item:Nn \l_my_colorcycle_clist { \l_tmpb_int } }{##1}~
   }
 }
\ExplSyntaxOff

\newcommand{\YYprime}{\sigma} %map from Y to Y prime
\newcommand{\twist}{d} %twist parameters.
\newcommand{\twistmap}{\pi_\twist'} %twists coming from \CYprime.
\newcommand{\invertedprimes}{\mathcal{S}}

\newcommand{\invertedprimesCsplit}{\mathcal{S}_{C,\mathrm{split}}}
\newcommand{\admissible}{A}